\documentclass[12pt]{amsart}
\usepackage{amsmath, amssymb, amsthm, latexsym}
\usepackage{amsthm}
\usepackage{fullpage}
\usepackage{amssymb}
\usepackage{latexsym}
\usepackage[center]{caption}
\usepackage{tikz}
\usepackage{mathtools}
\usepackage{enumerate}
\usepackage{hyperref}

\newcounter{braid}
\newcounter{strands}
\pgfkeyssetvalue{/tikz/braid height}{1cm}
\pgfkeyssetvalue{/tikz/braid width}{1cm}
\pgfkeyssetvalue{/tikz/braid start}{(0,0)}
\pgfkeyssetvalue{/tikz/braid colour}{black}
\pgfkeys{/tikz/strands/.code={\setcounter{strands}{#1}}}

\makeatletter
\def\cross{%
  \@ifnextchar^{\message{Got sup}\cross@sup}{\cross@sub}}

\def\cross@sup^#1_#2{\render@cross{#2}{#1}}

\def\cross@sub_#1{\@ifnextchar^{\cross@@sub{#1}}{\render@cross{#1}{1}}}

\def\cross@@sub#1^#2{\render@cross{#1}{#2}}

\def\render@cross#1#2{
  \def\strand{#1}
  \def\crossing{#2}
  \pgfmathsetmacro{\cross@y}{-\value{braid}*\braid@h}
  \pgfmathtruncatemacro{\nextstrand}{#1+1}
  \foreach \thread in {1,...,\value{strands}}
  {
    \pgfmathsetmacro{\strand@x}{\thread * \braid@w}
    \ifnum\thread=\strand
    \pgfmathsetmacro{\over@x}{\strand * \braid@w + .5*(1 - \crossing) * \braid@w}
    \pgfmathsetmacro{\under@x}{\strand * \braid@w + .5*(1 + \crossing) * \braid@w}
    \draw[braid] \pgfkeysvalueof{/tikz/braid start} +(\under@x pt,\cross@y pt) to[out=-90,in=90] +(\over@x pt,\cross@y pt -\braid@h);
    \draw[braid] \pgfkeysvalueof{/tikz/braid start} +(\over@x pt,\cross@y pt) to[out=-90,in=90] +(\under@x pt,\cross@y pt -\braid@h);
    \else
    \ifnum\thread=\nextstrand
    \else
     \draw[braid] \pgfkeysvalueof{/tikz/braid start} ++(\strand@x pt,\cross@y pt) -- ++(0,-\braid@h);
    \fi
   \fi
  }
  \stepcounter{braid}
}

\tikzset{braid/.style={double=\pgfkeysvalueof{/tikz/braid colour},double distance=1pt,line width=2pt,white}}

\newcommand{\braid}[2][]{%
  \begingroup
  \pgfkeys{/tikz/strands=2}
  \tikzset{#1}
  \pgfkeysgetvalue{/tikz/braid width}{\braid@w}
  \pgfkeysgetvalue{/tikz/braid height}{\braid@h}
  \setcounter{braid}{0}
  \let\sigma=\cross
  #2
  \endgroup
}
\makeatother

\input xypic
\newtheorem{theorem}{Theorem}%[section]
\newtheorem{proposition}[theorem]{Proposition}

\newtheorem{lemma}[theorem]{Lemma}

\newtheorem{corollary}[theorem]{Corollary}
\theoremstyle{definition}

\newtheorem{definition}[theorem]{Definition}
\newtheorem{construction}[theorem]{Construction}
\newtheorem{assumptions}[theorem]{Assumptions}
\newtheorem{remark}[theorem]{Remark}
\newtheorem{example}[theorem]{Example}

\def\Z{\mathbb{Z}}

\def\Pi{\mathbb{P}^{\infty}}

\def\Zpk{\mathbb{Z}/p^{k}}
\def\Zpk1{\mathbb{Z}/p^{k-1}}

\def\sl2{\widetilde{SL_{2}(\Z)}}

\DeclareMathOperator{\ord}{ord}

\DeclareMathOperator{\Ind}{Ind}
\DeclareMathOperator{\Nm}{Nm}

\DeclareMathOperator{\Gal}{Gal}
\DeclareMathOperator{\trace}{trace}
\DeclareMathOperator{\Frob}{Frob}
\DeclareMathOperator{\splt}{split}
\DeclareMathOperator{\nonsplit}{nonsplit}

\DeclareMathOperator{\End}{End}

\DeclareMathOperator{\Fil}{Fil}
\DeclareMathOperator{\Isog}{Isog}

\DeclareMathOperator{\Corr}{Corr}
\DeclareMathOperator{\Hom}{Hom}

\DeclareMathOperator{\add}{add}
\DeclareFontFamily{U}{wncy}{}
    \DeclareFontShape{U}{wncy}{m}{n}{<->wncyr10}{}
    \DeclareSymbolFont{mcy}{U}{wncy}{m}{n}
    \DeclareMathSymbol{\Sh}{\mathord}{mcy}{"58} 
\title{Generalized Heegner Cycles at Eisenstein Primes and the Katz $p$-adic $L$-function}
\author{Daniel Kriz}
\address{Department of Mathematics, Princeton University, Fine Hall, Washington Road, Princeton, NJ 08544-1000, USA}
\email{dkriz@princeton.edu}
\begin{document}
\begin{abstract} In this paper, we consider normalized newforms $f\in S_k(\Gamma_0(N),\varepsilon_f)$ whose non-constant term Fourier coefficients are congruent to those of an Eisenstein series modulo some prime ideal above a rational prime $p$. In this situation, we establish a congruence between the anticyclotomic $p$-adic $L$-function of Bertolini-Darmon-Prasanna and the Katz two-variable $p$-adic $L$-function. From this, we derive congruences between images under the $p$-adic Abel-Jacobi map of certain generalized Heegner cycles attached to $f$ and special values of the Katz $p$-adic $L$-function.  

In particular, our results apply to newforms associated with elliptic curves $E/\mathbb{Q}$ whose mod $p$ Galois representations $E[p]$ are reducible at a good prime $p$. As a consequence, we show the following: if $K$ is an imaginary quadratic field satisfying the Heegner hypothesis with respect to $E$ and in which $p$ splits, and if the bad primes of $E$ satisfy certain congruence conditions mod $p$ and $p$ does not divide certain Bernoulli numbers, then the Heegner point $P_{E}(K)$ is non-torsion,  in particular implying that $\text{rank}_{\mathbb{Z}}E(K) = 1$. From this, we show that when $E$ is semistable with reducible mod $3$ Galois representation, then a positive proportion of real quadratic twists of $E$ have rank 1 and a positive proportion of imaginary quadratic twists of $E$ have rank 0.
\end{abstract}

\maketitle
\tableofcontents

\section{Notation and Conventions}\label{notation} Throughout the paper, let us fix the following notational conventions. 

For $m,n\in \mathbb{Z}$, let $(m,n)$ denote the greatest common divisor of $m$ and $n$, and $lcm(m,n)$ denote the least common multiple. We let ``$\ell||N$'' denote ``$\ell$ strictly divides $N$''. We let $\Phi : \mathbb{Z} \rightarrow \mathbb{Z}$ denote the Euler totient function. Given an extension of number fields $L/K$ and an integral ideal $\mathfrak{a}$ of $\mathcal{O}_L$, let $\Nm_{L/K}\mathfrak{a}$ denote the relative ideal norm of $\mathfrak{a}$ and let $|\mathfrak{a}|$ denote the smallest positive rational integer in $\mathfrak{a}$. For ideals $\mathfrak{a}, \mathfrak{b}$, we will let $(\mathfrak{a},\mathfrak{b})$ denote the greatest common ideal divisor and $lcm(\mathfrak{a},\mathfrak{b})$ the least common ideal multiple. For a place $v$ of a number field, let $\Frob_v$ denote the \emph{arithmetic Frobenius} attached to $v$, i.e the Frobenius element which is sent to $v$ under the (inverse of the) Artin reciprocity map. 

Throughout, we will fix an algebraic closure $\overline{\mathbb{Q}}$ of $\mathbb{Q}$. All number fields in our discussion will be viewed as embedded in $\overline{\mathbb{Q}}$. For each rational prime $p$, we fix an algebraic closure $\overline{\mathbb{Q}}_p$ of $\mathbb{Q}_p$, and let $\mathbb{C}_p$ denote the topological closure of $\overline{\mathbb{Q}}_p$. We fix an embedding $i_{\infty}: \overline{\mathbb{Q}} \hookrightarrow \mathbb{C}$ as well as an embedding $i_p: \overline{\mathbb{Q}} \hookrightarrow \mathbb{C}_p$ for each $p$. We also fix field identifications $i : \mathbb{C}\xrightarrow{\sim}\mathbb{C}_p$ for each $p$. (We will use the same symbol $i$ for all $p$ as the underlying $p$ will be clear from context.)

Let $K$ denote a general number field. Let $\mathbb{A}_K$ and $\mathbb{A}_K^{\times}$ denote the ad\`{e}les and id\`{e}les over $K$, respectively, and let $\mathbb{A}_{K,f}$ and $\mathbb{A}_{K,f}^{\times}$ denote the finite ad\`{e}les and id\`{e}les, respectively. For a Hecke character $\chi : \mathbb{A}_K^{\times} \rightarrow \mathbb{C}^{\times}$, let $\mathfrak{f}(\chi) \subset \mathcal{O}_K$ denote the conductor of $\chi$. As usual, if $\mathfrak{a}\subset\mathcal{O}_K$ is an integral ideal with $(\mathfrak{a},\mathfrak{f}(\chi)) \neq 1$, then we put $\chi(\mathfrak{a}) = 0$.

Given a Dirichlet (i.e. finite order) character $\psi$ over a number field $K$ of conductor $\mathfrak{f} \subset \mathcal{O}_K$, we will identify $\psi$ with its associated finite order Hecke character on id\`{e}les, taking the following convention: for $x \in (\mathcal{O}/\mathfrak{f})^{\times}$,
$$\psi(x \mod \mathfrak{f}) = \prod_{v\nmid \mathfrak{f}} \psi_v(x)$$
where $v$ runs over all places of $E$ which do not divide $\mathfrak{f}$. When $K = \mathbb{Q}$ and $\chi : \mathbb{A}_{\mathbb{Q}}^{\times} \rightarrow \mathbb{C}^{\times}$ is a Dirichlet character, we define the Gauss sum of $\chi$ by
$$\mathfrak{g}(\chi) := \sum_{a\in (\mathbb{Z}/\mathfrak{f}(\chi)\mathbb{Z})^{\times}}\chi(a)e^{\frac{2\pi i a}{\mathfrak{f}(\chi)}}$$ 
and for a finite prime $\ell$, we define the Gauss sum of the local character $\chi_{\ell}: \mathbb{Q}_{\ell}^{\times} \rightarrow \mathbb{C}^{\times}$ similarly:
$$\mathfrak{g}(\chi_{\ell}) := \sum_{a \in (\mathbb{Z}_{\ell}/\mathfrak{f}(\chi)\mathbb{Z}_{\ell})^{\times}}\chi_{\ell}(a)e^{\frac{2\pi i a}{|\mathfrak{f}(\chi)|}},$$
so that $\prod_{\ell|\mathfrak{f}(\chi)}\mathfrak{g}(\chi_{\ell}) = \mathfrak{g}(\chi^{-1})$. 

Let $\mathbb{N}_{\mathbb{Q}} : \mathbb{A}_{\mathbb{Q}} \rightarrow \mathbb{C}$ denote the ad\`{e}lic norm, normalized so that $\mathbb{N}_{\mathbb{Q},\infty} = |\cdot|_{\infty}$ (the usual archimedean absolute value). Then, for any number field $K$, let $\Nm_{K/\mathbb{Q}} : \mathbb{A}_K\rightarrow\mathbb{A}_{\mathbb{Q}}$ denote the norm homomorphism (which induces the ideal norm recalled above), and put $\mathbb{N}_K := \mathbb{N}_{\mathbb{Q}}\circ \Nm_{K/\mathbb{Q}}$. Note that when $K$ is imaginary quadratic, $\mathbb{N}_K : \mathbb{A}_K^{\times} \rightarrow \mathbb{C}^{\times}$ is an algebraic Hecke character of infinity type $(-1,-1)$. For any Hecke character $\chi: \mathbb{A}_K^{\times} \rightarrow \mathbb{C}^{\times}$, let 
$$\chi_j := \chi\mathbb{N}_K^{-j}.$$ 

Using the fixed isomorphism $i: \mathbb{C}\xrightarrow{\sim}\mathbb{C}_p$ and the Artin isomorphism, we can view $\mathbb{N}_K^{-1}$ as a character $\mathbb{N}_K^{-1}: \Gal(\overline{K}/K) \rightarrow \mathbb{Z}_p^{\times}$. Then any $x\in\mathbb{Z}_p^{\times}$ can be uniquely written as $\omega(x)\cdot \langle x\rangle$ where $\omega(x)\in\mu_{2(p-1)}$ and $\langle x\rangle \in 1+2p\mathbb{Z}_p$. We then define the \emph{Teichm\"{u}ller character} $\omega: \Gal(\overline{K}/K) \rightarrow \mu_{2(p-1)}$ by $\omega_K(a) := \omega(\mathbb{N}_K^{-1}(a))$. For simplicity, we will let $\omega_{\mathbb{Q}} = \omega$ and $\langle \mathbb{N}_{\mathbb{Q}}\rangle = \langle \mathbb{N} \rangle$. Given an extension of number fields $K/L$ and a Hecke character $\phi$ over $L$, we will let $\phi_{/K} := \phi \circ \Nm_{K/L}$; note that viewing $\phi$ as a character $\Gal(\overline{L}/L) \rightarrow \mathbb{C}^{\times}$, this corresponds to restriction to the subgroup $\Gal(\overline{K}/K)$.

For a quadratic field $K/\mathbb{Q}$ with $K = \mathbb{Q}(\sqrt{D})$ where $D$ is a squarefree integer, we recall the fundamental discriminant of $K$ is given by
$$D_K = \begin{cases}
D & \text{if}\; D \equiv 1 \mod 4,\\
4D & \text{if}\; D \equiv 2, 3 \mod 4.\\
\end{cases}$$
For quadratic $K$, let $\varepsilon_K$ be the associated Dirichlet character of conductor $D_K$; conversely, given a quadratic Dirichlet character $\chi$, let $K_{\chi}$ be the associated imaginary quadratic field. For two quadratic fields $L = \mathbb{Q}(\sqrt{D})$ and $K = \mathbb{Q}(\sqrt{D'})$, we will let $L\cdot K$ denote the quadratic field $\mathbb{Q}(\sqrt{DD'})$. Throughout the paper, unless otherwise noted, all Dirichlet characters over $\mathbb{Q}$ are taken to be \emph{primitive}, and so are uniquely identified with finite order Hecke characters over $\mathbb{Q}$ under the arithmetic normalization described above. In particular, given Dirichlet characters $\psi_1, \psi_2$ over $\mathbb{Q}$, we define $\psi_1\psi_2$ to be the \emph{primitive} Dirichlet character equal to $\psi_1(a)\psi_2(a)$ for $a \in (\mathbb{Z}/lcm(\mathfrak{f}(\psi_1),\mathfrak{f}(\psi_2)))^{\times}$ (and indeed this equality holds for the associated Hecke characters). In particular, for quadratic $L$ and $K$ as above, we have $\varepsilon_{L\cdot K} = \varepsilon_L\varepsilon_K$.

Given a normalized newform $f \in S_k(\Gamma_1(N))$, let $E_f$ denote the finite extension of $\mathbb{Q}$ generated by its Fourier coefficients. Given a Hecke character $\chi$, let $E_{\chi}$ denote the finite extension of $\mathbb{Q}$ generated by its values. Put $E_{f,\chi} = E_fE_{\chi}$.

For $s \in \mathbb{C}$, let $\Re(s)$ and $\Im(s)$ denote the real and imaginary parts of $s$, respectively. Throughout, we let $\mathcal{H}^+ := \{s\in\mathbb{C} : \Im(s) > 0\}$.

\section{Introduction}

\label{outline} The study of Heegner points has provided some of the greatest insights into the Birch and Swinnerton-Dyer Conjecture. In particular, given an elliptic curve $E/\mathbb{Q}$ and an imaginary quadratic field $K$ satisfying a suitable \emph{Heegner hypothesis} with respect to $E$, the non-triviality of the Heegner point $P_E(K) \in E(K)\otimes \mathbb{Q}$ introduced in Section \ref{twoformulas} implies, via the descent argument of Kolyvagin \cite{Kolyvagin}, that $\text{rank}_{\mathbb{Z}}E(K) = 1$. Our main result establishes, for elliptic curves $E$ with $E[p]$ a reducible Galois representation at a good prime $p$, a congruence mod $p$ between the formal logarithm of the Heegner point and a special value of the Katz $p$-adic $L$-function with certain Euler factors removed. Using Gross's factorization of the Katz $p$-adic $L$-function on the cyclotomic line (\cite{Gross}), we can then find an explicit congruence between the formal logarithm of the Heegner point and a quantity involving certain Bernoulli numbers and (inverses of) Euler factors at primes of bad reduction. Thus we derive an explicit criterion (involving Bernoulli numbers) for the non-triviality of $P_{E}(K) \in E(K)\otimes\mathbb{Q}$ in certain families of $E$.

Our results also have higher-dimensional applications, namely in computing with algebraic cycle classes in Chow groups of motives attached to newforms. Suppose $N$ is a positive integer and $K$ is an imaginary quadratic field satisfying the Heegner hypothesis with respect to $N$, $f = \sum_{n\ge 1} a_nq^n$ is a weight $k \ge 2$ level $N$ normalized newform whose Fourier coefficients are congruent to those of an Eisenstein series outside the constant term modulo some prime ideal above $p$, and $\chi$ is an algebraic Hecke character over $K$ central critical with respect to $f$. We consider the Rankin-Selberg motive $M_{f,\chi^{-1}} := M_{f_{/K}}\otimes M_{\chi^{-1}}$, defined over $K$, associated with $(f,\chi)$, where $M_f$ is the motive associated with $f$ (see \cite{Scholl}), $M_{f_{/K}}$ is its base change to $K$, and $M_{\chi}$ the motive associated with $\chi$ (see \cite{Deninger}). For convenience of notation, we will formulate our discussion with respect to $M_{(f,\chi^{-1})_{/F}} = M_{f_{/F}}\otimes M_{(\chi^{-1})_{/F}}$, where $F/K$ is some large enough abelian extension in ``situtation $(S)$'' as described below, and $f_{/F}$ and $\chi_{/F}$ are the corresponding base extensions to $F$. In fact, one may recover $M_{\chi^{-1}}$ as a direct factor of the Grothendieck restriction of $M_{(\chi^{-1})_{/F}}$, see \cite{Deninger} or \cite{Geisser}. Indeed, we have a factorization $M_{(f,\chi^{-1})_{/F}} = \bigoplus_{\chi_0 : \Gal(F/K) \rightarrow \mathbb{C}^{\times}}M_{f,\chi^{-1}\chi_0}$. 

Under certain conditions guaranteeing the vanishing of the central value of $L(\pi_f\times\pi_{\chi^{-1}},s) = L(M_{f,\chi^{-1}},s)$ (and thus also that of $L((\pi_f\times \pi_{\chi^{-1}})_{/F},s) = L(M_{(f,\chi^{-1})_{/F}},s)$) at its central point, we seek to construct algebraic cycle classes in the Chow group associated with $M_{(f,\chi^{-1})_{/F}}$, thus providing evidence for the Beilinson-Bloch conjecture for Chow motives associated with $(f,\chi)_{/F}$. Natural candidates for representatives of such classes are \emph{generalized Heegner cycles} (in the sense of \cite{BDP2}), generated by algebraic cycles on a suitable \emph{generalized Kuga-Sato variety} which arise from graphs of isogenies between products of $CM$ curves. Under the above hypotheses, our main theorem provides a criterion for a generalized Heegner cycle class associated with $(f,\chi)$ with $\chi$ in a certain infinity type range to be non-trivial, namely, that an associated special value of the Katz $p$-adic $L$-function (with certain Euler factors removed) is not (too) divisible by $p$. One may thus use special value formulas of $p$-adic $L$-functions in order to get explicit $p$ non-divisibility criteria for the non-triviality of Heegner cycle classes. In our main corollary, we explicitly address the case when Gross's factorization theorem can be applied, which corresponds precisely with the case when the relevant generalized Heegner cycle arises from a \emph{classical} Heegner cycle (cf. Remark \ref{classicalHeegnerremark}).

\subsection{Chow motives}\label{Chowmotives} To make our discussion more precise, let us briefly recall some definitions pertaining to Chow motives. Let $X$ be a nonsingular variety over a number field $F$. An \emph{algebraic cycle} of $X$ is a formal sum of subvarieties of $X$. We let $\operatorname{CH}^j(X)$ denote the group generated by algebraic cycles of $X$ of codimension $j$ modulo rational equivalence. Call it the \emph{$j$th Chow group of $X$}, and $\operatorname{CH}(X) := \bigoplus_{0\le j \le \dim X}\operatorname{CH}^j(X)$ the \emph{Chow group of $X$}. For a varieties $X,Y$ over $F$, we call $\Corr^d(X,Y) := \operatorname{CH}^{\dim Y+d}(X\times_F Y)$ the \emph{group of (degree $d$) correspondences from $X$ to $Y$}. The group $\Corr(X,X) := \bigoplus_{0\le d\le \dim X}\Corr^d(X,X)$ has a ring structure defined by \emph{composition of correspondences}: given $g\in \Corr^{d_1}(X,Y)$ and $h\in \Corr^{d_2}(Y,Z)$, we define
$$h\circ g := \pi_{13,*}(\pi_{12}^*(g).\pi_{23}^*(h)) \in \Corr^{d_1+d_2}(X,Z)$$
where $\pi_{12}, \pi _{13}, \pi_{23} : X\times_F Y \times_F Z \rightarrow X\times_F Y, X\times_F Z, Y \times_F Z$ are the canonical projections, and ``$.$'' denotes the Chow intersection pairing. We call $\Corr(X,X)$ the \emph{ring of correspondences on $X$}. 
For any number field $E$, let $\Corr^d(X,Y)_E := \Corr^d(X,Y)\otimes_{\mathbb{Z}} E$. A \emph{Chow motive over $F$ with coefficients in $E$} is a triple $(X,e,m)$ consisting of a nonsingular variety $X$ over $F$, an idempotent $e\in \Corr^0(X,X)_E$, and an integer $m$. Define the \emph{category of Chow motives} $\mathcal{M}_{F,E}$ whose objects consist of Chow motives over $F$ with coefficients in $E$, and with morphisms given by
$$\Hom_{\mathcal{M}_{F,E}}((X,e,m),(Y,f,n)) := f\circ \Corr^{n - m}(X,Y)_{\mathbb{Q}}\circ e.$$
Define the \emph{category of Grothendieck motives} $\mathcal{M}_{F,E}^{hom}$ with objects being cycles of $X$ modulo (motivic) homological equivalence, and with morphisms defined in the same way as above. Note that since rational equivalence is stronger than homological equivalence, there is a natural functor
$$\mathcal{M}_{F,E} \rightarrow \mathcal{M}_{F,E}^{hom}.$$
Let $\operatorname{CH}^j(X)_0$ denote the subgroup of $\operatorname{CH}^j(X)$ consisting of null-homologous cycles. 

\subsection{The Beilinson-Bloch conjecture for Rankin-Selberg motives}\label{Chowmotives'}Now let $K$ be an imaginary quadratic field, let $H$ denote the Hilbert class field over $K$. Fix $f \in S_k(\Gamma_0(N),\varepsilon_f)$ where $k \ge 2$, let $r := k-2$. Let $\chi$ be a Hecke character over $K$ of infinity type $(r-j,j)$ which is central critical with respect to $f$, and where $0 \le j \le r$. Suppose also that every prime $\ell|N$ splits in $K$, i.e. $K$ satisfies the Heegner hypothesis with respect to $N$. Then we can choose an ideal $\mathfrak{N}$ of $\mathcal{O}_K$ with $\mathcal{O}_K/\mathfrak{N} = \mathbb{Z}/N$. Assume also that $\mathfrak{f}(\chi)|\mathfrak{N}$. 

The ambient motive in our setup will be $M := (X_r,\epsilon_X,0) \in \mathcal{M}_{F,E_{f,\chi}}$, where $F$, to be fixed later, is some large enough abelian extension of $K$ containing $H$, and $\epsilon_X$ is some projector in the ring of correspondences on $X$ which has induced actions on the various cohomological realizations of $M$ via the corresponding cycle class maps. (The projector $\epsilon_X$ essentially picks out the part of the cohomology of $X_r$ which comes from the Galois representations attached to pairs $(f,\chi)_{/F}$.) Here the underlying ($2r+1$-dimensional) variety is the \emph{generalized Kuga-Sato variety} $X_r := W_r \times A^r$, defined over $H$, where $A$ is a fixed $CM$ elliptic curve of $\mathcal{O}_K$-type (i.e. with $\End_H(A) = \mathcal{O}_K$) and $W_r := \mathcal{E}^r$ (the \emph{(classical) Kuga-Sato variety}), is the (canonical desingularization of the) $r$-fold fiber product of copies of the universal elliptic curve $\mathcal{E}$ with $\Gamma_1(N)$-level structure over the (compactified) modular curve $X_1(N) := \overline{Y_1(N)}$ (and thus is defined over $\mathbb{Q}$). This is well-defined in the case where the cusps of $Y_1(N)$ are \emph{regular} in the sense of \cite{DiamondShurman} Section 3.2, which holds in particular when $N > 4$. The fibers of $X_r \rightarrow X_1(N)$, outside of the cusps of $X_1(N)$, are of the form $E^r \times A^r$ where $E$ varies over elliptic curves.

The motive $M_f$ is given by the triple $(W_r,\epsilon_f,0)$ for some projector in the ring of correspondences on $W_r$ which picks out the $f$-isotypic component of the cohomology of $W_r$ under the Hecke action; in particular, this Chow motive is defined over $\mathbb{Q}$, with coefficients in $E_f$. For an extension $F/\mathbb{Q}$, we let $M_{f_{/F}}$ denote the base change to $F$. The definition of the motive $M_{(\chi^{-1})_{/F}}$ is slightly more subtle. Let $F/K$ be an abelian extension such that $\chi_A^{r-j}\overline{\chi}_A^j = \chi_{/F}$, where $\chi_A: \mathbb{A}_F^{\times} \rightarrow K^{\times}$ is the Hecke character  associated with the $CM$ elliptic curve $A$ having infinity type $(1,\ldots,1,0,\ldots,0)$ (with the first $[F:K]$ places corresponding to the embeddings $F\hookrightarrow\overline{\mathbb{Q}}$ preserving our fixed embedding $K \hookrightarrow \overline{\mathbb{Q}}$, and the next $[F:K]$ places to their complex conjugates). We then say $F$ is in ``situation $(S)$". (Such an $F$ always exists, see \cite[Proposition 1.3.1]{Deninger}.) We have a motive $M_{(\chi^{-1})_{/F}} := (A^r,\epsilon_{\chi_{/F}},0) \in \mathcal{M}_{F,E_{\chi}}$ for some suitable projector $\epsilon_{\chi_{/F}}$ picking out the $\chi_{/F} = \chi_A^{r-j}\overline{\chi}_A^j$-isotypic component of the cohomology of $A^r$ under the Galois action. %One can show that $M_{\chi_A^{-1}}$ descends to a motive $M_{\chi^{-1}} \in \mathcal{M}_{K,E_{\chi}}$, and conversely every such motive $M_{\chi^{-1}}$ attached to a Hecke character $\chi$ over $K$ arises in this way. (The $M_{\chi_A^{-1}}$ from which $M_{\chi^{-1}}$ arises is unique, at least up to absolute Hodge realization and extension of coefficients, see \cite[Proposition 1.3.2]{Deninger} and the ensuing Remark therein.) Thus, for our purposes, it will suffice to state Beilinson-Bloch in terms of $M_{\chi_A^{-1}}$. As an abuse of notation, we will henceforth put $\epsilon_{\chi} = \epsilon_{\chi_A}$ and $M_{\chi^{-1}} = M_{\chi_A^{-1}}$.

Let $H_{\mathfrak{N}}$ be the field over which the individual points of $A[\mathfrak{N}]$ are defined. We will henceforth take and fix an abelian extension $F/K$ large enough so that it contains $H_{\mathfrak{N}}$ and is in situation $(S)$. (Note this is possible since $H_{\mathfrak{N}}$ is an abelian extension of $K$.) The Rankin-Selberg motive $M_{(f,\chi^{-1})_{/F}} := M_{f_{/F}}\otimes M_{(\chi^{-1})_{/F}} = (W_r,\epsilon_{f_{/F}},0)\otimes(A^r,\epsilon_{\chi_{/F}},0) = (X_r,\epsilon_{(f,\chi)_{/F}},0)$ is a submotive of $M$ (in fact, it is the $(f,\chi)_{/F}$-isotypic component $\epsilon_{(f,\chi)_{/F}}M$), and the Beilinson-Bloch conjecture predicts that
\begin{align*}\dim_{E_{f,\chi}} \epsilon_{(f,\chi)_{/F}}\operatorname{CH}^{r+1}(X_r)_{0,E_{f,\chi}}(F) 
&= \ord_{s = r+1}L(H_{\text{\'{e}t}}^{2r+1}(M_{(f,\chi^{-1})_{/F}}),s)\\
&=\ord_{s = r+1}L(\epsilon_{(f,\chi)_{/F}}H_{\text{\'{e}t}}^{2r+1}(M),s) \\
&= \ord_{s = r+1}L((V_{f_{/K}}\otimes \chi^{-1})_{/F}(r+1),s)\\
&=\ord_{s = 0}L((V_{f_{/K}}\otimes \chi^{-1})_{/F},s) \\
&= \ord_{s=0}L(V_{f_{/K}}\otimes \chi^{-1} \otimes \Ind_F^K 1,s)\\
&= \ord_{s=0}\prod_{\chi_0 : \Gal(F/K) \rightarrow \mathbb{C}^{\times}}L(V_{f_{/K}}\otimes (\chi^{-1}\chi_0),s) \\
&= \ord_{s = \frac{1}{2}}\prod_{\chi_0 : \Gal(F/K) \rightarrow \mathbb{C}^{\times}}L(\pi_f\times\pi_{\chi^{-1}\chi_0},s).
\end{align*}
Here $\epsilon_{(f,\chi)_{/F}}$ is the projector corresponding to the $(f,\chi)_{/F}$-isotypic component of the cohomology of $X_r$ under the Hecke and Galois actions, $V_f$ is the $2$-dimensional (unique semisimple) $\Gal(\overline{\mathbb{Q}}/\mathbb{Q})$-representation associated with $f$, $V_{f_{/K}}$ is its restriction to $\Gal(\overline{K}/K)$, and $\pi_f, \pi_{\chi^{-1}}$ are the automorphic representations (over $K$) associated with $f$ and $\theta_{\chi^{-1}}$ under the unitary normalizations. 

Under the Heegner hypothesis, we have that $L(\pi_f\times\pi_{\chi^{-1}},\frac{1}{2}) = 0$, and so the Beilinson-Bloch conjecture predicts that $\dim_{E_{f,\chi}} \epsilon_{(f,\chi)_{/F}}\operatorname{CH}^{r+1}(X_r)_{0,E_{f,\chi}}(F) \ge 1$. Thus we should be able to find a null-homologous cycle class with non-vanishing $\epsilon_{(f,\chi)_{/F}}$-isotypic component in the above Chow group. To this end, we can use the $p$-adic Waldspurger formula of Bertolini-Darmon-Prasanna (\cite{BDP2}) to consider images under the $p$-adic Abel-Jacobi map of \emph{generalized Heegner cycles}. Generalized Heegner cycles are, essentially, generated by graphs $\Gamma_{\varphi}$ of isogenies $\varphi : A \rightarrow A'$ between $CM$ elliptic curves with $\Gamma_1(N)$-level structure (i.e., such that $\ker(\varphi)\cap A[\mathfrak{N}] = 0$). We can then view the graph $\Gamma_{\varphi}$ inside $X_r$:
$$\Gamma_{\varphi} \subset A^r\times (A')^r \subset A^r \times W_r = X_r.$$ 

We then define the associated \emph{generalized Heegner cycle} associated with $(\varphi,A)$ as 
$$\Delta_{\varphi} := \epsilon_X\Gamma_{\varphi} \in \operatorname{CH}^{r+1}(X_r)_{0,E_{f,\chi}}(F).$$
(See Section \ref{isogenies'} for details and a precise definition.) It is also a fact that $\epsilon_XH^{2r+2}(X_r,\mathbb{Q}) = 0$, and so $\Delta_{\varphi}$ is indeed null-homologous. Generalized Heegner cycles are those cycles in $\operatorname{CH}^{r+1}(X_r)_{0,E_{f,\chi}}(F)$ generated by formal $E_{f,\chi}$-linear combinations of $\Delta_{\varphi}$'s for varying $\varphi$. The space $\operatorname{CH}(X_r)_{0,E_{f,\chi}}(F)$ has an isotypic decomposition with respect to the action of the Hecke algebra (which is indexed by cuspidal eigenforms) and that of $\Gal(\overline{F}/F)$ (which is indexed by Hecke characters).

Let $F_p$ denote the $p$-adic completion of $F$ determined by our fixed embedding $i_p : \overline{\mathbb{Q}} \hookrightarrow \mathbb{C}_p$. We can now apply the $p$-adic Abel-Jacobi map over $F_p$ to $\Delta_{\varphi}$, which is a map
\begin{align*}AJ_{F_p}: \operatorname{CH}^{r+1}(X_r)_{0,\mathbb{Q}}(F_p) &\rightarrow (\Fil^{r+1}\epsilon_X H_{dR}^{2r+1}(X_r/F_p))^{\vee}\\
&\cong (S_k(\Gamma_1(N),F_p)\otimes Sym^rH_{dR}^1(A/F_p))^{\vee}
\end{align*}
where the superscript $\vee$ denotes the $F_p$-linear dual, and $\Fil^j$ denotes the $j^{th}$ step of the Hodge filtration. A basis of $\epsilon_XH_{dR}^{2r+1}(X_r/F)$ is given by elements of the form $\omega_f \wedge \omega_A^j\eta_A^{r-j}$ for $0 \le j \le r$, where $\omega_f \in \Fil^{r+1}\epsilon_XH_{dR}^{r+1}(W_r/F) \cong S_k(\Gamma_1(N),F)$ is associated with $f \in S_k(\Gamma_1(N),F)$, and the $\omega_A^j\eta_A^{r-j}$ form a basis of $\Fil^{r+1}\epsilon_X H_{dR}^r(A^r/F)\cong Sym^rH_{dR}^1(A/F)$. (Here $\omega_A\in \Omega^1(A/F) = H_{dR}^{1,0}(A/F)$ is a nowhere vanishing differential on $A$, $\eta_A \in H_{dR}^{0,1}(A/F)$ such that $\langle \omega_A,\eta_A \rangle = 1$ under the cup product pairing on de Rham cohomology.) 

Bertolini-Darmon-Prasanna's $p$-adic Waldspurger formula relates a special value of an anticyclotomic $p$-adic $L$-function $\mathcal{L}_p(f,\chi)$ to the Abel-Jacobi image of a certain generalized Heegner cycle $\Delta$, evaluated at the basis element $\omega_f\wedge\omega_A^j\eta_A^{r-j}$. The dual basis element of this latter element is in the $(f,\chi)_{/F}$-isotypic component of $(\Fil^{r+1}\epsilon_XH_{dR}^{2r+1}(X_r/F))^{\vee}$, and the idempotent $\epsilon_{(f,\chi)_{/F}}$ induces the projection onto this $(f,\chi)_{/F}$-isotypic component. By functoriality of projectors, the non-vanishing of $AJ_{F_p}(\Delta)$ at $\omega_f\wedge\omega_A^j\eta_A^{r-j}$ shows the non-triviality of $\epsilon_{(f,\chi)_{/F}}\Delta$. Hence showing the non-vanishing of a special value of Bertolini-Darmon-Prasanna's $p$-adic $L$-function verifies one consequence of the Beilinson-Bloch conjecture for the motive $M_{(f,\chi^{-1})_{/F}} = (X_r, \epsilon_{(f,\chi)_{/F}},0)$. 

More precisely, non-vanishing of the anticyclotomic $p$-adic $L$ function $\mathcal{L}_p(f,\chi)$ and functoriality of the action of correspondences on $X_r$ imply
\begin{align*}0&\neq AJ_{F_p}(\Delta)(\omega_f\wedge \omega_A^j\eta_A^{r-j}) \\
&= AJ_{F_p}(\Delta)(\epsilon_{(f,\chi)_{/F}}(\omega_f\wedge \omega_A^j\eta_A^{r-j})) \\
&= AJ_{F_p}(\epsilon_{(f,\chi)_{/F}}\Delta)(\omega_f\wedge \omega_A^j\eta_A^{r-j})\\
\end{align*}
and thus that $0 \neq \epsilon_{(f,\chi)_{/F}}\Delta \in \epsilon_{(f,\chi)_{/F}}\operatorname{CH}^{r+1}(X_r)_{0,E_{f,\chi}}(F)$. 

The classical situation $r = 0$ (i.e., $k = 2$) is perhaps instructive. In this case, the underlying generalized Kuga-Sato variety is simply $X_0 = X_1(N)$ Moreover, $\chi : \Gal(F/K) \rightarrow \mathbb{C}^{\times}$ is of finite order, so that $\chi_{/F} = 1$. Then $\operatorname{CH}^1(X_1(N))_{0,E_{f,\chi}}(F) = J_1(N)_{E_{f,\chi}}(F)$. The $p$-adic Abel-Jacobi map
$$AJ_{F_p} : J_1(N)_{E_{f,\chi}}(F_p) \rightarrow S_k(\Gamma_1(N),F_p)^{\vee}_{E_{f,\chi}}$$
is given by $P \mapsto (f \mapsto \log_{\omega_f}P)$, where $\log_{\omega_f}$ is the formal logarithm at $p$ associated with the non-vanishing differential $\omega_f$. Note that $\text{Gal}(\overline{\mathbb{Q}}/\mathbb{Q})$ acts on modular forms (as one may see by letting $\text{Gal}(\overline{\mathbb{Q}}/\mathbb{Q})$ act on the coefficients of $q$-expansions), and that this action also preserves eigenspaces of Hecke operators. For a modular form $f$, let $[f]$ denote its $\text{Gal}(\overline{\mathbb{Q}}/\mathbb{Q})$-orbit, and let $S_k^{new}(\Gamma_1(N))$ denote the space of $\Gamma_1(N)$-cuspidal newforms of weight $k$. Given a normalized newform $f$, Eichler-Shimura theory attaches to $f$ an abelian variety $A_f$ which arises as a quotient of $J_1(N)$. Let $A_{f,E} := A_f\otimes_{E_f} E$ for any ring $E$ containing $E_f$. We have an isotypic decomposition of the component of $J_1(N)$ under the Hecke algebra action in the isogeny category:
$$J_1(N)^{new} \sim \bigoplus_{\{[f], f \in S_k^{new}(\Gamma_1(N))\}} A_f$$
which implies
$$J_1(N)_{E_{f,\chi}}^{new} = \bigoplus_{\{[f], f \in S_k^{new}(\Gamma_1(N))\}} A_{f,E_{f,\chi}}.$$

We thus have natural surjections (called modular parametrizations) $\Phi_f : J_1(N) \twoheadrightarrow A_f$, and hence maps $\Phi_f: J_1(N)_{E_{f,\chi}}(F) \twoheadrightarrow A_{f,E_{f,\chi}}(F)$. Now suppose $\chi_0: \Gal(F/K) \rightarrow \mathbb{C}^{\times}$. Let $\Phi_{\chi_0}: A_{f,E_{f,\chi}}(F) \twoheadrightarrow A_{f,E_{f,\chi}}(F)^{\chi_0}$ denote the projection onto the $\chi_0$-isotypic component under the action of $\Gal(F/K)$, and put $\Phi_{f,\chi_0} = \Phi_{\chi_0}\circ\Phi_f : J_1(N)_{E_{f,\chi}}(F) \twoheadrightarrow A_{f,E_{f,\chi}}(F)^{\chi_0}$. Let $\omega_{A_{f}}$ be the unique invariant differential on the abelian variety $A_{f}/F$ satisfying $\Phi_{f}^*\omega_{A_{f}} = \omega_f$. Generalized Heegner cycles are simply classical Heegner points, i.e. degree 0 divisors $P(\chi_0) \in J_1(N)_{E_{f,\chi}}(F)$ arising from $\chi_0$-twisted $\Gal(F/K)$-traces of CM points on $X_1(N)(F)$. Then $P_f(\chi_0):=\Phi_{f,\chi_0}(P(\chi_0)) \in A_{f,E_{f,\chi}}(F)^{\chi_0}$. Note that when $\chi_0 = 1$, $P_f(1) \in A_{f,E_{f,\chi}}(K)$. The non-vanishing of the $p$-adic Abel-Jacobi map implies
\begin{align*}0&\neq\log_{\omega_f}P(\chi_0) \\
&= \log_{\Phi_{f,\chi_0}^*\omega_{A_{f}}}P(\chi_0) \\
&= \log_{\omega_{A_f}}\Phi_{f,\chi_0}(P(\chi_0)) \\
&= \log_{\omega_{A_f}}P_f(\chi_0),
\end{align*}
which implies $P_f(\chi_0) \in A_{f,E_{f,\chi}}(F)^{\chi_0}$ is non-trivial. Suppose $\chi_0 = 1$. If $E_f = \mathbb{Q}$, then $A_f$ is an elliptic curve, and $P_f(1)\in A_{f,\mathbb{Q}}(K)$ is a non-trivial. By Kolyvagin's theorem (\cite{Kolyvagin}), we have $\text{rank}_{\mathbb{Z}}A_f(K) = 1$.

\subsection{Main Results}\label{mainresults}

In the case where $f\in S_k(\Gamma_0(N),\varepsilon_f)$ is a normalized newform with \emph{partial Eisenstein descent} (see Definition \ref{Eisensteindescent}), i.e. whose Fourier coefficients are congruent to those of an Eisenstein series outside the constant term, we will establish a congruence between Bertolini-Darmon-Prasanna's anticyclotomic $p$-adic $L$-function and the Katz two-variable $p$-adic $L$-function holding wherever the former $1$-variable $p$-adic $L$-function is defined. A more precise statement is given in our Main Theorem below. Fix the following assumptions.

\begin{assumptions}\label{assumptions}Let
\begin{enumerate} 
\item $f\in S_k(\Gamma_0(N),\varepsilon_f)$ be a normalized newform of weight $k \ge 2$ and level $N>4$,
\item $p\nmid N$ be a rational prime,
\item $K/\mathbb{Q}$ be an imaginary quadratic extension with odd fundamental discriminant $D_K<-4$ and $p$ split in $K$,
\item{(Heegner hypothesis)} for any prime $\ell | N$, $\ell$ is split in $K/\mathbb{Q}$.
\end{enumerate}
\end{assumptions}

Note assumption (4) guarantees the existence of an integral ideal $\mathfrak{N}|N$ of $\mathcal{O}_K$ with 
$$\mathcal{O}_K/\mathfrak{N} = \mathbb{Z}/N.$$
Fix such an $\mathfrak{N}$ and write $\mathfrak{N} = \prod_{\ell|N} v$ for primes $v$ of $\mathcal{O}_K$ with $v|\ell$. Henceforth, let $a_n(f)$ denote the $n^{th}$ Fourier coefficient of $f$ (i.e., the $n^{th}$ coefficient of the $q$-expansion at $\infty$); when $f$ is obvious from context, we will often abbreviate $a_n(f)$ to $a_n$. Moreover, when $f$ is defined over $\mathbb{Q}$ and thus associated with an elliptic curve $E/\mathbb{Q}$, we will sometimes write $a_n(E) = a_n(f)$.

\begin{remark}The assumption that $D_K < -4$ is odd is made for calculational convenience in Bertolini-Darmon-Prasanna \cite[Remark 4.7]{BDP2}, and can most likely be removed. See Liu-Zhang-Zhang's recent generalization of the Bertolini-Darmon-Prasanna $L$-function and their $p$-adic Waldspurger formula \cite{LiuZhangZhang}.
\end{remark}

Note that our fixed embedding $i_p : \overline{\mathbb{Q}} \hookrightarrow \mathbb{C}_p$ determines a prime ideal $\mathfrak{p}$ of $\mathcal{O}_K$ above $p$. Let $\mathcal{L}_p(f,\chi)$ denote the Bertolini-Darmon-Prasanna anticyclotomic $p$-adic $L$-function described in Section \ref{twoformulas}, and let $L_p(\chi,0)$ denote the Katz $p$-adic $L$-function described in Section \ref{Katzsection}. Let $F'$ and $H_{\mathfrak{N}}$ be the fields defined in Section \ref{twoformulas}, and let $\mathfrak{p}'$ be the prime ideal of $\mathcal{O}_{F'}$ above $\mathfrak{p}$ determined by $i_p$. 

\begin{theorem}[Main Theorem]\label{newmainthm} Let $(f,p,K)$ be as in Assumptions \ref{assumptions}. Suppose $f$ has Eisenstein descent of type $(\psi_1,\psi_2,N_+,N_-,N_0)$ mod $\mathfrak{m}$ (see Definition \ref{Eisensteindescent}) for some integral ideal $\mathfrak{m}$ of the ring of integers of a $p$-adic field $M$ containing $E_f$, $\psi_1$ and $\psi_2$ are Dirichlet characters over $\mathbb{Q}$ with $\psi_1\psi_2 = \varepsilon_f$. Let $\mathfrak{t}$ be an integral ideal of $\mathcal{O}_K$ with $\mathcal{O}_K/\mathfrak{t} = \mathbb{Z}/\mathfrak{f}(\psi_2)$ and $\mathfrak{t}|\mathfrak{N}$. For all $\chi\in\hat{\Sigma}_{cc}(\mathfrak{N})$ (see Section \ref{twoformulas} for a precise definition), we have the following congruence:
\begin{align*}&\mathcal{L}_p(f,\chi) \equiv \psi_1^{-1}(D_K)\left(\frac{\mathfrak{f}(\psi_2)^k\chi^{-1}(\overline{\mathfrak{t}})}{4\mathfrak{g}(\psi_2^{-1})(2\pi i)^{k+2j}}\cdot\Xi\cdot L_p(\psi_{1/K}\chi^{-1},0)\right)^2 \hspace{-.15cm}\mod \mathfrak{m}\mathcal{O}_{F_{\mathfrak{p}'}'M}
\end{align*}
where $k+2j \in \mathbb{Z}/(p-1)\times\mathbb{Z}_p$ is the signature of $\chi \in \hat{\Sigma}_{cc}(\mathfrak{N})$ (see Definition \ref{signaturedefinition}), and 
\begin{align*}\Xi =&\prod_{\ell|N_+}\left(1-(\psi_{2/K}\chi^{-1})(\overline{v})\ell^{k-1}\right)\prod_{\ell|N_-}(1-(\psi_{1/K}\chi^{-1})(\overline{v}))\\
\cdot&\prod_{\ell|N_0}(1-(\psi_{2/K}\chi^{-1})(\overline{v})\ell^{k-1})\left(1-(\psi_{1/K}\chi^{-1})(\overline{v})\right).
\end{align*}
\end{theorem}

Invoking Bertolini-Darmon-Prasanna's $p$-adic Waldspurger formula (Theorem \ref{AbelJacobicongruence'}, see \cite[Theorem 5.12]{BDP2}), we can identify the left hand side of Theorem \ref{newmainthm} with a $p$-adic Abel-Jacobi image of a generalized Heegner cycle, and thus derive the following congruence. 

\begin{corollary}\label{newcorollary}Suppose $\chi \in \Sigma_{cc}^{(1)}(\mathfrak{N})$ (see Section \ref{twoformulas} for a precise definition) with infinity type $(k-1-j,1+j)$ where $0 \le j \le k-2$. In the setting of Theorem \ref{newmainthm}, we have, for $0 \le j \le k-2$,
\begin{align*}&\Omega_p^{2(k-2-2j)}\left(\frac{1-\chi^{-1}(\overline{\mathfrak{p}})a_p(f)+\chi^{-2}(\overline{\mathfrak{p}})\varepsilon_f(p)p^{k-1}}{\Gamma(j+1)}\right)^2\left(AJ_{F_{\mathfrak{p}'}'}(\Delta(\chi\mathbb{N}_K))(\omega_f\wedge \omega_A^j\eta_A^{k-2-j})\right)^2\\
&\equiv\psi_1^{-1}(D_K)\left(\frac{\mathfrak{f}(\psi_2)^k\chi^{-1}(\overline{\mathfrak{t}})}{4\mathfrak{g}(\psi_2^{-1})(2\pi i)^{k-2-2j}}\cdot\Xi\cdot L_p(\psi_{1/K}\chi^{-1},0)\right)^2 \mod \mathfrak{m}\mathcal{O}_{F_{\mathfrak{p}'}'M}.
\end{align*}
Here $\Omega_p$ is the $p$-adic period defined in Section \ref{twoformulas}, and
$$\Delta(\chi\mathbb{N}_K) := \sum_{[\mathfrak{a}] \in \mathcal{C}\ell(\mathcal{O}_K)}(\chi\mathbb{N}_K)^{-1}(\mathfrak{a})\cdot \Delta_{\varphi_{\mathfrak{a}}\varphi_0}\in \operatorname{CH}^{k-1}(X_{k-2})_{0,E_{\chi}}(H_{\mathfrak{N}})$$
where $\Delta_{\varphi_{\mathfrak{a}}\varphi_0}$ is defined as in Section \ref{twoformulas}.
\end{corollary}

\begin{remark}\label{BBremark}In light of the discussion in Section \ref{Chowmotives'}, Corollary \ref{newcorollary} can be viewed as providing a new method of verifying a consequence of the Beilinson-Bloch conjecture: namely, for $\chi \in \Sigma_{cc}^{(1)}(\mathfrak{N})$, we have $L(\pi_f\times\pi_{\chi^{-1}},\frac{1}{2}) = 0$, which implies $L(\pi_f\times\pi_{(\chi^{-1})_{/F}},\frac{1}{2}) = 0$ for any $F/H_{\mathfrak{N}}$ in situation $(S)$, and hence the Beilinson-Bloch conjecture predicts that $\dim_{E_{f,\chi}}\epsilon_{(f,\chi)_{/F}}\operatorname{CH}^{k-1}(X_{k-2})(F) \ge 1$. Corollary \ref{newcorollary} shows that in the setting where $f$ has Eisenstein descent mod $\mathfrak{m}$, if 
$$\Xi\cdot L_p(\psi_{1/K}\chi^{-1},0) \not\equiv 0 \mod \mathfrak{m}\mathcal{O}_{F_{\mathfrak{p}'}'},$$ 
then the generalized Heegner cycle $\Delta(\chi\mathbb{N}_K)$ provides the ostensible non-trivial algebraic cycle class. We are thus reduced to verifying the $p$-non-divisibility of the above special values of the Katz $p$-adic $L$-function. In Corollary \ref{newmaincorollary} we explicitly evaluate Katz $L$-values corresponding to \emph{classical} Heegner cycles (cf. Remark \ref{classicalHeegnerremark}) using a theorem of Gross in order to find explicit non-$p$-divisibility criteria. One might expect to be able to carry out studies of $p$-non-divisibility for more general Katz $L$-values in order to establish similar non-triviality criteria for more general Heegner cycles, but we do not do this here.
\end{remark}

Suppose that $\varepsilon_f = 1$ and $k$ is \emph{even}. Note the anticyclotomic line intersects with the cyclotomic line when $j = \frac{k}{2}-1$, in the notation of Corollary \ref{newcorollary}. For this $j$, we can take the particular character $\chi = \mathbb{N}_K^{-k/2} \in \Sigma_{cc}^{(1)}(\mathfrak{N})$. Applying Gross's factorization (Theorem \ref{Grosstheorem'}) of the Katz $L$-function on the cyclotomic line to the right hand side of Theorem \ref{newmainthm}, we get the following explicit congruences between special values of $p$-adic $L$-functions. First, fix the following notation for convenience.

\begin{definition}\label{evendefinition}Suppose we are given an imaginary quadratic field $K$. For any Dirichlet character $\psi$ over $\mathbb{Q}$, let 
\begin{align*}
\psi_0 := \begin{cases}\psi & \text{if} \; \psi \; \text{even},\\
\psi\varepsilon_K & \text{if} \; \psi \; \text{odd}.\\
\end{cases}
\end{align*}
\end{definition}

\begin{theorem}\label{newmaincorollary}In the setting of Theorem \ref{newmainthm}, specializing to $\varepsilon_f = 1$, $k$ even and $\psi_1 = \psi_2^{-1} = \psi$, we have the following congruence:
\begin{align*}&\left(\frac{p^{k/2}-a_p(f) + p^{k/2-1}}{p^{k/2}\Gamma(\frac{k}{2})}\right)^2AJ_{F_{\mathfrak{p}'}'}(\Delta(\mathbb{N}_K^{1-k/2}))^2(\omega_{f}\wedge\omega_A^{k/2-1}\eta_A^{k/2-1}) \\
&\equiv \frac{\Xi^2}{4}\\
&\cdot\begin{cases}\left(\frac{1}{k}(1-\psi^{-1}(p)p^{k/2-1})B_{\frac{k}{2},\psi_0^{-1}\varepsilon_K^{k/2}}\cdot L_p(\psi_0(\varepsilon_K\omega)^{1-k/2},\frac{k}{2})\right)^2 & \text{if}\;\psi_0(\varepsilon_K\omega)^{1-k/2} \neq 1 \; \text{or} \; k > 2,\\
\left(\frac{p-1}{p}\log_p(\overline{\alpha})\right)^2 & \text{otherwise}
\end{cases} \\
&\hspace{13.5cm} \mod \mathfrak{m}\mathcal{O}_{F_{\mathfrak{p}'}'M}.
\end{align*}

If we take $\mathfrak{m} = \lambda$ where $\lambda$ is some prime ideal of $\mathcal{O}_M$ above $p$, then more explicitly we have
\begin{align*}&\left(\frac{p^{k/2}-a_p(f) + p^{k/2-1}}{p^{k/2}\Gamma(\frac{k}{2})}\right)^2AJ_{F_{\mathfrak{p}'}'}(\Delta_{f}(\mathbb{N}_K^{1-k/2}))^2(\omega_{f}\wedge\omega_A^{k/2-1}\eta_A^{k/2-1})\\
&\equiv \frac{\Xi^2}{4}\\
&\cdot\begin{cases} 
\left(\frac{1}{k}(1-\psi^{-1}(p)p^{k/2-1})(1-(\psi\omega^{-k/2})(p))B_{\frac{k}{2},\psi_0^{-1}\varepsilon_K^{k/2}}B_{1,\psi_0\varepsilon_K(\varepsilon_K\omega)^{-k/2}}\right)^2 & \text{if}\;\psi_0(\varepsilon_K\omega)^{1-k/2}\neq 1,\\
\left(\frac{1}{k}(1-\psi^{-1}(p)p^{k/2-1})B_{\frac{k}{2},\psi_0^{-1}\varepsilon_K^{k/2}}\cdot L_p(1,\frac{k}{2})\right)^2 & \hspace{-1.2cm}\text{if}\;\psi_0(\varepsilon_K\omega)^{1-k/2} = 1, k>2,\\
\left(\frac{p-1}{p}\log_p(\overline{\alpha})\right)^2 & \text{otherwise}\\
\end{cases}\\
&\hspace{13.5cm}\mod \lambda\mathcal{O}_{F_{\mathfrak{p}'}'M}.
\end{align*}
Here
$$\Delta(\mathbb{N}_K^{1-k/2}) := \sum_{[\mathfrak{a}]\in \mathcal{C}\ell(\mathcal{O}_K)}\mathbb{N}_K^{k/2-1}(\mathfrak{a})\cdot \Delta_{\varphi_{\mathfrak{a}}\varphi_0}\in \operatorname{CH}^{k-1}(X_{k-2})_{0,\mathbb{Q}}(H_{\mathfrak{N}})$$
where $\Delta_{\varphi_{\mathfrak{a}}\varphi_0}$ is defined as in Section \ref{twoformulas}, and
\begin{align*}\Xi =&\prod_{\ell|N_+}\left(1-\frac{\psi^{-1}(\ell)}{\ell^{1-k/2}}\right)\prod_{\ell|N_-}\left(1-\frac{\psi(\ell)}{\ell^{k/2}}\right)\prod_{\ell|N_0}\left(1-\frac{\psi^{-1}(\ell)}{\ell^{1-k/2}}\right)\left(1-\frac{\psi(\ell)}{\ell^{k/2}}\right),
\end{align*}
$\log_p$ is the Iwasawa $p$-adic logarithm (i.e with branch $\log_p(p) = 0$), and $\overline{\alpha}\in\mathcal{O}_K$ such that $(\overline{\alpha}) = \overline{\mathfrak{p}}^{h_K}$. 
\end{theorem}

\begin{remark}\label{classicalHeegnerremark}In the setting of Theorem \ref{newmaincorollary}, the generalized Heegner cycle $\Delta(\mathbb{N}_K^{1-k/2})$ can be mapped via a correspondence to a \emph{classical} Heegner cycle arising on the classical Kuga-Sato variety $W_{k-2}$ defined in Section \ref{Chowmotives'}; cf., for example, \cite[Proposition 4.1.1, Theorem 4.1.3]{BDP3}. Therefore, in view of Remark \ref{BBremark}, Theorem \ref{newmaincorollary} gives an explicit congruence criterion for showing the non-triviality of a classical Heegner cycle class and thus verifying the aforementioned consequence of the Beilinson-Bloch conjecture.
\end{remark}

\begin{remark}\label{classnumberdivisibilityremark}For any $x \in \overline{\mathfrak{p}}$, we have $v_p(x) = 0$, so we can write $x = a + 2\pi$ for $\pi \in \mathfrak{p}\mathcal{O}_{K_{\mathfrak{p}}}$ and $a \in \mathcal{O}_{K_{\mathfrak{p}}}^{\times}$. Then $x^{h_K} = a^{h_K} + h_Ka^{h_K-1}2\pi + \ldots$, so $v_p\left(\frac{p-1}{p}\log_p(\overline{\alpha})\right) = v_p(2h_K)$. 
\end{remark}

\begin{remark}In certain situations, one can use explicit special value formulas to further evaluate the right hand side of the congruences in Theorem \ref{newmaincorollary}.
For example, when $k = 2$ and $\psi_0 \neq 1$, by Leopoldt's formula we have
\begin{align*}
L_p(\psi_0,1) = \left(1-\frac{\psi(p)}{p}\right)\frac{\mathfrak{g}(\psi_0)}{\mathfrak{f}(\psi_0)}\sum_{a = 1}^{\mathfrak{f}(\psi_0)}\psi_0(a)\log_p\left(1-\exp\left(\frac{2\pi ia}{\mathfrak{f}(\psi_0)}\right)\right).
\end{align*}
Furthermore, Diamond computes explicit formulas for the values of $L_p(\chi,n)$ when $n$ is a positive integer, see \cite{Diamond}.
\end{remark}

\begin{definition}Attached to any normalized newform $f \in S_k(\Gamma_0(N),\varepsilon_f)$ and prime $\lambda$ above $p$ of a number field $M$ containing $E_f$ is a unique semisimple $\lambda$-adic Galois representation $\rho_f : \Gal(\overline{\mathbb{Q}}/\mathbb{Q}) \rightarrow GL_2(M_{\lambda})$ (see Section \ref{EisensteinDescent}). Choosing a $\Gal(\overline{\mathbb{Q}}/\mathbb{Q})$-invariant lattice, one can obtains an associated semisimple mod $\lambda$ Galois representation $\bar{\rho}_f$ (which turns out to be independent of the choice of lattice). If $p\nmid N$ and $\bar{\rho}_f$ is reducible, we will see (Theorem \ref{classify}) that $N = N_+N_-N_0$ where $N_+N_-$ is squarefree and $N_0$ is squarefull, such that there are Dirichlet characters $\psi_1$ and $\psi_2$ over $\mathbb{Q}$ with $\psi_1\psi_2 = \varepsilon_f$, and where $\ell|N_+$ implies $a_{\ell} \equiv \psi_1(\ell) \mod \lambda$, $\ell|N_-$ implies $a_{\ell} \equiv \psi_2(\ell)\ell^{k-1} \mod \lambda$, and $\ell|N_0$ implies $a_{\ell} \equiv 0 \mod \lambda$. We then say that $\bar{\rho}_f$ is \emph{reducible of type $(\psi_1,\psi_2,N_+,N_-,N_0)$}. 
\end{definition}

By Theorem \ref{classify}, we immediately get the following corollary.

\begin{corollary}Let $(f,p,K)$ be as in Assumptions \ref{assumptions}, and suppose $\bar{\rho}_f$ is reducible of type $(\psi_1,\psi_2,N_+,N_-,N_0)$, where $\bar{\rho}_f = \rho_f \mod \lambda$ for some prime ideal $\lambda$ of the ring of integers of a $p$-adic field containing $E_f$. Then the conclusions of Theorem \ref{newmainthm} and Theorem \ref{newmaincorollary} hold mod $\lambda$.
\end{corollary}

Applying Theorem \ref{newmaincorollary} to the case when $k = 2$ and $f$ is defined over $\mathbb{Q}$, and is thus associated with an elliptic curve $E/\mathbb{Q}$ of conductor $N$, we can show the Heegner point $P_E(K)$ is non-torsion when the primes dividing $pN$ satisfy certain congruence conditions and $p$ does not divide certain Bernoulli numbers. First, decompose the conductor $N = N_{\splt}N_{\nonsplit}N_{\add}$ such that $\ell|N_{\splt}$ implies $\ell$ is of split multiplicative reduction, $\ell|N_{\nonsplit}$ implies $\ell$ is of nonsplit multiplicative reduction, and $\ell|N_{\add}$ implies $\ell$ is of additive reduction. 

\begin{theorem}\label{Heegnercorollary}Suppose $E/\mathbb{Q}$ is any elliptic curve of conductor $N$ with reducible mod $p$ Galois representation $E[p]$, or equivalently, $E[p]^{ss} \cong \mathbb{F}_p(\psi)\oplus \mathbb{F}_p(\psi^{-1}\omega)$, where $p > 2$ is a prime of good reduction for $E$ and $\psi$ is some Dirichlet character with $\mathfrak{f}(\psi)^2|N_{\add}$. Suppose further that
\begin{enumerate}
\item $\psi(p) \neq 1$,
\item $N_{\splt} = 1$ (i.e. $E$ has no primes of split multiplicative reduction),
\item $\ell|N_{\add}$ implies either $\psi(\ell) \neq 1$ and $\ell \not\equiv -1 \mod p$, or $\psi(\ell) = 0$.
\end{enumerate}
Then for any imaginary quadratic field $K$ such that $p$ splits in $K$, $K$ satisfies the Heegner hypothesis with respect to $E$, and $p\nmid B_{1,\psi_0^{-1}\varepsilon_K}B_{1,\psi_0\omega^{-1}}$, the associated Heegner point $P_{E}(K)\in E(K)$ (see Section \ref{twoformulas}) is non-torsion. In particular, $\text{rank}_{\mathbb{Z}}E(K) = 1$. 
\end{theorem}

\begin{remark}When $\psi$ is quadratic, the condition $p\nmid B_{1,\psi_0\varepsilon_K}$ of Theorem \ref{Heegnercorollary} is equivalent to $p\nmid h_{K_{\psi_0}\cdot K}$. (This follows because $B_{1,\psi_0^{-1}\varepsilon_K} = -2\frac{h_{K_{\psi_0}\cdot K}}{|\mathcal{O}_{K_{\psi_0}\cdot K}^{\times}|}$ by the functional equation and analytic class number formula.)
\end{remark}

\begin{remark}\label{splitremark}In light of Theorems \ref{classify} and \ref{congruence'}, which imply that $a_{\ell}(E) \equiv \psi(\ell)$ or $\psi^{-1}(\ell)\ell \mod p$ for $\ell||N$ (i.e. for $\ell|N_{\splt}N_{\nonsplit}$) when $E[p]^{ss} \cong \mathbb{F}_p(\psi)\oplus\mathbb{F}_p(\psi^{-1}\omega)$, condition (2) can be phrased as
\begin{align*}
\text{\emph{(2)' for every $\ell||N$, $\psi(\ell) \equiv -1$ or $-\ell \mod p$.}}
\end{align*}
In terms of local root numbers $w_{\ell}(E)$, since $w_{\ell}(E) = -a_{\ell}(E)$ for $\ell||N$, conditions $(2)$ and $(2)'$ can also be viewed as requiring the corresponding local root numbers to all be $+1$.
\end{remark}

\begin{remark}Theorem \ref{Heegnercorollary} (at least partially) recovers a much earlier result of Mazur \cite[Theorem, p. 231]{Mazur}, which considers the case $N = \text{prime} \neq p$. In Mazur's setting, we suppose $E[p]^{ss} \cong \mathbb{F}_p\oplus \mathbb{F}_p(\omega)$ where $N = \ell \neq p$, let $\psi$ be an \emph{even} quadratic character and choose an imaginary quadratic field $K$. (In Mazur's notation, we are taking $\chi = \psi\varepsilon_K$ and $K_{\chi} = K_{\psi}\cdot K = K_{\psi\varepsilon_K}$.) The case of Mazur's theorem we recover is:
\\

\noindent Suppose $(E,p,\psi,K)$ as above further satisfy:
\begin{enumerate}
\item $K$ satisfies the Heegner hypothesis with respect to $E\otimes \psi$,
\item $p$ splits in $K$ and $D_K < -4$,
\item $\psi(p) = -1$,
\item $p\nmid B_{1,\psi\omega^{-1}}$,
\item $lcm(\ell,|\mathfrak{f}(\psi)|^2)>4$, $\psi(\ell) \neq 1$ and $p\nmid h_{K_{\psi}\cdot K}$.
\end{enumerate}
Then $\text{rank}_{\mathbb{Z}}(E\otimes \psi\varepsilon_K)(\mathbb{Q}) = 0$.
\\

The above theorem follows from Theorem \ref{Heegnercorollary} since the latter implies $\text{rank}_{\mathbb{Z}}(E\otimes \psi)(K) = 1$, and then $\text{rank}_{\mathbb{Z}}(E\otimes\psi\varepsilon_K)(\mathbb{Q}) = 0$ follows from root number considerations (see Proposition \ref{rootnumberproposition}).

Assumptions $(1)-(4)$ above are extraneous in the full generality of Mazur's theorem, with $(5)$ being the pertinent hypothesis. These assumptions have the following effects: $(1)$ is part $(4)$ of Assumptions \ref{assumptions} and guarantees that $L((E\otimes\psi)/K,1) = 0$; $(2)$ is part $(3)$ of Assumptions \ref{assumptions}; $(3)$ excludes the possibility of a ``trivial zero" from the Kubota-Leopoldt factor $L_p(\psi\varepsilon_K\omega,0)$ which shows up in the congruence of Theorem \ref{newmaincorollary}; $(4)$ essentially controls the Selmer group $\text{Sel}_p((E\otimes\psi)/\mathbb{Q})$ (see Remark \ref{Selmerremark}) so that $\text{rank}_{\mathbb{Z}}(E\otimes\psi)(\mathbb{Q}) \le 1$. Mazur's original statement also allows for $p|\mathfrak{f}(\psi)$, which is ruled out by part $(2)$ of Assumptions \ref{assumptions}. We should note, however, that Mazur requires the additional assumptions
\begin{align*}
&\hspace{-7.5cm}\text{(6) $\ell \ge 11$,}\\
&\hspace{-7.5cm}\text{(7) $p$ divides the numerator of $\left(\frac{\ell-1}{12}\right)$,}\\
&\hspace{-7.5cm}\text{(8) $(p,|\mathfrak{f}(\psi\varepsilon_K)|) \neq (3,3)$.}
\end{align*}
Assumption $(6)$ is a technical assumption and is stronger than our assumption $lcm(\ell,|\mathfrak{f}(\psi)|^2) > 4$ on the level of $E\otimes \psi$, which in turn originates from part $(1)$ of Assumptions \ref{assumptions}; $(7)$ means that the newform $f_E \in S_2(\Gamma_0(\ell))$ of $E$ has \emph{full Eisenstein descent} (see Definition \ref{Eisensteindescent}), and is automatically satisfied when $E[p]$ is reducible and $p > 3$ (see Remark \ref{constantterm0}); $(8)$ is taken care of by our assumption $D_K < -4$ from part $(3)$ of Assumptions \ref{assumptions}. 

Perhaps more interesting is the connection between Theorem \ref{Heegnercorollary} and the ``supplement" to Mazur's theorem, which states that a Heegner point is non-torsion when ``$\psi(\ell) \neq 1$'' in Assumption $(5)$ above is replaced with ``$\psi(\ell) = 1$'' (see \cite[Theorem, p. 237]{Mazur}). In particular, $E$ satisfies the Heegner hypothesis with respect to both $K$ and $K_{\psi}\cdot K$, and $E\otimes \psi\varepsilon_K$ has root number $-1$ so that $P_{E\otimes\psi}(K) \in (E\otimes\psi\varepsilon_K)(\mathbb{Q})$. As is pointed out in \cite{Mazur}, under this assumption, the Kubota-Leopoldt factor $L_p(\psi^{-1}\varepsilon_K\omega,0)$ from Theorem \ref{newmaincorollary} can be related mod $p$ to $\frac{1}{p}\frac{d}{ds}L(E,\psi\varepsilon_K,s)_{|s=1}$, itself being related to the $p$-adic height mod $p$ of $P_E(K_{\psi}\cdot K)^- := P_E(K_{\psi}\cdot K) - \overline{P_E(K_{\psi}\cdot K)} \in (E\otimes\psi\varepsilon_K)(\mathbb{Q})$. Note that $\psi(\ell) = 1$ places $E\otimes \psi$ outside the scope of Theorem \ref{Heegnercorollary}, but perhaps this descent result in tandem with Mazur's observation and the mod $p$ factorization of Theorem \ref{newmaincorollary} suggests an identity (in certain cases) relating the formal logarithm of $P_{E\otimes \psi}(K) \in (E\otimes\psi\varepsilon_K)(\mathbb{Q})$ with the $p$-adic height of  $P_E(K_{\psi}\cdot K)^- \in (E\otimes\psi\varepsilon_K)(\mathbb{Q})$ times another quantity, perhaps related to the order of $\Sh(E\otimes\psi/\mathbb{Q})[p^{\infty}]$. A deeper comparison between the results of this paper and Mazur's would be an interesting direction for further investigation. 
\end{remark}

\begin{remark}One of the referees has also pointed out an analogy between Mazur's result and Theorem \ref{Heegnercorollary}, and the two ``reciprocity laws" which Bertolini and Darmon established in their work on the anticyclotomic Iwasawa Main Conjecture for elliptic curves \cite{BD}. Let $E/\mathbb{Q}$ be an elliptic curve with newform $f_E$. Mazur's descent result assumes $E\otimes \varepsilon_K$ has root number $-1$  over $\mathbb{Q}$ and computes the image of $P_E(K)^-$ inside the $p$-primary part of the component group of the special fiber of a N\'{e}ron model of $E$ over $\mathbb{Z}_N$. Using an Eisenstein congruence at $p$, Mazur shows that non-vanishing of this image is equivalent to $p\nmid h_{K}$. Bertolini-Darmon's first reciprocity law also assumes $E$ has root number $-1$ over $K$ and computes the image inside the $p$-primary part of the component group of the special fiber of a N\'{e}ron model over $\mathbb{Z}_{\ell^2}$ of some abelian variety (arising as a quotient of $J_1(N\ell)$ for some prime $\ell\nmid N$ which is inert in $K$). (Here $\mathbb{Z}_{\ell^2}$ is the unramified quadratic extension of $\mathbb{Z}_{\ell}$.) Using a level-raising congruence modulo some power of $p$ of $f_E$ with an eigenform $g$ of level $N\ell$ (so that $g$ has root number $+1$ over $K$), Bertolini and Darmon equate this image (using a Waldspurger-type formula) with a special value of an anticyclotomic $p$-adic $L$-function attached to $g$. Theorem \ref{newmaincorollary}, applied to $E$, assumes $E$ has root number $-1$ over $K$ and computes the $p$-adic formal logarithm of $P_{E}(K)$ modulo some power of $p$. We then use an Eisenstein congruence at $p$ to equate this with a special value of a Katz $p$-adic $L$-function. Bertolini-Darmon's second reciprocity law similarly assumes $E$ has root number $-1$ over $K$ and computes the image of $P_{E}(K)$ inside the finite part of some power-of-$p$-(Bloch-Kato) Selmer group of $E$. The formal logarithm at $p$ modulo this same power of $p$ in particular maps into this Selmer group. Using another level-raising congruence to an eigenform $g$ of level $N\ell_1\ell_2$ for distinct ``admissible" primes $\ell_1, \ell_2\nmid N$, Bertolini-Darmon again equate this image with a special value of an anticyclotomic $p$-adic $L$-function of $g$. It should be noted that the $p$-adic Waldspurger formula of Bertolini-Darmon-Prasanna (Theorem \ref{AbelJacobicongruence'}) is itself established through exploiting yet another ``sign change'' congruence, namely by considering congruences with a newform $f$ and its various anticyclotomic twists, considering these twists in a $p$-adic family, varying the twists through the interpolation region where the global root number is +1, and taking a $p$-adic limit into a region where the root number is $-1$.
\end{remark}

Let us return to the setting of Theorem \ref{Heegnercorollary}. Considering quadratic twists of $E$ when $E[p]^{ss} \cong \mathbb{F}_p(\psi)\oplus\mathbb{F}_p(\psi^{-1}\omega)$ (which amounts to varying $\psi$ through quadratic characters), we obtain congruence criteria for the $\mathfrak{f}(\psi)$ which, when satisfied, imply $\text{rank}_{\mathbb{Z}}(E\otimes\psi)(K) = 1$. In Section \ref{ellipticcurves}, we consider $p = 3$ and use quadratic class number 3-divisibility results to show the following.

\begin{corollary}\label{positiveproportioncorollary}Suppose $E/\mathbb{Q}$ has reducible mod $3$ Galois representation $E[3]$. We have:
\begin{enumerate}
\item If $E$ is reducible of type $(\psi,\psi^{-1},N_+,N_-,N_0)$ for some quadratic character $\psi$, then a positive proportion of quadratic twists of $E$, ordered by absolute value of discriminant, have rank $0$ or $1$. 
\item If $E$ is semistable, then a positive proportion of real quadratic twists of $E$, ordered by absolute value of discriminant, have rank 1 and a positive proportion of imaginary quadratic twists of $E$, likewise ordered, have rank 0.
\end{enumerate}
\end{corollary}
\begin{remark}For $E$ as in the statement of Corollary \ref{positiveproportioncorollary}, explicit lower bounds for these proportions are given in the statement of Theorem \ref{realtwistpositiveproportion} and Theorem \ref{twistpositiveproportion}, respectively. One can also use congruence criteria provided by Theorem \ref{Heegnercorollary} to find explicit examples of quadratic twists $E\otimes\psi$ with $\text{rank}_{\mathbb{Z}}(E\otimes\psi)(K) = 1$.
\end{remark}

\begin{remark}Recent work of Harron and Snowden (\cite{HarronSnowden}) in particular shows that 
$$\lim_{X\rightarrow \infty}\frac{N_{\mathbb{Z}/3}(X)}{X^{\frac{1}{3}}} \approx 1.5221$$
where $N_G(X)$ denotes the number of (isomorphism classes of) elliptic curves $E$ over $\mathbb{Q}$ up to (naive) height $X$ with $E(\mathbb{Q})^{tors} \cong G$. Corollary \ref{positiveproportioncorollary} applies to all curves with non-trivial rational 3-torsion, and so applies to at least a number of curves on the order of $X^{\frac{1}{3}}$ up to height $X$.
\end{remark}

The organization of this paper is as follows. In Sections \ref{preliminaries}, we review some preliminaries, including definitions pertinent to our discussion and a review of algebraic and $p$-adic modular forms. In Section \ref{twoformulas}, we recall Bertolini-Darmon-Prasanna's anticyclotomic $p$-adic $L$-function as well as their $p$-adic Waldspurger formula, which gives the value of the image of a certain generalized Heegner cycle under the $p$-adic Abel-Jacobi in terms a special value of their $p$-adic $L$-function". In Section \ref{Katzsection}, we recall Katz's $p$-adic $L$-function attached to Hecke characters over imaginary quadratic fields $K$ as well as recall Gross's factorization of the $L$-function on the cyclotomic line. In Section \ref{EisensteinDescent}, we define the notion of a normalized newform $f$ having Eisenstein descent, and relate it to the reducibility of the residual $p$-adic Galois representation of $f$.

In Section \ref{ProofofMainTheorem}, we prove our main congruence by showing that certain twisted traces over $CM$ points of Maass-Shimura derivatives $\partial^j$ of Eisenstein series represent special values of complex $L$-functions of Hecke characters. The twisted traces considered in \cite{BDP2} are exactly of the same kind, with the sole difference being that the Eisenstein series is replaced by a cusp form. Interpolating these twisted traces yields the Katz $p$-adic $L$-function and the Bertolini-Darmon-Prasanna anticylcotomic $p$-adic $L$-functions respectively. When $f$ has partial Eisenstein descent with associated Eisenstein series $G$, we use the corresponding congruence between the twisted traces of $\partial^jf$ and $\partial^jG$ to derive the congruence between $p$-adic $L$-functions.

In Section \ref{ConcreteApplications}, we apply our congruence formula to the problem of finding non-trivial algebraic cycles whose existence is predicted by the Beilinson-Bloch conjecture. We also give examples of Eisenstein descent, including ways to construct such newforms (see Constructions \ref{newformconstruction} and \ref{newformconstructionlevelN}), and applications of our main theorem to showing non-triviality of algebraic cycle classes. In Section \ref{Ramanujan}, we address the particular case of showing non-triviality of generalized Heegner cycles attached to quadratic twists of the Ramanujan $\varDelta$ function. In Section \ref{ellipticcurves}, we examine the case of elliptic curves and derive an explicit criterion to show $\text{rank}_{\mathbb{Z}}(E(K)) = 1$ for elliptic curves $E/\mathbb{Q}$ with reducible mod $p$ Galois representation. 

In Section \ref{positiveproportionsection}, we show that when $E$ has reducible mod $3$ Galois representation, a positive proportion of quadratic twists of $E$ have rank $0$ or $1$. These results in certain cases extend those of Vatsal (\cite{Vatsal}), who exhibited an infinite familiy of elliptic curves (namely, those which are semistable with rational 3-torsion), for each member of which a positive proportion of its quadratic twists have rank 0. Our result gives a larger infinite family of elliptic curves (namely, those which are semistable with reducible mod $3$ Galois representation), for each member of which a positive proportion of its quadratic twists have rank 0 and a positive proportion have rank 1.

\noindent\textbf{Acknowledgments.} The author thanks Chris Skinner for suggesting this investigation and helpful discussions, and is also indebted to Shou-Wu Zhang and Barry Mazur for valuable discussions. The author also thanks the anonymous referees for their insightful comments and suggestions. This work was partially supported by the National Science Foundation under grant DGE 1148900. Parts of the research contributing to this paper were completed while the author was a participant in the 2013 Princeton Summer Research Program.
\section{Preliminaries} \label{preliminaries}
In this section, we review some preliminaries relating to the Bertolini-Darmon-Prasanna and Katz $p$-adic $L$-functions. Our discussion follows \cite[Sections 2-3]{BDP2}.
\subsection{Algebraic modular forms}\label{algebraicmodularforms}Recall
$$\Gamma_1(N) := \left\{\gamma \in SL_2(\mathbb{Z}) : \gamma \equiv \left(\begin{array}{ccc}
1 & * \\
0 & 1 \end{array} \right) \mod N\right\}$$
$$\Gamma_0(N):= \left\{\gamma \in SL_2(\mathbb{Z}) :\gamma \equiv \left(\begin{array}{ccc}
* & * \\
0 & * \end{array} \right) \mod N\right\}$$
and note that $\Gamma_1(N) \subset \Gamma_0(N)$. From now on, let $\Gamma = \Gamma_1(N)$ or $\Gamma_0(N)$. 
For $i = 0,1$, let $Y_i(N)$ be the associated modular curve whose complex points are in bijection with the Riemann surface $\Gamma_i(N)\backslash \mathcal{H}^+$, and which classifies pairs $(E,t)$ consisting of an elliptic curve $E$ and a $\Gamma_i(N)$-level structure $t$ on $E$ (so $\mathbb{Z}/N \cong t \subset E[N]$ for $i = 0$, and $t \in E[N]$ for $i = 1$). Let $X_i(N)$ denote the compactification of $Y_i(N)$. Let $\pi: \mathcal{E} \rightarrow Y_1(N)$ denote the universal elliptic curve with $\Gamma_1(N)$-level structure. Throughout the paper, we shall use the interpretation of algebraic modular forms as global sections of powers $\overline{\omega}^k$ of the Hodge bundle $\overline{\omega}:= \pi_*\Omega_{\mathcal{E}/Y_1(N)}^1$ of relative differentials $\mathcal{E}/Y_1(N)$ (which we may extend to $X_1(N)$ when $N > 4$), so that an algebraic modular form can be thought of as a function on isomorphism classes of triples $[(E,t,\omega)]$ satisfying base-change compatibility and homogeneity conditions; here $E$ is an elliptic curve, $t$ a $\Gamma$-level structure, and $\omega$ an invariant differential on $E$. For details, see \cite[Section 1.1]{BDP2}. Denote by 
$$S_k(\Gamma,R) \subset M_k(\Gamma,R) \subset M_k^{*}(\Gamma,R)$$
the space of weight $k$ $\Gamma$-cuspforms over a ring $R$, the space of weight $k$ algebraic $\Gamma$-modular forms, and the space of weight $k$ weakly holomorphic $\Gamma$-modular forms respectively. (See Section 1.1 of loc. cit. for precise definitions.) We have similar inclusions for the corresponding spaces considered with nebentypus $\varepsilon$. We will suppress the base ring ``$R$" when it is obvious from context. 

Suppose $F \subset \mathbb{C}$ is a field containing the values of a Dirichlet character $\varepsilon$, and $f \in M_k^*(\Gamma_0(N),F,\varepsilon)$. Then we have an associated holomorphic function on $\mathcal{H}^+$ given by 
$$f(\tau) := f\left(\mathbb{C}/(\mathbb{Z}+\mathbb{Z}\tau),\frac{1}{N},2\pi i dz\right)$$
where $dz$ is the standard differential on $\mathbb{C}/\Lambda$ with period lattice equal to $\Lambda$. Note $f(\tau)$ satisfies the usual weight $k$ $\epsilon$-twisted modularity condition for $\Gamma_0(N)$:
$$f\left(\gamma\cdot\tau\right) = \varepsilon(d)j(\gamma,\tau)^kf(\tau)$$
for $\gamma=\left(\begin{array}{ccc} a & b\\
c & d\\
\end{array}\right)\in \Gamma_0(N)$. This function $f(\tau)$ in fact completely determines the weakly holomorphic modular form $f \in M_k^*(\Gamma_0(N),F,\varepsilon)$. 

Denote the space of weight $k$ nearly holomorphic $\Gamma$-modular forms over $F$ and nebentypus $\varepsilon$ by $N_k(\Gamma,F,\varepsilon)$. (See Section 1.2 of loc. cit. for precise definitions.) We have the classical \emph{Maass-Shimura} operator $\partial_k : N_k(\Gamma,F,\varepsilon) \rightarrow N_{k+2}(\Gamma,F,\varepsilon)$ acting on nearly-holomorphic $\Gamma$-modular forms by the formula
$$\partial_kf(\tau) := \frac{1}{2\pi i} \left(\frac{d}{d\tau}+ \frac{k}{\tau - \overline{\tau}}\right)f(\tau).$$
Note that $\partial_k$ does not preserve holomorphy, but preserves near-holomorphy. We let $\partial_k^{j} := \partial_{k+2(j-1)}\circ \cdots \circ \partial_k : N_k(\Gamma,F,\varepsilon) \rightarrow N_{k+2j}(\Gamma,F,\varepsilon)$, and omit $k$ when it is obvious from context. 

\subsection{$p$-adic modular forms}\label{padicmodularforms} Fix a rational prime $p$. We will interpret $p$-adic modular forms analogously to the previous definition of algebraic modular forms, i.e. as global sections of a rigid analytic line bundle over the ordinary locus of $X_i(N)/\mathbb{C}_p$, or equivalently as functions on isomorphism classes of triples $[(E,t,\omega)]$ defined over $p$-adic rings satisfying certain homogeneity and base-change properties. (For details, see \cite[Section 1.3]{BDP2}.)

We have the classical \emph{Atkin-Serre} operator $\theta : M_k^{(p)}(\Gamma,F,\varepsilon) \rightarrow M_{k+2}^{(p)}(\Gamma,F,\varepsilon)$ acting on $p$-adic modular forms, whose effect on $q$-expansions is given by
$$\theta\left(\sum_{n =0}^{\infty}a_{n}q^{n}\right) = q\frac{d}{dq}\sum_{n =0}^{\infty}a_{n}q^{n} = \sum_{n =1}^{\infty}{n}a_{n}q^{n}.$$

Now recall our complex and $p$-adic embedding $i_{\infty}: F \hookrightarrow \mathbb{C}$ and $i_p : F \hookrightarrow \mathbb{C}_p$, as well as our field isomorphism $i : \mathbb{C} \xrightarrow{\sim} \mathbb{C}_p$ such that $i_p = i \circ i_{\infty}$. Let $K$ be an imaginary quadratic field and let $H$ be the Hilbert class field over $K$. Suppose $E/F$ is a curve with complex multiplication (i.e., such that $\End_H(E)$ is isomorphic to an order of $\mathcal{O}_K$), and let $(E,t,\omega)$ be a triple defined over $F$. Using $i_{\infty}$ and $i_p$, we can view $(E,t,\omega)$ as a triple over $\mathbb{C}$ and $\mathbb{C}_p$ respectively. Using $i$, and $f \in M_k^{*}(\Gamma,F,\varepsilon)$ can be viewed as an element of $M_k^{(p)}(\Gamma,F,\varepsilon)$ (after possibly rescaling). The following Theorem, due to Shimura and Katz, states that the values of $\partial^jf$ and $\theta^jf$ coincide on ordinary $CM$ triples. 

\begin{theorem}[See Proposition 1.12 in \cite{BDP2}]\label{MSthetathm}Suppose $(E,t,\omega)$ is a triple defined over $F$ where $E$ has complex multiplication by an order of $\mathcal{O}_K$, and assume that $E$, viewed as an elliptic curve over $\mathcal{O}_{\mathbb{C}_p}$, is ordinary. Let $f \in M_k^*(\Gamma,F,\varepsilon)$. Then for $j \ge 0$, we have
\begin{enumerate}
\item $\partial^{j}f(E,t,\omega) \in i_{\infty}(F)$.
\item $\theta^j f(E,t,\omega) \in i_p(F)$.
\item Viewing these two quantities as elements of $F$ (i.e., identifying $i_{\infty}(F) \xrightarrow{i_p\circ i_{\infty}^{-1}} i_p(F)$), we have
$$\partial^jf(E,t,\omega) = \theta^j f(E,t,\omega).$$
\end{enumerate}
\end{theorem}

Denote by $\flat$ the ``$p$-depletion" operator on $M_k^{(p)}(\Gamma,F,\varepsilon)$ (see \cite[Section 3.8]{BDP2}), whose action on $q$-expansions is given by
$$\left(\sum_{n=0}^{\infty}a_nq^n\right)^{\flat} = \sum_{n=0, (n,p) = 1}^{\infty}a_{n}q^{n}.$$

\subsection{Isogenies and generalized Heegner cycles}\label{isogenies'}
Recall our assumption that the imaginary quadratic field $K$ satisfies the \emph{Heegner hypothesis} with respect to $N$: 
$$\text{for all primes}\hspace{2mm} \ell|N,\: \ell \hspace{2mm} \text{is split in} \; K/\mathbb{Q}.$$
This condition implies that there exists an integral ideal $\mathfrak{N}$ of $\mathcal{O}_K$ such that $\Nm_{K/\mathbb{Q}}(\mathfrak{N}) = N$. 
Fix a cyclic ideal $\mathfrak{N}$ of $\mathcal{O}_K$ lying above $N$, and fix a $CM$ elliptic curve $A$ with $\End_H(A) = \mathcal{O}_K$. Let $H_{\mathfrak{N}}/H$ denote the extension over which the individual $\mathfrak{N}$-torsion points of $A$ are defined, which is an abelian extension of $K$. A choice of $t_A \in A[\mathfrak{N}]$ of order $N$ determines a $\Gamma_1(N)$-level structure on $A$ defined over any $F/H_{\mathfrak{N}}$. Fix a choice of $t_A$ once and for all. 

Now consider the set 
$$\Isog(A) := \{(\varphi,A')\}/\text{Isomorphism}$$
where $A'$ is an elliptic curve and $\varphi: A \rightarrow A'$ is an isogeny defined over $\overline{K}$, and two pairs $(\varphi_1,A_1')$ and $(\varphi_2,A_2')$ are isomorphic if there is an isomorphism $\iota: A_1' \rightarrow A_2'$ over $\overline{K}$ such that $\iota\varphi_1 = \varphi_2$. Let $\Isog^{\mathfrak{N}}(A) \subset \Isog(A)$ be the subset consisting of isomorphism classes of $(\varphi,A')$ such that $\ker(\varphi) \cap A[\mathfrak{N}] = \{0\}$. Note that the group $\mathbb{I}^{\mathfrak{N}}$ of $\mathcal{O}_K$-ideals relatively prime to $\mathfrak{N}$ act on $\Isog^{\mathfrak{N}}(A)$ via $\mathfrak{a}\star(\varphi,A') = (\varphi_{\mathfrak{a}}\varphi,A'/A'[\mathfrak{a}])$ where
$$\varphi_{\mathfrak{a}} : A' \rightarrow A'/A'[\mathfrak{a}]$$
is the natural surjection. Given triples $(A_1,t_1,\omega_1)$ and $(A_2,t_2,\omega_2)$, we define an isogeny from $(A_1,t_1,\omega_1)$ to $(A_2,t_2,\omega_2)$ to be an isogeny 
$$\varphi : A_1 \rightarrow A_2 \hspace{.25cm} \text{such that}\hspace{.25cm} \varphi(t_1) = t_2 \hspace{.25cm} \text{and} \hspace{.25cm} \varphi^*(\omega_2) = \omega_1.$$
We have an action of $\mathbb{I}^{\mathfrak{N}}$ on (isomorphism classes of) triples $(A',t',\omega')$ with $\End(A') = \mathcal{O}_K$ and $t' \in A'[\mathfrak{N}]$ given by 
$$\mathfrak{a}\star(A',t',\omega') = (A'/A'[\mathfrak{a}],\varphi_{\mathfrak{a}}(t'),\omega_{\mathfrak{a}}') \hspace{.25cm} \text{where} \hspace{.25cm} \varphi_{\mathfrak{a}}^*(\omega_{\mathfrak{a}}') = \omega'.$$
 
\begin{definition}\label{generalizedHeegnercycle} Note a pair $(\varphi,A') \in \Isog^{\mathfrak{N}}(A)$ determines a point 
$(A',\varphi(t_A)) \in Y_1(N)(H_{\mathfrak{N}}) \subset Y_1(N)(F)$ and, and for any integer $r \ge 0$, an embedding $(A')^r \hookrightarrow W_r$. (For the definition of $W_r$, see Section \ref{Chowmotives'}.) Using this we view the graph as embedded in $X_r := W_r \times A^r$ (see Section \ref{Chowmotives'} for a definition):
$$\Gamma_{\varphi} \subset A^r \times (A')^r \subset A^r \times W_r \cong X_r.$$
We define the \emph{generalized Heegner cycle} associated with $(\varphi,A)$ as $\Delta_{\varphi} := \epsilon_X\Gamma_{\varphi}$, where $\epsilon_X$ is the projector defined in \cite[Section 2.2]{BDP2}. Note $\Gamma_{\phi}$ is defined over $H_{\mathfrak{N}}$, and thus determines a class in the Chow group $\operatorname{CH}_0^{r+1}(X_r)(H_{\mathfrak{N}})$, as defined in Section \ref{Chowmotives'}.
\end{definition}

\subsection{Bertolini-Darmon-Prasanna's $p$-adic $L$-function and $p$-adic Waldspurger formula}\label{twoformulas} 
Fix a normalized newform $f\in S_k(\Gamma_0(N),\varepsilon_f)$. Given a Hecke character $\chi$ of infinity type $(j_1,j_2)$, recall the central character $\varepsilon_{\chi}$ of $\chi$ is the finite order character defined by 
$$\varepsilon_{\chi} = \chi_{|\mathbb{A}_{\mathbb{Q}}^{\times}}\mathbb{N}_{\mathbb{Q}}^{j_1+j_2}.$$
We define
$$L(f,\chi^{-1},0) := L\left(\pi_f\times\pi_{\chi^{-1}},s-\frac{k-1-j_1-j_2}{2}\right)$$
where $\pi_f$ and $\pi_{\chi^{-1}}$ are the (unitarily normalized) automorphic representations associated with $f$ and $\chi$, respectively. The Hecke character $\chi$ is said to be \emph{central critical with respect to $f$} if $j_1 + j_2 = k$ and either
\begin{enumerate}
\item $1\le j_1, j_2 \le k-1$, or
\item $j_1 \ge k$ and $j_2 \le 0$
\end{enumerate}
as well as the following condition on the central character of $\chi$:
$$\varepsilon_{\chi} = \varepsilon_f.$$
(Note, since we assume the Heegner hypothesis, this implies that $\chi$ is unramified outside of $\mathfrak{f}(\varepsilon_f)$.) Central criticality of $\chi$ is equivalent to requiring that $\pi_f\times\pi_{\chi^{-1}}$ is self-dual and the point of symmetry for the functional equation of $L(f,\chi^{-1},s)$ is $s = 0$.

For central critical $\chi$ with infinity type in range $(1)$ (resp. range $(2)$) above, we have that the root number satisfies $\epsilon_{\infty}(f,\chi^{-1}) = -1$ (resp. $\epsilon_{\infty}(f,\chi^{-1}) = +1$). We henceforth make the auxiliary assumption that 
$$\text{the local root numbers of $L(f,\chi^{-1},s)$ satisfy $\epsilon_{\ell}(f,\chi^{-1})=+1$ at all finite primes $\ell$}.$$
This assumption is automatic for $\ell \in S_f = \{\ell:\ell|(N,D_K),\ell\nmid \mathfrak{f}(\varepsilon_f)\}$ since we assume the Heegner hypothesis, and thus is automatically satisfied when $(N,D_K) = 1$, as we assume in Assumptions \ref{assumptions}. These conditions on the local root numbers of $L(f,\chi^{-1},s)$ imply that the global root number is $\epsilon(f,\chi^{-1}) = -1$ for $\chi$ in infinity type range $(1)$ above, and $\epsilon(f,\chi^{-1}) = +1$ for $\chi$ in range $(2)$ above. Thus we have $L(f,\chi^{-1},0) = 0$ in range $(1)$ and expect $L(f,\chi^{-1},0) \neq 0$ (generically) in range $(2)$.

For any Hecke character $\varepsilon$ over $\mathbb{Q}$ of conductor $N_{\varepsilon}|N$, define $\mathfrak{N}_{\varepsilon}$ to be the unique ideal in $\mathcal{O}_K$ that divides $\mathfrak{N}$ and has norm equal to $N_{\varepsilon}$. For a Hecke character $\chi$ over $K$, we say that $\chi$ is of \emph{finite type} $(\mathfrak{N},\varepsilon)$ if 
$$\chi_{|\hat{\mathcal{O}}_K^{\times}} = \psi_{\varepsilon}$$
where $\psi_{\varepsilon}$ is the composite map
$$\hat{\mathcal{O}}_K^{\times} \rightarrow \prod_{v|\mathfrak{N}_{\varepsilon}}(\mathcal{O}_{K,v}/\mathfrak{N}_{\varepsilon}\mathcal{O}_{K,v})^{\times} \cong \prod_{\ell|N_{\varepsilon}}(\mathbb{Z}_{\ell}/N_{\varepsilon}\mathbb{Z}_{\ell})^{\times} \xrightarrow{\prod_{\ell|N_{\varepsilon}} \varepsilon_{\ell}} \mathbb{C}^{\times},$$
or equivalently,
$$\hat{\mathcal{O}}_K \rightarrow (\hat{\mathcal{O}}_K/\mathfrak{N}_{\varepsilon}\hat{\mathcal{O}}_K)^{\times} \cong (\mathcal{O}_K/\mathfrak{N}_{\varepsilon_K}\mathcal{O}_K)^{\times} \cong (\mathbb{Z}/N_{\varepsilon}\mathbb{Z})^{\times}\xrightarrow{\varepsilon^{-1}} \mathbb{C}^{\times}.$$

Let $\Sigma_{cc}(\mathfrak{N})$ denote the set of central critical characters of finite type $(\mathfrak{N},\varepsilon)$ satisfying the above auxiliary condition on local root numbers, with $\mathfrak{f}(\chi)|\mathfrak{N}$, and let $\Sigma_{cc}^{(1)}(\mathfrak{N})$ and $\Sigma_{cc}^{(2)}(\mathfrak{N})$ denote the subsets of such characters with infinity type $(k+j,-j)$ where $1-k \le j \le -1$ and $j\ge 0$ respectively, so that $\Sigma_{cc}(\mathfrak{N}) = \Sigma_{cc}^{(1)}(\mathfrak{N}) \sqcup \Sigma_{cc}^{(2)}(\mathfrak{N})$.

We now define a field $F'$ as follows. Note $\Gal(H_{\mathfrak{N}}/H) \cong (\mathbb{Z}/N)^{\times}$, and let $H_{\mathfrak{N}}'\subset H_{\mathfrak{N}}$ be the subfield fixed by $\ker(\varepsilon_f)$. The values $\chi(\mathfrak{a})$ generate a finite extension as $\chi$ ranges over $\Sigma_{cc}^{(2)}(\mathfrak{N})$ and $\mathfrak{a}$ ranges over $\mathbb{A}_{K,f}^{\times}$. Let $F'$ be the field which is the compositum of this latter extension with $E_f$ and $H_{\mathfrak{N}}'$, and let $\mathfrak{p}'$ be the prime ideal of $\mathcal{O}_{F'}$ above $p$ determined by our embedding $i_p$. The set $\Sigma_{cc}^{(2)}(\mathfrak{N})$ has a natural topology defined as follows, namely the topology of uniform convergence induced by the $p$-adic metric on the completion of $F'$ in $\overline{\mathbb{Q}}_p$, when $\Sigma_{cc}^{(2)}(\mathfrak{N})$ is viewed as a space of functions on finite id\`{e}les prime to $p$. (See \cite[Section 5.2]{BDP2} for details.) Let $\hat{\Sigma}_{cc}(\mathfrak{N})$ denote the completion of $\Sigma_{cc}^{(2)}(\mathfrak{N})$ in this topology. As explained in Section 5.3 of loc. cit., we may view $\Sigma_{cc}^{(1)}(\mathfrak{N}) \subset \hat{\Sigma}_{cc}(\mathfrak{N})$. 

\begin{definition}\label{signaturedefinition}We define the \emph{signature} of an element $\chi \in \hat{\Sigma}_{cc}(\mathfrak{N})$ as follows. For a Hecke character over $K$ of infinity type $(a,b)$, we define the signature of $\chi$ as $a-b \in \mathbb{Z}$. Thus, for $\chi \in \Sigma_{cc}^{(2)}(\mathfrak{N})$ of infinity type $(k+j,-j)$, this is $k+2j$. Given $\chi \in \hat{\Sigma}_{cc}(\mathfrak{N})$, we choose a Cauchy sequence $\{\chi_i\}_{i=1}^{\infty}\subset \Sigma_{cc}^{(2)}(\mathfrak{N})$ converging to $\chi$, where $\chi_i$ has infinity type $(k+j_i,-j_i)$. Given $M \in \mathbb{Z}$, there exists an $i(M)$ such that for all $i_1, i_2 \ge i(M)$, we have $\chi_{i_1}(x) \equiv \chi_{i_2}(x) \mod (\mathfrak{p}')^M$ for all $x \in \mathbb{A}_{K,f}^{\times,p}$. Evaluating on $x \in \mathbb{A}_{K,f}^{\times,p}$ congruent to $1 \mod \mathfrak{N}$, we see that $j_{i_1} \equiv j_{i_2} \mod (p-1)p^{M-1}$. Hence the Cauchy sequence $\{j_i\}_{i=1}^{\infty}\subset \mathbb{Z}$ converges to an element $j \in \mathbb{Z}/(p-1) \times \mathbb{Z}_p$. We define the signature of $\chi$ to be $k+2j \in \mathbb{Z}/(p-1)\times\mathbb{Z}_p$.
\end{definition}

The elliptic curve $A_0 = \mathbb{C}/\mathcal{O}_K$ over $\mathbb{C}$ has complex multiplication by $\mathcal{O}_K$, and hence is defined over $H$. A choice of invariant differential $\omega_0 \in \Omega_{A_0/H}^1$ determines a \emph{complex period} $\Omega_{\infty}$ via 
$$\omega_0 = \Omega_{\infty}\cdot 2\pi i dz$$
where $z$ is the standard coordinate on $\mathbb{C}$.
We now make the assumption that
$$p \; \text{is split in $K/\mathbb{Q}$}$$
so that in particular $i_p(K) \subset \mathbb{Q}_p$. Let $\mathfrak{p}$ be a prime of $K$ above $p$. Let $\mathcal{A}_0$ be a good integral model over $A_0$. The completion of $\mathcal{A}_0$ along its identity section $\hat{\mathcal{A}}_0$ is (non-canonically) isomorphic to $\hat{\mathbb{G}}_m$ over $\mathcal{O}_{\mathbb{C}_p}$. Fix an isomorphism $\iota :\hat{\mathcal{A}}_0 \xrightarrow{\sim}\hat{\mathbb{G}}_m$ (which is equivalent to fixing an isomorphism $\mathcal{A}_0[\mathfrak{p}^{\infty}] := \lim_{\xrightarrow{n}}\mathcal{A}_0[\mathfrak{p}^n] \xrightarrow{\sim} \lim_{\xrightarrow{n}}\mu_{p^n} =: \mu_{p^{\infty}}$, which is determined up to multiplication by a scalar in $\mathbb{Z}_p^{\times}$). Let $\omega_{can} := \iota^{*}\frac{du}{u}$. Now we define a \emph{$p$-adic period} $\Omega_p$ via
$$\omega_0 = \Omega_p\cdot \omega_{can}.$$

Bertolini-Darmon-Prasanna in \cite{BDP2} define an \emph{anti-cyclotomic $p$-adic $L$-function associated with $f$ and $\chi$}, interpolating a twisted trace over $CM$ triples attached to $(f,\chi)$ for $\chi \in \Sigma_{cc}^{(2)}(\mathfrak{N})$. Recall our fixed triple $(A,t_A,\omega_{can})$. Then there is an isogeny $\varphi_0: A \rightarrow A_0$, and we let $(A_0,t_0,\omega_0)$ denote the unique triple fitting into an induced isogeny of triples $\varphi_0: (A,t_A,\omega_{can}) \rightarrow (A_0,t_0,\omega_0)$ in the sense of Section \ref{isogenies'}. Henceforth, fix a complex period $\Omega_{\infty}$ and a $p$-adic period $\Omega_p$ associated with $\omega_0$. For $\mathfrak{a} \in \mathbb{I}^{\mathfrak{N}}$, let $\varphi_{\mathfrak{a}} : A_0 \rightarrow \mathfrak{a}\star A_0$ be the natural isogeny, so that $(\varphi_{\mathfrak{a}}\varphi_0,A) \in \Isog^{\mathfrak{N}}(A)$.

\begin{theorem}[\cite{BDP2} Theorem 5.9] \label{BDPinterpolationproperty}Fix a normalized newform $f\in S_k(\Gamma_0(N),\varepsilon_f)$. Suppose $p$ is a prime split in $K/\mathbb{Q}$. Let $\{\mathfrak{a}\}$ be a set of representatives of $\mathcal{C}\ell(\mathcal{O}_K)$ which are prime to $\mathfrak{N}$. Then there exists a unique continuous function $\hat{\Sigma}_{cc}(\mathfrak{N}) \rightarrow \mathbb{C}_p: \chi\mapsto \mathcal{L}_p(f,\chi)$ satisfying
\begin{align*}\mathcal{L}_p(f,\chi) &= \left[\sum_{[\mathfrak{a}] \in \mathcal{C}\ell(\mathcal{O}_K)}(\chi_j)^{-1}(\mathfrak{a})\cdot (\theta^jf)^{\flat}(\mathfrak{a}\star(A_0,t_0,\omega_{can}))\right]^2
\end{align*}
for all $\chi \in \Sigma_{cc}^{(2)}(\mathfrak{N})$ of infinity type $(k+j,-j)$ with $j\ge 0$.
\end{theorem}

Bertolini-Darmon-Prasanna's $p$-adic Waldspurger formula, which relates the special value of $\mathcal{L}_p(f,\chi)$ at some $\chi \in \Sigma_{cc}^{(1)}(\mathfrak{N}) \subset \hat{\Sigma}_{cc}(\mathfrak{N})$ to the $p$-adic Abel-Jacobi image of a certain generalized Heegner cycle, is the main result of \cite{BDP2}.
\begin{theorem}[\cite{BDP2} Theorem 5.13] \label{AbelJacobicongruence'} Suppose $\chi \in \Sigma_{cc}^{(1)}(\mathfrak{N})$ has infinity type $(k-1-j,1+j)$ where $0\le j \le r = k-2$. Then
\begin{align*}\frac{\mathcal{L}_p(f,\chi)}{\Omega_p^{2(r-2j)}} &= (1-\chi^{-1}(\overline{\mathfrak{p}})a_p(f) + \chi^{-2}(\overline{\mathfrak{p}})\varepsilon_f(p)p^{k-1})^2\\
&\hspace{.5cm}\cdot\left[\frac{1}{\Gamma(j+1)}\sum_{[\mathfrak{a}]\in \mathcal{C}\ell(\mathcal{O}_K)}(\chi\mathbb{N}_K)^{-1}(\mathfrak{a})\cdot AJ_{F_{\mathfrak{p}'}'}(\Delta_{\varphi_{\mathfrak{a}}\varphi_0})(\omega_f\wedge\omega_A^j\eta_A^{r-j})\right]^2.
\end{align*}
Here $AJ_{F_{\mathfrak{p}'}'}$ is the $p$-adic Abel-Jacobi map defined in \cite[Section 3.4]{BDP2}, and $\Delta_{\varphi_{\mathfrak{a}}\varphi_0}$ is the generalized Heegner cycle attached to $(\varphi_{\mathfrak{a}}\varphi_0,A) \in \Isog^{\mathfrak{N}}(A)$ as defined in Definition \ref{generalizedHeegnercycle}.
\end{theorem}

Suppose $\chi$ is a finite order Hecke character on $K$. We define the Heegner point attached to $\chi$ as 
$$P(\chi) := \sum_{\sigma \in \Gal(H_{\mathfrak{N}}/K)} \chi^{-1}(\sigma)\cdot \Delta^{\sigma}\in J_1(N)(H_{\mathfrak{N}}) \otimes_{\mathcal{O}_{E_f}} E_{f,\chi}$$
where $\Delta = [(A_0, t, \omega_0)] - (\infty) \in J_1(N)(H_{\mathfrak{N}})$ and the isomorphism class $[(A_0,t,\omega_0)]$ is viewed as a point in $X_1(N)(H_{\mathfrak{N}})$. The embedding $i_p : \overline{\mathbb{Q}} \hookrightarrow \mathbb{C}_p$ allows us define the formal group logarithm on the formal group $\log_{\omega_f} : \hat{J_1}(N)(\mathfrak{p}'\mathcal{O}_{F_{\mathfrak{p}'}'}) \rightarrow \mathbb{C}_p$, which we extend to a map
$$\log_{\omega_f} : J_1(N)(F_{\mathfrak{p}'}')\otimes_{\mathcal{O}_{E_f}}E_{f,\chi} \rightarrow \mathbb{C}_p$$
by linearity.

For $k = 2$, we have $AJ_{F_{\mathfrak{p}'}'}(P(\chi))(\omega_f) = \log_{\omega_f}(P(\chi))$, and thus Theorem \ref{AbelJacobicongruence'} becomes
\begin{theorem}\label{BDP1thm} Suppose $\chi$ is a finite order Hecke character on $K$ with $\chi \mathbb{N}_K \in \Sigma_{cc}^{(1)}(\mathfrak{N})$. Then
$$\mathcal{L}_p(f,\chi\mathbb{N}_K) = (1 - \chi^{-1}(\overline{\mathfrak{p}})p^{-1}a_p(f) + \chi^{-2}(\overline{\mathfrak{p}})\varepsilon_f(p)p^{-1})^2\log_{\omega_f}^2(P(\chi)).$$
\end{theorem}

\subsection{Katz's $p$-adic $L$-function and Gross's factorization on the cyclotomic line}\label{Katzsection}We can view any $p$-adic Hecke character $\chi : K^{\times}\backslash \mathbb{A}_K^{\times} \rightarrow \mathbb{C}_p^{\times}$, as a Galois character $\chi : \Gal(\overline{K}/K) \rightarrow \mathcal{O}_{\mathbb{C}_p}^{\times}$ via the Artin isomorphsim of class field theory. Let $\mathfrak{C}$ be an integral ideal of $\mathcal{O}_K$ which is prime to $p$, and let $K(\mathfrak{C}p^r)$ denote the ray class field of $K$ of conductor $\mathfrak{C}p^r$ and let $K(\mathfrak{C}p^{\infty}) = \bigcup_r K(\mathfrak{C}p^r)$. Given a $p$-adic character $\chi$ arising from a type $A_0$ algebraic Hecke character, denote its complex counterpart by $(\chi)_{\infty}$.

We recall Katz's $p$-adic $L$-function for CM fields (applied to our fixed imaginary quadratic field $K$), as well as Hida-Tilouine's extension to the case of general auxiliary conductor \cite{HidaTilouine}. Here we use the normalization found in \cite{Gross} (with $\delta = \frac{\sqrt{D_K}}{2}$ and so $\mathfrak{c} = 1$, in the notation of \cite[Theorem 5.3.0]{Katz1}).
Fix a decomposition $\mathfrak{C} = \mathfrak{F}\mathfrak{F}_c\mathfrak{I}$ where $\mathfrak{F}\mathfrak{F}_c$ consists of split primes in $K$, $\overline{\mathfrak{F}}\subset\mathfrak{F}_c$, and $\mathfrak{I}$ consists of inert or ramified primes in $K$. Recall our fixed identification $i : \mathbb{C} \xrightarrow{\sim} \mathbb{C}_p$, the prime $\mathfrak{p}|p$ determining $K\hookrightarrow \overline{\mathbb{Q}}_p$, and the periods $\Omega_p$ and $\Omega_{\infty}$ defined in Section \ref{twoformulas}.

\begin{theorem}[\cite{Katz1}, see also \cite{HidaTilouine}]\label{KatzpadicLfunction'} There exists a unique $p$-adic analytic function $\chi \mapsto L_p(\chi,0)$ for $\chi : \Gal(K(\mathfrak{C}p^{\infty})/K) \rightarrow \mathbb{C}_p^{\times}$ which satisfies (under our fixed identification $i: \mathbb{C} \xrightarrow{\sim} \mathbb{C}_p$)
\begin{align*}&L_p(\chi,0) = 4\text{Local}_{\mathfrak{p}}(\chi)(-1)^{k_1+k_2}\\
&\hspace{2cm}\cdot\left(\frac{\Omega_p}{\Omega_{\infty}}\right)^{k_1-k_2}\left(\frac{2\pi i}{\sqrt{D_K}}\right)^{-k_2}(k_1-1)!(1-\chi(\overline{\mathfrak{p}}))(1-\check{\chi}(\overline{\mathfrak{p}}))L(\chi,0)\prod_{v|\mathfrak{C}}(1-\chi(v))
\end{align*}
for characters $\chi$ with infinity type $(-k_1,-k_2)$ where $k_1 \ge 1$ and $k_2 \le 0$, and such that $\mathfrak{f}(\chi)$ is divisible by all the prime factors of $\mathfrak{F}$. Here $\text{Local}_{\mathfrak{p}}(\chi)$ is a complex scalar associated with our fixed embedding $K \hookrightarrow \overline{\mathbb{Q}}_p$, as in \cite{Katz1}. (We have in fact that $\text{Local}_{\mathfrak{p}}(\chi) = 1$ if $\chi$ is unramified at $\mathfrak{p}$.)
 
Recall the dual character $\check{\chi}$ defined by $\check{\chi}(\mathfrak{a}) = \chi^{-1}(\overline{\mathfrak{a}})\mathbb{N}_K(\mathfrak{a})$. $L_p$ satisfies the functional equation
$$W((\chi)_{\infty},\sqrt{D_K})L_p(\check{\chi},0) = L_p(\chi,0)$$
where $W((\chi)_{\infty},\delta)$ is given by
$$W((\chi)_{\infty},\delta) = \prod_{v|\mathfrak{F}}G((\chi)_{\infty,v}^{-1},\overline{\delta})\prod_{v|\mathfrak{F}_c}G((\chi)_{\infty,\overline{v}}^{-1},\delta)\prod_{v|\mathfrak{I}}G((\chi)_{\infty,v}^{-1},\delta)$$
and
$$G((\chi)_{\infty,v},\delta) = (\chi)_{\infty,v}(\pi_v^{-e})\sum_{u\in (\mathcal{O}_{K,v}/\pi_v^e\mathcal{O}_{K,v})^{\times}}(\chi)_{\infty,v}(u)\Psi_{K,v}(-u\pi_v^{-e}\delta^{-1})$$
where $e$ is the exponent of $v$ in $\mathfrak{f}(\chi)$, $\pi_v$ is a local uniformizer of $K_v$, and $\Psi_K : \mathbb{A}_K \rightarrow \mathbb{C}$ is the standard additive character normalized so that $\Psi_{K,\infty}(x_{\infty}) = \exp(2\pi i\text{Tr}_{\mathbb{C}/\mathbb{R}}(x_{\infty}))$.
\end{theorem}

Henceforth, put
$$G((\chi)_{\infty},\delta) = \prod_{v|\mathfrak{F}}G((\chi)_{\infty,v},\delta).$$

Now define the $p$-adic $L$-function
$$L_p(\chi,s) := L_p(\chi\langle \mathbb{N}_K\rangle^{s},0).$$

Gross \cite{Gross} gives the following factorization of $L_p$ along the cyclotomic line.

\begin{theorem}[\cite{Gross}]\label{Grosstheorem'} In the situation of Theorem \ref{KatzpadicLfunction'}, supose that $\mathfrak{F}_c = \overline{\mathfrak{F}}$, $\mathfrak{I} = 1$ and $\mathfrak{C} = f\mathcal{O}_K = \mathfrak{F}\overline{\mathfrak{F}}$. Suppose $\chi : \Gal(K(\mu_{fp^{\infty}})/\mathbb{Q}) \rightarrow \mathbb{C}_p^{\times}$ is a continuous $p$-adic Galois character which is trivial on complex conjugation (and so corresponds to an even $p$-adic Hecke character). Let $\chi_{/K}$ be the restriction of $\chi$ to $\Gal(K(\mu_{fp^{\infty}})/K)$. Then
$$\frac{f^s}{\prod_{\ell|f}\chi_{\ell}^{-1}(-\sqrt{D_K})\mathfrak{g}(\chi_{\ell}^{-1})}L_p(\chi_{/K},s) = L_p(\chi\varepsilon_K \omega,s) L_p(\chi^{-1},1-s).$$
Here the terms on the right are the usual Kubota-Leopoldt $p$-adic $L$-functions for algebraic Hecke characters over $\mathbb{Q}$.
\end{theorem}
\begin{remark}Gross in \cite{Gross} originally proves Theorem \ref{Grosstheorem'} for auxiliary conductor $f = 1$, even though his proof translates \emph{mutatis mutandis} to the case of arbitrary $f$. The general auxiliary conductor version of the theorem seems to be widely present in the literature, although the author has not been able to locate a complete proof. Thus, for the sake of completeness, we give a proof of Theorem \ref{Grosstheorem'}, following Gross's method. We will first need a lemma.
\end{remark}
\begin{lemma}\label{Gausssumlemma} In the situation of Theorem \ref{KatzpadicLfunction'}, suppose that $\mathfrak{F}_c = \overline{\mathfrak{F}}$, $\mathfrak{I} = 1$ and $\mathfrak{C} = f\mathcal{O}_K = \mathfrak{F}\overline{\mathfrak{F}}$. Suppose $\chi = \tilde{\chi}\circ\mathbb{N}_K$ for some Hecke character $\tilde{\chi}$ over $\mathbb{Q}$. Then we have
$$W((\chi)_{\infty},\delta) = \frac{G((\chi)_{\infty}^{-1},\overline{\delta})}{G((\check{\chi})_{\infty}^{-1},-\overline{\delta})}.$$
\end{lemma}
\begin{proof}Since $\chi = \tilde{\chi}\circ\mathbb{N}_K$, we have $\chi_v(x) = \chi_{\overline{v}}(\overline{x})$ for any place $v$ of $K$. For each $v|\mathfrak{F}\overline{\mathfrak{F}}$, since $K_v = \mathbb{Q}_{\ell}$ (where $\ell$ is the rational prime below $v$) we can choose representatives $\{u\}$ of $(\mathcal{O}_K/\pi_v^e\mathcal{O}_K)^{\times} = (\mathbb{Z}/\ell^e\mathbb{Z})^{\times}$ such that $\overline{u} = u$ and local uniformizers $\pi_v$ such that $\overline{\pi_v} = \pi_v$. (This is accomplished by replacing $\pi_v$ with $\pi_v\overline{\pi_v}$ if needed.)

Thus, we have for any $v|\mathfrak{F}$,
\begin{align*}&G((\check{\chi})_{\infty,v}^{-1},\delta) = (\check{\chi})_{\infty,v}^{-1}(\pi_v^{-e})\sum_{u\in (\mathcal{O}_{K,v}/\pi_v^e\mathcal{O}_{K,v})^{\times}}(\check{\chi})_{\infty,v}^{-1}(u)\Psi_{K,v}(-u\pi_v^{-e}\delta^{-1})\\
&=(\mathbb{N}_K)_{\infty,v}^{-1}(\pi_v^{-e})(\chi)_{\infty,\overline{v}}(\overline{\pi_v^{-e}})\sum_{u\in (\mathcal{O}_{K,v}/\pi_v^e\mathcal{O}_{K,v})^{\times}}(\mathbb{N}_K)_{\infty,v}^{-1}(u)(\chi)_{\infty,\overline{v}}(\overline{u})\overline{\Psi_{K,v}(\overline{u\pi_v^{-e}\delta^{-1}})}\\
&=(\mathbb{N}_K)_{\infty,v}^{-1}(\pi_v^{-e})(\chi)_{\infty,v}(\pi_v^{-e})\overline{\sum_{u\in (\mathcal{O}_{K,v}/\pi_v^e\mathcal{O}_{K,v})^{\times}}(\chi)_{\infty,v}^{-1}(u)\Psi_{K,v}(u\pi_v^{-e}\overline{\delta^{-1}})}\\
&=(\chi)_{\infty,v}(\pi_v^{-e})\frac{1}{\sum_{u\in (\mathcal{O}_{K,v}/\pi_v^e\mathcal{O}_{K,v})^{\times}}(\chi)_{\infty,v}^{-1}(u)\Psi_{K,v}(u\pi_v^{-e}\overline{\delta^{-1}})}=\frac{1}{G((\chi)_{\infty,v}^{-1},-\overline{\delta})}
\end{align*}
where in the penultimate equality, we have used the fact that 
$$\left|\sum_{u\in (\mathcal{O}_{K,v}/\pi_v^e\mathcal{O}_{K,v})^{\times}}(\chi)_{\infty,v}^{-1}(u)\Psi_{K,v}(-u\pi_v^{-e}\overline{\delta^{-1}})\right|^2 = (\mathbb{N}_K)_{\infty,v}(\pi_v^{-e}).$$

Now simply note that by definition of $W((\chi)_{\infty})$ and the above computation,
$$W((\chi)_{\infty},\delta) = \prod_{v|\mathfrak{F}}G((\chi)_{\infty,v}^{-1},\overline{\delta})\prod_{v|\mathfrak{F}}G((\chi)_{\infty,v}^{-1},\delta) =  \frac{G((\chi)_{\infty}^{-1},\overline{\delta})}{G((\check{\chi})_{\infty}^{-1},-\overline{\delta})}.$$
\end{proof}

\begin{proof}[Proof of Theorem \ref{Grosstheorem'}] Interpreting the statement as an assertion about $p$-adic measures, and proceeding as in \cite{Gross} using Lemma 1.1 of loc. cit., it suffices to prove the theorem at $s = 0$ and when $\chi$ is of finite order. 

First, suppose $\chi$ is unramified at each place dividing $D_K$ and so $\chi :\Gal(K(\mu_{fp^r})/\mathbb{Q}) \rightarrow \mathbb{C}_p^{\times}$ is an even Dirichlet character, where $fp^r\mathcal{O}_K = \mathfrak{f}(\chi_{/K})$, $(f,p) = 1$ and $fp^r>0$. 
Let $w_{fp^r}$ denote the number of roots of unity in $K^{\times}$ congruent to $1 \bmod fp^r$, and let $\mathcal{C}\ell(fp^r)$ denote the ray class group of $\mathcal{O}_K$ of modulus $fp^r\mathcal{O}_K$. 
Recall our fixed embeddings $i_{\infty}: \overline{\mathbb{Q}} \hookrightarrow \mathbb{C}$ and $i_p: \overline{\mathbb{Q}} \hookrightarrow \mathbb{C}_p$ compatible with the identification $i: \mathbb{C} \xrightarrow{\sim} \mathbb{C}_p$. For $a \in (\mathbb{Z}/fp^r)^{\times}$, let 
$$C_{fp^r}(a)_{\infty} = 1 - e^{\frac{-2\pi ia}{fp^r}} \hspace{.5cm} \text{and} \hspace{.5cm} C_{fp^r}(a) = i_{\infty}^{-1}(C_{fp^r}(a)_{\infty}) \hspace{.5cm} \text{and} \hspace{.5cm} C_{fp^r}(a)_p = i^{-1}(C_{fp^r}(a)_{\infty}).$$
Further let $E_{fp^r}(c)_{\infty} \in \overline{\mathbb{Q}}^{\times} \subset \mathbb{C}^{\times}$ be the ``elliptic units" of Robert \cite{Robert}, and let 
$$E_{fp^r}(c) = i_{\infty}^{-1}(E_{fp^r}(c)_{\infty}) \hspace{.5cm} \text{and} \hspace{.5cm} E_{fp^r}(c)_p = i(E_{fp^r}(c)_{\infty}).$$
Now for $a\in (\mathbb{Z}/fp^r)^{\times}$, choose $c \in \mathcal{C}\ell(fp^r)$ so that $\Nm_{K/\mathbb{Q}}(c) \equiv a \mod fp^r$, and let 
$$F_{fp^r}(a) = \Nm_{K(fp^r)/K(\mu_{fp^r})}(E_{fp^r}(c)).$$
Finally, for $a \in A = (\mathbb{Z}/fp^r)^{\times}/\{\pm 1\}$ let 
$$F^+(a) = F_{fp^r}(a)F_{fp^r}(-a) \hspace{.5cm} \text{and} \hspace{.5cm} C^+(a) = C_{fp^r}(a)C_{fp^r}(-a).$$

Note that in our situation, for any $s \in \mathbb{Z}_p$,
$$G((\chi_{/K}\langle\mathbb{N}_K\rangle^s)_{\infty}^{-1},-\sqrt{D_K}) = \prod_{v|\mathfrak{F}}G((\chi_{/K}\langle\mathbb{N}_K\rangle^s)_{\infty,v}^{-1},-\sqrt{D_K}) = f^{-s}\prod_{\ell|f}\mathfrak{g}(\chi_{\ell}^{-1})\chi_{\ell}^{-1}(\sqrt{D_K}).$$
Hence, using the special value formulas from \cite[10.4.9-10.4.12]{Katz1} (suitably modified with respect to the normalization of the $p$-adic $L$-function in \cite{Gross}) we have 
\begin{align*}&\frac{1}{\prod_{\ell|f}\chi_{\ell}^{-1}(-\sqrt{D_K})\mathfrak{g}(\chi_{\ell}^{-1})}L_p(\chi_{/K},0) = \frac{\prod_{\ell|f}\chi_{\ell}(-1)}{G((\chi_{/K})_{\infty}^{-1},-\sqrt{D_K})}L_p(\chi_{/K},0) \\
&= -\frac{1}{3fp^r w_{fp^r}}(1-\chi_{/K}(\overline{\mathfrak{p}}))\left(1-\frac{\chi_{/K}^{-1}(\mathfrak{p})}{p}\right)\frac{\mathfrak{g}(\chi^{-1})}{fp^r}\sum_{a\in A}\chi(a)\log_pF^+(a)_p,
\end{align*}
$$L_p(\chi\varepsilon_K\omega,0) = -(1-\chi\varepsilon_K(p))B_{1,\chi\varepsilon_K},$$
$$L_p(\chi^{-1},1) = -\left(1-\frac{\chi^{-1}(p)}{p}\right)\frac{\mathfrak{g}(\chi^{-1})}{fp^r}\sum_{a\in A}\chi(a)\log_p C^+(a)_p.$$
On the complex side, we have (by Kronecker's second limit formula, see \cite{Stark}),
$$L'((\chi_{/K})_{\infty},0) = -\frac{1}{6fp^r w_{fp^r}}\sum_{a\in A}(\chi)_{\infty}(a)\log F^+(a)_{\infty},$$
$$L((\chi)_{\infty}\varepsilon_K,0) = -B_{1,(\chi)_{\infty}\varepsilon_K},$$
$$L((\chi)_{\infty}^{-1},1) = -\frac{\mathfrak{g}((\chi)_{\infty}^{-1})}{fp^r}\sum_{a\in A}(\chi)_{\infty}(a)\log C^+(a)_{\infty},$$
and thus by the functional equation, we have
$$L'((\chi)_{\infty},0) = -\frac{1}{2}\sum_{a\in A}(\chi)_{\infty}(a)\log C^+(a)_{\infty}$$
where $\log: \mathbb{R}^{\times} \rightarrow \mathbb{R}$ is the map $x \mapsto \log|x|$.
Now we claim that
$$\frac{1}{\prod_{\ell|f}\chi_{\ell}^{-1}(-\sqrt{D_K})\mathfrak{g}(\chi_{\ell}^{-1})}L_p(\chi_{/K},0) = L_p(\chi\varepsilon_K\omega,0)L_p(\chi^{-1},1).$$
Note $(1-\chi_{/K}(\overline{\mathfrak{p}}))\left(1-\frac{\chi_{/K}^{-1}(\mathfrak{p})}{p}\right) = (1-\chi\varepsilon_K(p))\left(1-\frac{\chi^{-1}(p)}{p}\right)$ since $\varepsilon_K(p)=1$. If this quantity is 0 then the above identity of special values of $p$-adic $L$-functions is trivial, so assume this is not the case. By the above formulas, this identity is equivalent to
$$-3fp^rw_{fp^r}B_{1,\chi\varepsilon_K}\sum_{a\in A}\chi(a)\log_pC^+(a)_p = \sum_{a\in A}\chi(a)\log_pF^+(a)_p.$$
From the complex factorization
$$L((\chi_{/K})_{\infty},s) = L((\chi)_{\infty}\varepsilon_K,s)L((\chi)_{\infty},s)$$
we get, by taking the derivative at $s=0$,
$$L'((\chi_{/K})_{\infty},0) = L((\chi)_{\infty}\varepsilon_K,s)L'((\chi)_{\infty},0)$$
which is equivalent, by the above formulas, to
$$-3fp^rw_{fp^r}B_{1,(\chi)_{\infty}\varepsilon_K}\sum_{a\in A}(\chi)_{\infty}(a)\log C^+(a)_{\infty} = \sum_{a\in A}(\chi)_{\infty}(a)\log F^+(a)_{\infty}.$$
Now $C^+(a)_{\infty}$ and $F^+(a)_{\infty}$ are $p$-units in the field $M_{\infty} = \mathbb{Q}(\cos\frac{2\pi i}{fp^r})$. Let $E(M_{\infty})$ denote the group of all $p$-units viewed as a finitely generated subgroup of $\mathbb{R}^{\times}$. Then we claim that the identity
$$-3fp^rw_{fp^r}B_{1,(\chi)_{\infty}\varepsilon_K}\sum_{a\in A}(\chi)_{\infty}(a)\otimes C^+(a)_{\infty} = \sum_{a\in A}(\chi)_{\infty}(a)\otimes F^+(a)_{\infty}$$
holds in $\mathbb{C}\otimes_{\mathbb{Z}}E(M_{\infty})$. The representation of $A = \Gal(M_{\infty}/\mathbb{Q})$ on this complex vector space is isomorphic to the regular representation, and the elements
$$\sum_{a\in A}(\chi)_{\infty}(a)C^+(a)_{\infty}\hspace{.5cm} \text{and} \hspace{.5cm} \sum_{a\in A}(\chi)_{\infty}(a)F^+(a)_{\infty}$$
are both in the $(\chi)_{\infty}^{-1}$-eigenspace. Because this eigenspace is one-dimensional, the elements above differ by a complex scalar. Applying the $\mathbb{C}$-linear map
$$\mathbb{C}\otimes_{\mathbb{Z}}E(M_{\infty}) \xrightarrow{1\otimes \log} \mathbb{R}\otimes \mathbb{C} \xrightarrow{\text{mult}}\mathbb{C}$$
and considering the identity above concerning special values of complex $L$-functions, we identify this scalar as $-3fp^rw_{fp^r}B_{1,(\chi)_{\infty}\varepsilon_K}$. Thus our identity in $\mathbb{C}\otimes_{\mathbb{Z}} E(M_{\infty})$ holds. Now applying our identification $i: \mathbb{C} \xrightarrow{\sim} \mathbb{C}_p$, we obtain the identity
$$-3fp^rw_{fp^r}B_{1,\chi\varepsilon_K}\sum_{a\in A}\chi(a)\otimes C^+(a)_{\infty} = \sum_{a\in A}\chi(a)\otimes F^+(a)_{\infty}$$ 
in $\mathbb{C}_p\otimes E(M_p)$, where $M_p = i(M_{\infty})$. Finally, applying the homomorphism
$$\mathbb{C}_p\otimes \mathbb{C}_p \xrightarrow{1\otimes \log_p}\mathbb{C}_p\otimes\mathbb{C}_p \xrightarrow{\text{mult}}\mathbb{C}_p$$
we obtain our identity of special values of $p$-adic $L$-functions.
Now, to extend to general $\chi$ (including when $\chi$ is ramified at places dividing $D_K$), we use the functional equation of Theorem \ref{KatzpadicLfunction'} and Lemma \ref{Gausssumlemma}:
\begin{align*}&\frac{f^s}{\prod_{\ell|f}\chi_{\ell}^{-1}(-\sqrt{D_K})\mathfrak{g}(\chi_{\ell}^{-1})}L_p(\chi_{/K},s)\\
&=\frac{1}{G((\chi_{/K}\langle\mathbb{N}_K\rangle^s)_{\infty}^{-1},-\sqrt{D_K})}L_p(\chi_{/K},s) \\
&=  \frac{W((\chi_{/K}\langle\mathbb{N}_K\rangle^s)_{\infty},\sqrt{D_K})}{G((\chi_{/K}\langle\mathbb{N}_K\rangle^s)_{\infty}^{-1},-\sqrt{D_K})}L_p(\chi_{/K}^{-1}\omega_K^{-1},1-s) \\
&= \frac{1}{G((\check{\chi}_{/K}\langle\mathbb{N}_K\rangle^{-s})_{\infty}^{-1},\sqrt{D_K})}L_p((\chi\omega\varepsilon_K)_{/K}^{-1},1-s) \\
&= \frac{f^{1-s}}{\prod_{\ell|f}\chi_{\ell}(-\sqrt{D_K})\mathfrak{g}(\chi_{\ell})}L_p((\chi\omega\varepsilon_K)_{/K}^{-1},1-s)
\end{align*}
where $\check{\chi}_{/K}$ denotes the dual of $\chi_{/K}$ and we have used the fact that for $\ell|f$, $\omega$ and $\varepsilon_K$ are unramified at $\ell$, $\varepsilon_{K,\ell}(\ell) = 1$ since $\ell$ is split in $K$, and $(\varepsilon_K)_{/K} = 1$.
\end{proof}

\subsection{Eisenstein descent}\label{EisensteinDescent}
We now define the notion of cuspforms which have \emph{Eisenstein descent}. In this setting, we prove a congruence between the BDP $p$-adic $L$-function and the Katz $p$-adic $L$-function on the anticyclotomic line.

\begin{definition}\label{Eisensteindescent}Fix a global or local field $M$ containing $E_f$ and fix an integral ideal $\mathfrak{m}$ of $\mathcal{O}_M$. Suppose $(N_+,N_-,N_0)$ is a triple of pairwise coprime positive integers, where $N = N_+N_-N_0$, $N_+N_-$ is squarefree, $N_0$ is squarefull, and $\psi_1$ and $\psi_2$ are Dirichlet characters over $\mathbb{Q}$. We say that a normalized newform $f = \sum_{n=1}^{\infty}a_n q^n \in S_k(\Gamma_0(N),\varepsilon_f)$, where $k$ is a positive integer, has \emph{partial Eisenstein descent of type $(\psi_1,\psi_2,N_+,N_-,N_0)$ (over $M$) mod $\mathfrak{m}$} if $\psi_1\psi_2 = \varepsilon_f$ and we have
\begin{enumerate}
\item $a_{\ell} \equiv \psi_1(\ell) + \psi_2(\ell)\ell^{k-1} \mod \mathfrak{m}$ for $\ell\nmid N$;
\item $a_{\ell} \equiv \psi_1(\ell) \mod \mathfrak{m}$ for $\ell|N_+$;
\item $a_{\ell} \equiv \psi_2(\ell)\ell^{k-1} \mod \mathfrak{m}$ for $\ell|N_-$;
\item $a_{\ell} \equiv 0 \mod \mathfrak{m}$ for $\ell|N_0$.\\

If further, $f$ satisfies
\item
$$\delta_{\psi_1 = 1}\frac{B_{1,\psi_2}B_{k,\psi_1}}{k}\prod_{\ell|N_+}(1-\psi_1(\ell)\ell^{k-1})\prod_{\ell|N_-}(1-\psi_2(\ell))\prod_{\ell|N_0}(1-\psi_1(\ell)\ell^{k-1})(1-\psi_2(\ell)) \equiv 0 \mod \mathfrak{m}$$
\end{enumerate}
where 
$$\delta_{\psi = 1} := \begin{cases} 1 & \text{if}\; \psi = 1,\\
0 & \text{otherwise},\\
\end{cases}$$
then we say $f$ has \emph{(full) Eisenstein descent of type $(\psi_1,\psi_2,N_+,N_-,N_0)$ mod $\mathfrak{m}$}.
\end{definition}
\begin{remark}\label{Eisensteinremark}Recall, for $k \ge 2 $, the Eisenstein series $E_{k}^{\psi_1,\psi_2,(N)} \in M_k^*(\Gamma_0(N),\psi_1\psi_2)$. (When either $k > 2$, or $\psi_1 \neq 1$ or $\psi_2 \neq 1$, then in fact $E_k^{\psi_1,\psi_2,(N)} \in M_k(\Gamma_0(N),\psi_1\psi_2)$.) It has $q$-expansion
\begin{align*}&E_{k}^{\psi_1,\psi_2,(N)}(q) \\
&:= -\delta_{\psi_1 = 1}L^{(N_-N_0)}(\psi_2,0)L^{(N_+N_0)}(\psi_1,1-k) + \sum_{n=1}^{\infty}\sigma_{k-1}^{\psi_1,\psi_2,(N)}(n)q^n \\
&=-\delta_{\psi_1 = 1}\frac{B_{k,\psi_1}}{2k}\\
&\cdot\prod_{\ell|N_+}(1-\psi_1(\ell)\ell^{k-1})\prod_{\ell|N_-}(1-\psi_2(\ell))\prod_{\ell|N_0}(1-\psi_1(\ell)\ell^{k-1})(1-\psi_2(\ell)) + \sum_{n=1}^{\infty}\sigma_{k-1}^{\psi_1,\psi_2,(N)}(n)q^n
\end{align*}
where for a Dirichlet character $\psi$, $L^{(N)}(\psi,s)$ denotes its $L$-function with Euler factors at primes $\ell|N$ removed, and
$$\sigma_{k-1}^{\psi_1,\psi_2,(N)}(n) := \sum_{0<d|n, (d,N_+) = 1, (n/d,N_-) = 1, (n,N_0) = 1}\psi_1(n/d)\psi_2(d)d^{k-1}.$$
Here, in keeping with our conventions, $\psi(m) = 0$ if $(m,\mathfrak{f}(\psi))\neq 1$. Then for $f \in S_k(\Gamma_0(N),\varepsilon_f)$ to have partial Eisenstein descent of type $(\psi_1,\psi_2,N_+,N_-,N_0)$ at $\lambda|p$ is equivalent to
$$\theta^jf(q) \equiv \theta^jE_{k}^{\psi_1,\psi_2,(N)}(q) \mod \mathfrak{m}$$
for all $j \ge 1$. When $f$ has full Eisenstein descent, this congruence holds for $j \ge 0$. 
\end{remark}
\begin{remark}\label{constantterm0}Suppose we are given $f \in S_k(\Gamma_0(N),\varepsilon_f)$ with partial Eisenstein descent of type $(\psi_1,\psi_2,N_+,N_-,N_0)$ over $M$ mod $\mathfrak{m}$, in the sense of Definition \ref{Eisensteindescent}. In many situations, $(5)$ is forced to hold \emph{a priori} so that $f$ automatically has full Eisenstein descent. Suppose $\mathfrak{m}\neq\mathcal{O}_M$ (for otherwise, the conditions of Definition \ref{Eisensteindescent} are vacuous).

If $\psi_1$ is non-trivial, then $\delta_{\psi = 1} = 0$ and $(5)$ holds. Now suppose that $\psi_1 = 1$. If $k$ is odd, then $B_{k,\psi_1} = B_k = 0$, again forcing $(5)$ to hold. Suppose $k$ is even. Let $\lambda|\mathfrak{m}$ be a prime ideal of residual characteristic not equal to $2$. By parts (3) and (1) of Theorem \ref{classify}, we have $\psi_1\psi_2 = \varepsilon_f$; in particular, we have $\psi_2(-1) = \varepsilon_f(-1) = (-1)^k = 1$. Hence $\psi_2$ is even, and so $B_{1,\psi_2} = 0$ unless $\psi_2 = 1$. Now further suppose that $\psi_2 = 1$. Then $(5)$ is still forced to hold unless $N_-N_0 = 1$. Now suppose $N_-N_0 = 1$. Let $\lambda|\mathfrak{m}$ be a prime ideal of residual characteristic $p$. Then conditions (1)-(4) still imply
$$\theta f(q) \equiv \theta E_{k}^{\psi_1,\psi_2,(N)}(q) \mod \lambda.$$
Suppose $p > k+1$, so that by \cite[Corollary 3, p. 326]{Serre2}, $\theta$ is injective on mod $\lambda$ modular forms. Then the above congruence implies
$$f(q) \equiv E_{k}^{\psi_1,\psi_2,(N)}(q) \mod \lambda.$$
Hence, if $\mathfrak{m}$ has order $0$ or $1$ at every prime of $\mathcal{O}_M$ and if for every prime $\lambda|\mathfrak{m}$, its residual characteristic $p$ satisfies $p > k+1$, then (5) is forced to hold. (See \cite[Theorem 4.1]{BillereyMenares} where a similar argument is given.) In particular, we have
$$\frac{B_k}{2k}\prod_{\ell|N_+}(1-\ell^{k-1}) \equiv 0 \mod \mathfrak{m}.$$
In particular, when $k = 2$ we have
$$\frac{1}{24}\prod_{\ell|N_+}(1-\ell) \equiv 0 \mod \mathfrak{m},$$
and thus there exists at least one $\ell|N_+ = N$ such that $\ell \equiv 1 \mod \mathfrak{m}$, i.e. $\ell \equiv 1 \mod \mathfrak{m}\cap\mathbb{Z}$. 
\end{remark}

Suppose we are given a normalized eigenform $f(q) = \sum_{n=0}^{\infty}a_nq^n\in M_k(\Gamma_0(N),\varepsilon_f)$. Again let $M$ be a number field containing $E_f$. Let $k_{\lambda}$ denote the residue field of $\mathcal{O}_M$ at a prime $\lambda|p$. By a construction of Deligne (see \cite{DeligneGalois}) we can attach a unique semisimple $p$-adic Galois representation $\rho_f: \Gal(\overline{\mathbb{Q}}/\mathbb{Q}) \rightarrow GL_2(M_{\lambda})$ unramified outside $pN$ and such that $\rho_f(\Frob_{\ell})$ has characteristic polynomial
$$T^2 - a_{\ell}T + \varepsilon_f(\ell)\ell^{k-1}$$
for all $\ell\nmid pN$. Taking a $\Gal(\overline{\mathbb{Q}}/\mathbb{Q})$-stable lattice and the reduction mod $\lambda$, we get a representation $\bar{\rho}: \Gal(\overline{\mathbb{Q}}/\mathbb{Q}) \rightarrow GL_2(k_{\lambda})$ (whose semisimplification is independent of the choice of lattice). By a Theorem 3 of \cite{Atkin-Lehner}, we have $a_{\ell} = \pm \ell^{k/2-1}$ for $\ell||N$, and $a_{\ell} = 0$ for $\ell^2|N$. 

We have the following characterization of partial Eisenstein descent mod $\lambda$. See \cite[Theorem 4.1]{BillereyMenares} for a similar result to Part (2). The elliptic curves case of Part (1) was essentially done by Serre in \cite{Serre}.
\begin{theorem}\label{classify}Suppose $f\in S_k(\Gamma_0(N),\varepsilon_f)$ is a normalized newform, and let $\bar{\rho}_f: \Gal(\overline{\mathbb{Q}}/\mathbb{Q}) \rightarrow GL_2(k_{\lambda})$ be the mod $\lambda$ reduction of the associated semisimple Galois representation. Then the following hold:
\begin{enumerate}
\item Suppose $\bar{\rho}_f$ is reducible. Then $\bar{\rho}_f \cong k_{\lambda}(\tilde{\psi}_1) \oplus k_{\lambda}(\tilde{\psi}_2\omega^{k-1})$ where $\tilde{\psi}_i = \psi_i \mod \lambda$ for $i = 1, 2$ for some Dirichlet characters $\psi_1$ and $\psi_2$ over $\mathbb{Q}$ with $\psi_1\psi_2 = \varepsilon_f$.
\item Suppose $\bar{\rho}_f$ is reducible. For all $\ell||N$, either $a_{\ell} \equiv \psi_1(\ell) \mod \lambda$ and $a_{\ell} \equiv \psi_1^{-1}(\ell)\ell^{k-2} \mod \lambda$ and $\ell^{k-2} \equiv \psi_1^2(\ell) \mod \lambda$, or $a_{\ell} \equiv \psi_2(\ell)\ell^{k-1} \mod \lambda$ and $a_{\ell} \equiv \psi_2^{-1}(\ell)\ell^{-1} \mod \lambda$ and $\ell^k \equiv \psi_2^{-2}(\ell) \mod \lambda$.
\item  Let $N_+$ denote any product of all primes $\ell||N$ satisfying $a_{\ell} \equiv \psi_1(\ell) \mod \lambda$ and $N_-$ any product of $\ell||N$ satisfying $a_{\ell} \equiv \psi_2(\ell)\ell^{k-1} \mod \lambda$, such that $N_+N_-$ is the squarefree part of $N$ (so that, in particular, $(N_+,N_-) = 1$). Let $N_0$ the squarefull part of $N$. Then $\bar{\rho}_f$ is reducible if and only if $f$ has partial Eisenstein descent of type $(\psi_1,\psi_2,N_+,N_-,N_0)$ over $M_{\lambda}$ mod $\lambda$.%, where $\frac{\mathfrak{f}(\psi_2)}{(\mathfrak{f}(\psi_1),\mathfrak{f}(\psi_2))}|N_+$, $\frac{\mathfrak{f}(\psi_1)}{(\mathfrak{f}(\psi_1),\mathfrak{f}(\psi_2))}|N_-$, and $(\mathfrak{f}(\psi_1),\mathfrak{f}(\psi_2))^2|N_0$.
\end{enumerate}
\end{theorem}
\begin{proof}(1): Since $\bar{\rho}_f$ is reducible and semisimple, we can write $\bar{\rho}_f = k_{\lambda}(\chi_1)\oplus k_{\lambda}(\chi_2)$ where $\chi_1, \chi_2: \Gal(\overline{\mathbb{Q}}/\mathbb{Q}) \rightarrow k_{\lambda}^{\times}$. Hence $\chi_1\chi_2 = \det(\bar{\rho}_f) = \varepsilon_f\omega^{k-1}$, and so for our statement it suffices to show one of $\{\chi_1,\chi_2\}$ is unramified outside the squarefull part of $N$ (since $\mathfrak{f}(\varepsilon_f)$ divides the squarefull part of $N$). Since $\bar{\rho}_f$ is unramified outside $pN$, clearly both $\chi_1, \chi_2$ are unramified outside $pN$. For $\ell||N$, the local representation $\bar{\rho}_{f,\ell}$ (i.e., restriction of $\bar{\rho}_{f,\ell}$ to the decomposition group $\Gal(\overline{\mathbb{Q}}_{\ell}/\mathbb{Q}_{\ell})$) has conductor $\ell$. Thus the corresponding automorphic representation $\pi_{\ell}$ of $\rho_{\ell}$ has conductor $\ell$, and so by the classification of admissible representations of $GL_2$ over local fields (see \cite{Gelbart} p. 73), at most one of $\chi_1$ and $\chi_2$  is ramified at $\ell$. Thus $\chi_1,\chi_2$ are unramified inside the squarefree part of $N$. Finally, since $\bar{\rho}_{f,p}$ is reducible, by a theorem of Deligne (see \cite{Gross2}), $a_p \not\equiv 0 \mod \lambda$ and $\bar{\rho}_{f,p}$ is of the form
$$\left( \begin{array}{ccc}
\omega^{k-1}\mu_p(\varepsilon_f(p)/a_p) & * \\
0 & \mu_p(a_p) \end{array} \right)$$
where $\mu_p(\alpha)$ is the unramified character of $\Gal(\overline{\mathbb{Q}}_p/\mathbb{Q}_p)$ taking $\Frob_p$ to $\alpha$. Hence exactly one of $\chi_1, \chi_2$ is ramified at $p$. 

Putting this all together, we see that $\{\chi_1,\chi_2\} = \{\psi_1 \mod \lambda,\psi_2\omega^{k-1} \mod \lambda\}$ for some Dirichlet characters $\psi_1$ and $\psi_2$ with $\psi_1\psi_2 = \varepsilon_f$, and so we are done.

(2): It is a theorem of Langlands (see \cite[Proposition 2.8]{LoefflerWeinstein}) that for $\ell||N$, $\bar{\rho}_{f,\ell}$ is of the form
$$\left( \begin{array}{ccc}
\omega_{\ell}^{k/2}\mu_{\ell}(\varepsilon_f(\ell)\ell^{k/2-1}/a_{\ell}) & * \\
0 & \omega_{\ell}^{k/2-1}\mu_{\ell}(a_{\ell}/\ell^{k/2-1}) \end{array} \right)$$
where $\omega_{\ell}$ is the localization of $\omega$ to $\ell$. Thus by (1), we have 
$$\{\psi_{1,\ell} \mod \lambda,\psi_{2,\ell}\omega_{\ell}^{k-1} \mod \lambda\} = \{\omega_{\ell}^{k/2}\mu_{\ell}(\varepsilon_f(\ell)\ell^{k/2-1}/a_{\ell}), \omega_{\ell}^{k/2}\mu_{\ell}(a_{\ell}/\ell^{k/2-1})\}.$$
Note that $\mu_{\ell}(\ell^{k/2-1}/a_{\ell})$ is a quadratic character.

Suppose first that $\psi_{1,\ell} \equiv \omega_{\ell}^{k/2}\mu_{\ell}(\varepsilon_f(\ell)\ell^{k/2-1}/a_{\ell}) \mod \lambda$ and $\psi_{2,\ell}\omega_{\ell}^{k-1} \equiv \omega_{\ell}^{k/2-1}\mu_{\ell}(\ell^{k/2-1}/a_{\ell}) \mod \lambda$. Plugging in $\Frob_{\ell}$ to both congruences, the first congruence implies both $a_{\ell} \equiv \psi_{2,\ell}(\ell)\ell^{k-1} = \psi_2(\ell)\ell^{k-1}\mod \lambda$, and the second congruence implies $a_{\ell} \equiv \psi_{2,\ell}^{-1}(\ell)\ell^{-1} = \psi_2^{-1}(\ell)\ell^{-1} \mod \lambda$.

Suppose next that $\psi_{1,\ell} \equiv \omega_{\ell}^{k/2-1}\mu_{\ell}(a_{\ell}/\ell^{k/2-1}) \mod \lambda$ and $\psi_{2,\ell}\omega_{\ell}^{k-1} \equiv \omega_{\ell}^{k/2}\mu_{\ell}(\varepsilon_f(\ell)\ell^{k/2-1}/a_{\ell}) \mod \lambda$. Plugging in $\Frob_{\ell}$ to both congruences, the first congruence implies both $a_{\ell}\equiv \psi_{1,\ell}(\ell) = \psi_1(\ell) \mod \lambda$ and $a_{\ell}\equiv \psi_{1,\ell}^{-1}(\ell)\ell^{k-2} = \psi_1^{-1}(\ell)\ell^{k-2} \mod \lambda$ (we get both congruences since $\mu_{\ell}(a_{\ell}/\ell^{k/2-1}) = \mu_{\ell}^{-1}(a_{\ell}/\ell^{k/2-1})$), and the second congruence gives no new congruences. 

(3): For $\ell\nmid pN$, we have $a_{\ell} = \trace(\bar{\rho}_f)(\Frob_{\ell}) \equiv \psi_1(\ell) + \psi_2(\ell)\ell^{k-1}\mod \lambda$. For $\ell|N_+$, we have $a_{\ell} \equiv \psi_1(\ell) \mod \lambda$, for $\ell|N_-$ we have $a_{\ell}\equiv \psi_2(\ell)\ell^{k-1} \mod \lambda$, and for $\ell|N_0$ we have $a_{\ell} \equiv 0 \mod \lambda$. Finally, by the Cebotarev density theorem and continuity of $\trace(\bar{\rho}_f)$, we have $a_p = \trace(\bar{\rho}_f(\Frob_p)) \equiv \psi_1(p) + \psi_2(p)p^{k-1} \equiv \psi_1(p) \mod \lambda$. Hence $f$ has partial Eisenstein descent of type $(\psi_1,\psi_2,N_+,N_-,N_0)$ at $\lambda|p$. 

If $f$ has partial Eisenstein descent of type $(\psi_1,\psi_2,N_+,N_-,N_0)$ at $\lambda|p$, then for $\ell\nmid N$ we have $\text{trace}(\bar{\rho}_f )(\Frob_{\ell})\equiv \psi_1(\ell) + \psi_2(\ell)\ell^{k-1} \mod \lambda$. Hence if $\bar{\rho} := k_{\lambda}(\psi_1)\oplus k_{\lambda}(\psi_2\omega^{k-1})$, we have $\text{trace}(\bar{\rho}_f)(\Frob_{\ell}) \equiv \text{trace}(\bar{\rho})(\Frob_{\ell})$ for all $\ell\nmid N$. Thus by the Cebotarev density theorem and continuity of trace, $\text{trace}(\bar{\rho}_f)(g) \equiv \text{trace}(\bar{\rho})(g)$ for all $g \in \Gal(\overline{\mathbb{Q}}/\mathbb{Q})$. Hence, by the Brauer-Nesbitt theorem, $\bar{\rho}_f \cong \bar{\rho}$. 
\end{proof}

For $k = 2$ and $M = E_f = \mathbb{Q}$, we note the following corollary.
\begin{theorem}\label{congruence'} Suppose $E/\mathbb{Q}$ is an elliptic curve and $p$ is a prime of $E$ such that $E[p]$ is a reducible mod $p$ Galois representation. Then the associated normalized newform $f_E\in S_2(\Gamma_0(N))$ has partial Eisenstein descent over $\mathbb{Q}_p$ mod $p\mathbb{Z}_p$.
\end{theorem}

\section{Proof of the Main Theorem}\label{ProofofMainTheorem}
In this section, we prove the Main Theorem. First, we examine complex $L$-values arising from twisted traces of Eisenstein series evaluated at CM points, analogous to such traces for normalized newforms interpolated by the Bertolini-Darmon-Prasanna $p$-adic $L$-function (see Theorem \ref{BDPinterpolationproperty}). Parts of the calculation in Section \ref{charactertwist} are implicit in \cite{HidaTilouine}, but for our purposes which require explicit identities, and for the sake of completeness, we include the full calculation here.

\subsection{Twisted traces of Eisenstein series over CM points}\label{charactertwist}
Let $k \ge 2$ be an integer, and let $\psi_1,\psi_2$ be two Dirichlet characters over $\mathbb{Q}$ of conductors $u$ and $t$, respectively (here $u,t$ are not necessarily coprime), such that $ut = N'$ and $(\psi_1\psi_2)(-1) = (-1)^k$. Recall our Eisenstein Series (see Remark \ref{Eisensteinremark})
$$E_{k}^{\psi_1,\psi_2}(\tau) := \delta_{\psi_1=1}\frac{L(\psi_1,1-k)}{2} + \sum_{n=1}^{\infty}\sigma_{k-1}^{\psi_1,\psi_2}(n)q^n$$
where $q = e^{2\pi i\tau}$ and
$$\sigma_{k-1}^{\psi_1,\psi_2}(n) = \sum_{0<d|n}\psi_1(n/d)\psi_2(d)d^{k-1}.$$ 
Note that the nebentypus of $E_k^{\psi_1,\psi_2}$ is $\varepsilon_{E_k^{\psi_1,\psi_2}} = \psi_1\psi_2$. 

Recalling the Maass-Shimura derivative $\partial$ defined in Section \ref{algebraicmodularforms}, one checks by direct computation that
\begin{align*}&\partial^jE_k^{\psi_1,\psi_2}(\tau) = \frac{t^k\Gamma(k+j)}{2(2\pi i)^{k+j}\mathfrak{g}(\psi_2^{-1})}\sum_{c=0}^{u-1}\sum_{d=0}^{t-1}\sum_{e=0}^{u-1}\sum_{(m,n) \equiv (ct,d+et)\: (N')}\frac{\psi_1(c)\psi_2^{-1}(d)}{(m\tau+n)^{k+2j}}\left(\frac{|m\tau+n|^2}{\tau-\overline{\tau}}\right)^j,
\end{align*}
where $\sum_{(m,n) \equiv (ct,d+et) \: (N')}$ denotes that the sum is taken over $(m,n) \in \mathbb{Z}^2\setminus\{0\}$ such that $(m,n) \equiv (ct,d+et) \mod N'$. 

Recall that under our Assumptions \ref{assumptions}, $D_K$ is taken to be odd. (The calculations for the even case are entirely analogous to those of the odd case, but we do not explicitly write these out here.) Given an integral primitive ideal (i.e. having no rational integral divisors other than $\pm 1$) $\mathfrak{a}$ of $\mathcal{O}_K$, we can write 
$$\mathfrak{a} = \mathbb{Z}\frac{b+\sqrt{D_K}}{2} + \mathbb{Z}a$$
where $a = |\Nm_{K/\mathbb{Q}}(\mathfrak{a})|$ and $b^2 - 4ac = D_K$; the triple $(a,-b,c)$ determines the primitive positive definite binary quadratic form associated with the ideal class of $\mathfrak{a}$. We put
$$\tau_{\mathfrak{a}} := \frac{b+\sqrt{D_K}}{2a}$$
so that $\tau_{\mathfrak{a}} \in \mathcal{H}^+$ is the root of the dehomogenized quadratic form  $a\tau^2 - b\tau + c=0$ with positive imaginary part. We call $\tau_{\mathfrak{a}}$ a \emph{$CM$ point}. Note that $\langle \tau_{\mathfrak{a}},1\rangle$ generates $\overline{\mathfrak{a}}^{-1}$. (Here for $\tau \in \mathcal{H}^+$, $\langle \tau, 1 \rangle = \mathbb{Z}\tau + \mathbb{Z}$.)

Suppose that $K$ satisifies the Heegner hypothesis with respect to $N'$. Fix an ideal $\mathfrak{N}'$ such that $\mathcal{O}_K/\mathfrak{N}' = \mathbb{Z}/N'$, and write
$$\mathfrak{N}' = \mathbb{Z}\frac{b_{N'}+\sqrt{D_K}}{2} + \mathbb{Z}N' \hspace{1.5cm}\text{and}\hspace{1.5cm} \mathfrak{aN'} = \mathbb{Z}\frac{b_{aN'}+\sqrt{D_K}}{2} + \mathbb{Z}aN'$$
where $b_{aN'}^2 - 4aN'c = D_K$ for some $c\in \mathbb{Z}$. Write $\mathfrak{N}' = \mathfrak{ut}$ where $\mathcal{O}_K/\mathfrak{u} = \mathbb{Z}/u$ and $\mathcal{O}_K/\mathfrak{t}= \mathbb{Z}/t$. 

Suppose we are given a Dirichlet character $\phi : (\mathbb{Z}/N')^{\times} \rightarrow \mathbb{C}^{\times}$. By the identification $\mathcal{O}_K/\mathfrak{N}' = \mathbb{Z}/N'$, we can view $\phi$ as a character $\phi : \mathbb{A}_K^{\times,\mathfrak{N}'} \twoheadrightarrow \prod_{v|\mathfrak{N}'}(\mathcal{O}_{K,v}/\mathfrak{N}'\mathcal{O}_{K,v})^{\times} \cong (\mathcal{O}_K/\mathfrak{N}')^{\times} \rightarrow \mathbb{C}^{\times}$, where $\mathbb{A}_K^{\times,\mathfrak{N}'}$ denotes the id\`{e}les prime to $\mathfrak{N}'$; when we view $\phi$ in this way, we will write $\phi(x \mod \mathfrak{N}')$ for its value at $x \in \mathbb{A}_K^{\times,\mathfrak{N}'}$. Given a Hecke character $\chi$ over $K$ of finite type $(\mathfrak{N}',\phi)$ and infinity type $(j_1,j_2)$, recall that the associated Grossencharacter on ideals $\mathfrak{a}$ prime to $\mathfrak{N}'$ is given by
$$\chi(\mathfrak{a}) = \chi(x)\phi(x \mod \mathfrak{N}')x_{\infty}^{j_1}\overline{x}_{\infty}^{j_2}$$
where $x \in \mathbb{A}_K^{\times,\mathfrak{N}'}$ is such that $\ord_v(x) = \ord_v(\mathfrak{a})$ for all finite places $v$. We consider the twisted trace
$$\sum_{[\mathfrak{a}]\in \mathcal{C}\ell(\mathcal{O}_K)}(\chi_j)^{-1}(\overline{\mathfrak{aN'}})\partial^jE_k^{\psi_1,\psi_2}(\tau_{\mathfrak{aN'}})$$
where $\mathfrak{a}$ ranges over a set of primitive integral ideal representatives of $\mathcal{C}\ell(\mathcal{O}_K)$ chosen to be prime to $N'$, and $\chi$ is of infinity type $(k+j,-j)$ and of finite type $(\mathfrak{N}',\psi_1\psi_2)$ (so that the above summands depend only on the ideal classes of the $\mathfrak{\overline{aN'}}$).

Suppose $\alpha = m\frac{b_{aN'}+\sqrt{D_K}}{2aut} + n \in \overline{\mathfrak{au}}^{-1}$ with $(m,n) \equiv (ct,d+et) \mod N'$. Then $au\alpha = m\frac{b_{aN'}+\sqrt{D_K}}{2t} + aun \in \mathfrak{au} \subset \mathcal{O}_K$ is mapped to $b_{aN'}c \in \mathbb{Z}/u$ under the identification $\mathcal{O}_K/\overline{\mathfrak{u}} = \mathbb{Z}/u$. One sees this by writing $m = ct + qN'$ for some $q \in \mathbb{Z}$ and $au\alpha = \frac{mb_{aN'}}{t} - m\frac{b_{aN'}-\sqrt{D_K}}{2} + aun = b_{aN'}c - m\frac{b_{aN'}-\sqrt{D_K}}{2} + qu \equiv b_{aN'}c \mod \overline{\mathfrak{u}}$. We thus have $\psi_1((au\alpha) \mod \overline{\mathfrak{u}}) = \psi_1(b_{aN'}c)$. Since $b_{aN'}+\sqrt{D_K} \in \mathfrak{aN'}$, in particular we have $b_{aN'} \equiv -\sqrt{D_K} \mod \mathfrak{u}$, meaning $\psi_1(b_{aN'}) = \psi_1(-\sqrt{D_K} \mod \mathfrak{u})$; henceforth we will write $\psi_1(-\sqrt{D_K}) = \psi_1(-\sqrt{D_K} \mod \mathfrak{u})$ for simplicity. On the other hand, note that $\alpha$ is mapped to $d \in \mathbb{Z}/t$ under the isomorphism $\overline{\mathfrak{au}}^{-1}/\overline{\mathfrak{au}}^{-1}\mathfrak{t} \cong \mathcal{O}_K/\mathfrak{t} = \mathbb{Z}/t$. Thus, we have $\psi_2(\overline{\mathfrak{au}}(\alpha) \mod \mathfrak{t}) = \psi_2(d)$ (since $\overline{\mathfrak{au}}$ is prime to $\mathfrak{t}$). In all, we can write
\begin{align*}\frac{\psi_1(c)\overline{\psi_2}(d)}{(m\tau_{\mathfrak{aN'}}+n)^{k+2j}}&\left(\frac{|m\tau_{\mathfrak{aN'}}+n|^2}{\tau_{\mathfrak{aN'}}-\overline{\tau_{\mathfrak{aN'}}}}\right)^j \\
&= \frac{\psi_1^{-1}(-\sqrt{D_K})\psi_1((au\alpha) \mod \overline{\mathfrak{u}})\psi_2^{-1}(\overline{\mathfrak{au}}(\alpha) \mod \mathfrak{t})}{\alpha^{k+2j}}\left(\frac{\Nm_{K/\mathbb{Q}}(\overline{\mathfrak{a}\mathfrak{N}'}(\alpha))}{\sqrt{D_K}}\right)^j.\end{align*}

Suppose first that $k > 2$. Then we can rewrite each summand in the twisted trace as
\begin{align*}&(\chi_j)^{-1}(\overline{\mathfrak{aN'}})\partial^jE_k^{\psi_1,\psi_2}(\tau_{\mathfrak{aN'}}) = \frac{t^k\Gamma(k+j)\psi_1^{-1}(-\sqrt{D_K})}{2(2\pi i)^{k+j}\mathfrak{g}(\psi_2^{-1})}(\chi_j)^{-1}(\overline{\mathfrak{aN'}})\\
&\hspace{1.2cm}\cdot \sum_{\alpha\in (\overline{\mathfrak{aN'}})^{-1}, (\overline{\mathfrak{aN'}}(\alpha),\overline{\mathfrak{u}}\mathfrak{N}') = 1}\frac{\psi_1((au\alpha) \mod \overline{\mathfrak{u}})\psi_2^{-1}(\overline{\mathfrak{au}}(\alpha)\mod \mathfrak{t})}{\alpha^{k+2j}}\left(\frac{|\Nm_{K/\mathbb{Q}}(\overline{\mathfrak{aN'}}(\alpha))|}{\sqrt{D_K}}\right)^j\\
&= \frac{t^k\Gamma(k+j)\psi_1^{-1}(-\sqrt{D_K})}{2(2\pi i)^{k+j}\mathfrak{g}(\psi_2^{-1})\sqrt{D_K}^j}\frac{\chi^{-1}(\overline{\mathfrak{t}})(\chi_{-\frac{k}{2}})^{-1}(\overline{\mathfrak{au}})}{|\Nm_{K/\mathbb{Q}}(\overline{\mathfrak{au}})|^{k/2}}\\
&\hspace{3.8cm}\cdot\sum_{\alpha\in (\overline{\mathfrak{au}})^{-1}, (\overline{\mathfrak{au}}(\alpha),\overline{\mathfrak{u}}\mathfrak{N}') = 1}\left(\frac{\overline{\alpha}}{\alpha}\right)^{k/2+j}\frac{\psi_1((au\alpha) \mod \overline{\mathfrak{u}})\psi_2^{-1}(\overline{\mathfrak{au}}(\alpha) \mod \mathfrak{t})}{|\Nm_{K/\mathbb{Q}}(\alpha)|^{k/2}}\\
&= \frac{t^k\Gamma(k+j)\psi_1^{-1}(-\sqrt{D_K})\chi^{-1}(\overline{\mathfrak{t}})}{2(2\pi i)^{k+j}\mathfrak{g}(\psi_2^{-1})\sqrt{D_K}^j}\sum_{\alpha\in (\overline{\mathfrak{au}})^{-1}, (\overline{\mathfrak{au}}(\alpha),\overline{\mathfrak{u}}\mathfrak{N}') = 1}(\chi_{-\frac{k}{2}})^{-1}(\overline{\mathfrak{au}}(\alpha))\\
&\cdot\psi_1(\overline{\mathfrak{au}}(\alpha)\mod \mathfrak{u})\psi_2(\overline{\mathfrak{au}}(\alpha) \mod \mathfrak{t})\frac{\psi_1((au\alpha) \mod \overline{\mathfrak{u}})\psi_2^{-1}(\overline{\mathfrak{au}}(\alpha) \mod \mathfrak{t})}{|\Nm_{K/\mathbb{Q}}(\overline{\mathfrak{au}}(\alpha))|^{k/2}}\\
&= \frac{t^k\Gamma(k+j)\psi_1^{-1}(-\sqrt{D_K})\chi^{-1}(\overline{\mathfrak{t}})}{(2\pi i)^{k+j}\mathfrak{g}(\psi_2^{-1})\sqrt{D_K}^j}\\
&\hspace{4.6cm}\cdot\sum_{\alpha\in (\overline{\mathfrak{au}})^{-1}, (\overline{\mathfrak{au}}(\alpha),\overline{\mathfrak{u}}\mathfrak{N}') = 1}(\chi_{-\frac{k}{2}})^{-1}(\overline{\mathfrak{au}}(\alpha))\frac{\psi_1((au\Nm_{K/\mathbb{Q}}(\alpha)) \mod \mathfrak{u})}{|\Nm_{K/\mathbb{Q}}(\overline{\mathfrak{au}}(\alpha))|^{k/2}}\\
&= \frac{|\mathcal{O}_K^{\times}|}{2}\frac{t^k\Gamma(k+j)\psi_1^{-1}(-\sqrt{D_K})\chi^{-1}(\overline{\mathfrak{t}})}{(2\pi i)^{k+j}\mathfrak{g}(\psi_2^{-1})\sqrt{D_K}^j}\sum_{\mathfrak{b}\subset \mathcal{O}_K, [\mathfrak{b}] = [\overline{\mathfrak{au}}] \in \mathcal{C}\ell(\mathcal{O}_K), (\mathfrak{b},\overline{\mathfrak{u}}\mathfrak{N}') = 1}\frac{(\psi_{1/K}(\chi_{-\frac{k}{2}})^{-1})(\mathfrak{b})}{|\Nm_{K/\mathbb{Q}}(\mathfrak{b})|^\frac{k}{2}}.
\end{align*}
The penultimate equality is justified by the fact that 
$$\psi_1((au\alpha) \mod \overline{\mathfrak{u}}) = \psi_1(\overline{\mathfrak{au}}(\alpha) \mod \overline{\mathfrak{u}}) =  \psi_1(\mathfrak{au}(\overline{\alpha}) \mod \mathfrak{u})$$ 
(the first equality here follows since $\mathfrak{au}$ is prime to $\overline{\mathfrak{u}}$). The factor of $|\mathcal{O}_K^{\times}|$ appears because any integral ideal $\mathfrak{b}\subset\mathcal{O}_K$ can be written as $\overline{\mathfrak{au}}(\alpha)$ for some ideal class representative $\mathfrak{a}$ and some element $\alpha\in (\overline{\mathfrak{au}})^{-1}$ determined uniquely up to an element of $\mathcal{O}_K^{\times} = \{\pm 1\}$, and since $(\psi_1\psi_2)(-1) = (-1)^k = (-1)^{k+2j}$, we see that each summand of the second line has the same value for $\alpha$ or $-\alpha$. Now since $|\mathcal{O}_K^{\times}| = 2$, we finally have, after summing both sides of the above equality over all our ideal class representatives $\mathfrak{a}$,
$$\sum_{[\mathfrak{a}]\in \mathcal{C}\ell(\mathcal{O}_K)}(\chi_j)^{-1}(\overline{\mathfrak{aN'}})\partial^jE_k^{\psi_1,\psi_2}(\tau_{\mathfrak{aN'}}) = \frac{t^k\Gamma(k+j)\psi_1^{-1}(-\sqrt{D_K})\chi^{-1}(\overline{\mathfrak{t}})}{(2\pi i)^{k+j}\mathfrak{g}(\psi_2^{-1})\sqrt{D_K}^j}L(\psi_{1/K}(\chi_{-\frac{k}{2}})^{-1},\frac{k}{2}).$$

For $k = 2$, note that by the same calculation as above, we have for all $s\in \mathbb{C}$ with $\Re(s) > 0$ that the series
\begin{align*}&\sum_{[\mathfrak{a}]\in \mathcal{C}\ell(\mathcal{O}_K)}(\chi_j)^{-1}(\overline{\mathfrak{aN'}})\sum_{\alpha\in (\overline{\mathfrak{au}})^{-1}, (\overline{\mathfrak{au}}(\alpha),\overline{\mathfrak{u}}\mathfrak{N}') = 1}\\
&\hspace{2.5cm}\frac{\psi_1^{-1}(-\sqrt{D_K})\psi_1(\overline{\mathfrak{a}}(au\alpha) \mod \overline{\mathfrak{u}})\psi_2^{-1}(\overline{\mathfrak{au}}(\alpha) \mod \mathfrak{t})}{\alpha^{2+2j}}\left(\frac{|\Nm_{K/\mathbb{Q}}(\overline{\mathfrak{aN'}}(\alpha))|}{\sqrt{D_K}}\right)^{j-s} \\
&\hspace{8cm}= \frac{\psi_1^{-1}(-\sqrt{D_K})\chi^{-1}(\overline{\mathfrak{t}})}{\sqrt{D_K}^{j-s}t^s}L(\psi_{1/K}(\chi_{-1})^{-1},1+s).\\
&\hspace{8cm}
\end{align*}
Note that each summand of the inner sum can be written as the special value of a real analytic Eisenstein series at a certain CM point:
\begin{align*}\sum_{c=0}^{u-1}\sum_{d=0}^{t-1}&\sum_{e=0}^{u-1}\sum_{(m,n)\in\mathbb{Z}^2\setminus\{0\}, (m,n)\equiv (ct,d+et) \mod N'}\frac{\psi_1(c)\psi_2^{-1}(d)}{(m\tau_{\mathfrak{aN'}}+n)^{2+2j}}\left(\frac{|m\tau_{\mathfrak{aN'}}+n|^2}{\tau_{\mathfrak{aN'}}-\overline{\tau_{\mathfrak{aN'}}}}\right)^{j-s}.
\end{align*}
Using standard analytic arguments, one can show the above expression tends to $\partial^jE_2^{\psi_1,\psi_2}(\tau_{\mathfrak{aN'}})$ as $s \rightarrow 0$. Hence, combining the above, we have
\begin{align*}\sum_{[\mathfrak{a}]\in \mathcal{C}\ell(\mathcal{O}_K)}(\chi_j)^{-1}(\overline{\mathfrak{aN'}})\partial^jE_2^{\psi_1,\psi_2}(\tau_{\mathfrak{aN'}}) &\hspace{-.05cm}=\hspace{-.05cm}\lim_{s\rightarrow 0} \frac{t^2\Gamma(2+j)\psi_1^{-1}(-\sqrt{D_K})\chi^{-1}(\overline{\mathfrak{t}})}{(2\pi i)^{2+j}\mathfrak{g}(\psi_2^{-1})\sqrt{D_K}^{j-s}t^s}L(\psi_{1/K}(\chi_{-1})^{-1},1+s) \\
&\hspace{-.05cm}=\hspace{-.05cm}\frac{t^2\Gamma(2+j)\psi_1^{-1}(-\sqrt{D_K})\chi^{-1}(\overline{\mathfrak{t}})}{(2\pi i)^{2+j}\mathfrak{g}(\psi_2^{-1})\sqrt{D_K}^j}L(\psi_{1/K}(\chi_{-1})^{-1},1).
\end{align*}
Thus, in all we have
\begin{proposition}\label{Lvalueproposition} Let $k \ge 2$ be an integer. Suppose $\chi$ is of infinity type $(k+j,-j)$ and $\chi_{-\frac{k}{2}}$ has trivial central character and is of finite type $(\mathfrak{N}',\psi_1\psi_2)$. Suppose also that $\psi_1$ and $\psi_2$ are Dirichlet characters over $\mathbb{Q}$ with conductors $u$ and $t$, respectively, such that $ut = N$ and $(\psi_1\psi_2)(-1) = (-1)^k$. Then
$$\sum_{[\mathfrak{a}]\in \mathcal{C}\ell(\mathcal{O}_K)}(\chi_j)^{-1}(\mathfrak{a})\partial^jE_k^{\psi_1,\psi_2}(\tau_{\mathfrak{a}})= \frac{t^k\Gamma(k+j)\psi_1^{-1}(-\sqrt{D_K})\chi^{-1}(\overline{\mathfrak{t}})}{(2\pi i)^{k+j}\mathfrak{g}(\psi_2^{-1})\sqrt{D_K}^j}L(\psi_{1/K}\chi^{-1},0).$$
\end{proposition}

\subsection{Stabilization operators}\label{modify}
We have the following ``stabilization operators" acting on normalized $\Gamma_0(N')$-eigenforms $F$ with character $\varepsilon_F$. Given a rational prime $\ell\nmid N'$, let $a_{\ell}$ denote the eigenvalue of $F$ under the Hecke operator $T_{\ell}$, let $(\alpha_{\ell},\beta_{\ell})$ denote (some henceforth fixed ordering of) the algebraic numbers such that $\alpha_{\ell} + \beta_{\ell} = a_{\ell}$, $\alpha_{\ell}\beta_{\ell} = \ell^{k-1}\varepsilon_F(\ell)$, and $\ord_{\ell}(\alpha_{\ell}) \le \ord_{\ell}(\beta_{\ell})$. Now define
\begin{align*}&F(q) \mapsto F^{(\ell^+)}(q) := F(q) - \beta_{\ell}F(q^{\ell})\in M_k(\Gamma_0(N'\ell),\varepsilon_F),\\
&F(q) \mapsto F^{(\ell^-)}(q) := F(q) - \alpha_{\ell}F(q^{\ell})\in M_k(\Gamma_0(N'\ell),\varepsilon_F),\\
&F(q) \mapsto F^{(\ell^0)}(q) := F(q) - a_{\ell}F(q^{\ell}) + \varepsilon_F(\ell)\ell^{k-1}F(q^{\ell^2})\in M_k(\Gamma_0(N'\ell^2),\varepsilon_F).
\end{align*}
Note that for $\ell_1 \neq \ell_2$, the stabilization operators $F \mapsto F^{(\ell_1^{\epsilon_1})}$ and $F\mapsto F^{(\ell_2^{\epsilon_2})}$ commute for any $\epsilon_1, \epsilon_2 \in \{+,-,0\}$. Then we define, for integers $S = \prod_i \ell_i^{e_i}$, $\epsilon,\epsilon_1,\epsilon_2 \in \{+,-,0\}$,
$$F^{(S^{\epsilon})} := F^{\prod_i (\ell_i^{\epsilon})}, \hspace{1cm} F^{(S_1^{\epsilon_1}S_2^{\epsilon_2})} := F^{(S_1^{\epsilon_1}),(S_2^{\epsilon_2})}.$$

These operators clearly extend to $p$-adic modular forms. On the $p$-adic modular forms $\theta^jE_k^{\psi_1,\psi_2}$ from Section \ref{charactertwist}, we explicitly have
\begin{align*}\theta^jE_k^{\psi_1,\psi_2,(\ell^+)}(\tau) &:= \delta_{\psi_1=1,j=0}L(\psi_1,1-k)(1-\psi_2(\ell)\ell^{k-1}) + 2\sum_{n=1}^{\infty}n^j\sigma_{k-1}^{\psi_1,\psi_2,(\ell^+)}(n)q^n \\
&= \theta^jE_k^{\psi_1,\psi_2}(\tau) - \psi_2(\ell)\ell^{k-1+j}\theta^jE_k^{\psi_1,\psi_2}(\ell\tau),
\end{align*}
\begin{align*}\theta^jE_k^{\psi_1,\psi_2,(\ell^-)}(\tau) &:= \delta_{\psi_1=1,j=0}L(\psi_1,1-k)(1-\psi_2(\ell)) + 2\sum_{n=1}^{\infty}n^j\sigma_{k-1}^{\psi_1,\psi_2,(\ell^-)}(n)q^n \\
&= \theta^jE_k^{\psi_1,\psi_2}(\tau) - \psi_1(\ell)\ell^j\theta^jE_k^{\psi_1,\psi_2}(\ell\tau),
\end{align*}
\begin{align*}&\theta^jE_k^{\psi_1,\psi_2,(\ell^0)}(\tau) \\
&:= \delta_{\psi_1=1,j=0}L(\psi_1,1-k)(1-\psi_1(\ell)-\psi_2(\ell)\ell^{k-1}+\psi_2(\ell^2)\ell^{k-1}) + 2\sum_{n=1}^{\infty}n^j\sigma_{k-1}^{\psi_1,\psi_2,(\ell^0)}(n)q^n \\
&= \theta^jE_k^{\psi_1,\psi_2}(\tau) - \ell^j(\psi_1(\ell)+\psi_2(\ell)\ell^{k-1})\theta^jE_k^{\psi_1,\psi_2}(\ell\tau) + (\psi_1\psi_2)(\ell)\ell^{k-1+2j}\theta^jE(\ell^2\tau),
\end{align*} 
where $\delta_{\psi=1,j=0} = 1$ if $\psi$ is trivial and $j=0$, and 0 otherwise, and
\begin{align*}&\sigma_{k-1}^{\psi_1,\psi_2, (\ell^+)}(n) := \sum_{0<d|n, (d,\ell) = 1}\psi_1(n/d)\psi_2(d)d^{k-1},\\
&\sigma_{k-1}^{\psi_1,\psi_2, (\ell^-)}(n) := \sum_{0<d|n, (n/d,\ell) = 1}\psi_1(n/d)\psi_2(d)d^{k-1},\\
&\sigma_{k-1}^{\psi_1,\psi_2, (\ell^0)}(n) := \sum_{0<d|n, (n,\ell)=1}\psi_1(n/d)\psi_2(d)d^{k-1}.
\end{align*}

Let $N'$ be as in Section \ref{charactertwist}, let $N'' = N_+''N_-''N_0''$ be prime to $N'$, and suppose additionally that $K$ satisfies the Heegner hypothesis with respect to $N''$ so that every prime dividing $N := N''N' $ is split in $K$. Let $\mathfrak{N} = \mathfrak{N}''\mathfrak{N}'$ where $\mathfrak{N}'$ is as in Section \ref{charactertwist} and $\mathfrak{N}''$ is some choice of integral ideal such that $\mathcal{O}_K/\mathfrak{N}'' = \mathbb{Z}/N''$. Write $\mathfrak{N}$ as a product of (distinct) primes $\prod_{\ell|N} v$ where $v|\ell$; in other words, for each $\ell|N$, we can write
$$v = \mathbb{Z}\frac{b_{\ell}+\sqrt{D_K}}{2} + \mathbb{Z}\ell$$
for some $b_{\ell} \in \mathbb{Z}$ such that $b_{\ell}^2 \equiv D_K \mod 4\ell$.

For any integral primitive ideal $\mathfrak{a}$ of $\mathcal{O}_K$ coprime with $\mathfrak{N}$, recall the associated point $\tau_{\mathfrak{a}\mathfrak{N}} = \frac{b_{aN} + \sqrt{D_K}}{2aN}\in \mathcal{H}^+$ such that $(\overline{\mathfrak{aN}})^{-1} = \mathbb{Z}\tau_{\mathfrak{aN}}+ \mathbb{Z}$. Note that for $v|\mathfrak{N}''$ where $v|\ell$, we have 
$$\overline{v}(\overline{\mathfrak{a}\mathfrak{N}})^{-1} = (\overline{\mathfrak{a}\prod_{v'\neq v}v'})^{-1} = \mathbb{Z}\ell\frac{b_{aN}+\sqrt{D_K}}{2aN} + \mathbb{Z},$$
and hence
\begin{align*}&\overline{v}^{-1}\star(\overline{\mathfrak{aN}} \star \mathbb{C}/\mathcal{O}_K) = \mathbb{C}/(\overline{v}(\overline{\mathfrak{aN}})^{-1}) = \mathbb{C}/(\mathbb{Z}\ell\tau_{\mathfrak{aN}} + \mathbb{Z}).
\end{align*}

In terms of the action of $\mathbb{I}^{\mathfrak{N}}$ on CM triples $(A',t',\omega')$, we have for any $F\in M_k(\Gamma_0(N),\varepsilon_F)$ of weight $k+2j$,
$$F(\ell\tau_{\mathfrak{a}\mathfrak{N}}) = F(\overline{v}^{-1}\star(\mathbb{C}/(\mathbb{Z}\tau_{\mathfrak{a}\mathfrak{N}}+\mathbb{Z}),t,2\pi i dz)) =  F(\overline{v}^{-1}\overline{\mathfrak{aN}}\star(\mathbb{C}/\mathcal{O}_K,t,2\pi i dz)),$$
$$F(\ell^2\tau_{\mathfrak{a}\mathfrak{N}}) = F(\overline{v}^{-2}\star(\mathbb{C}/(\mathbb{Z}\tau_{\mathfrak{a}\mathfrak{N}}+\mathbb{Z}),t,2\pi i dz)) = F(\overline{v}^{-2}\overline{\mathfrak{aN}}\star(\mathbb{C}/\mathcal{O}_K,t,2\pi i dz)).$$
Thus, for any normalized eigenform $F$, viewed as a $p$-adic modular form, recalling our notation $A_0 = \mathbb{C}/\mathcal{O}_K$, $t_0 \in A_0[\mathfrak{N}]$ from Section \ref{twoformulas}, and $\omega_{can}$ as our ``canonical" differential on $\hat{\mathcal{A}}_0$ under our fixed isomorphism $i: \hat{\mathcal{A}}_0 \xrightarrow{\sim} \hat{\mathbb{G}}_m$ (see Section \ref{twoformulas}), we have for $\ell\nmid N'$
\begin{align*}&\theta^jF^{(\ell^+)}(\overline{\mathfrak{aN}}\star(A_0,t_0,\omega_{can}))= \theta^jF(\overline{\mathfrak{aN}}\star(A_0,t_0,\omega_{can})) - \beta_{\ell}\ell^j\theta^jF(\overline{v}^{-1}\overline{\mathfrak{aN}}\star(A_0,t_0,\omega_{can})),
\end{align*}
\begin{align*}&\theta^jF^{(\ell^-)}(\overline{\mathfrak{aN}}\star(A_0,t_0,\omega_{can}))= \theta^jF(\overline{\mathfrak{aN}}\star(A_0,t_0,\omega_{can})) - \alpha_{\ell}\ell^j\theta^jF(\overline{v}^{-1}\overline{\mathfrak{aN}}\star(A_0,t_0,\omega_{can})),
\end{align*}
\begin{align*}&\theta^jF^{(\ell^0)}(\overline{\mathfrak{aN}}\star(A_0,t_0,\omega_{can}))= \theta^jF(\overline{\mathfrak{aN}}\star(A_0,t_0,\omega_{can})) - \ell^ja_{\ell}\theta^jF(\overline{v}^{-1}\overline{\mathfrak{aN}}\star(A_0,t_0,\omega_{can}))\\
&\hspace{8cm}+ \varepsilon_F(\ell)\ell^{k-1+2j}\theta^jF(\overline{v}^{-2}\overline{\mathfrak{aN}}\star(A_0,t_0,\omega_{can})).
\end{align*}
Choosing some set of representatives $\mathfrak{a}$ prime to $\mathfrak{N}$, we thus have
\begin{align*}&\sum_{[\mathfrak{a}]\in\mathcal{C}\ell(\mathcal{O}_K)}(\chi_j)^{-1}(\mathfrak{a})\theta^jF^{(\ell^+)}(\mathfrak{a}\star(A_0,t_0,\omega_{can}))\\
&=\sum_{[\mathfrak{a}]\in\mathcal{C}\ell(\mathcal{O}_K)}(\chi_j)^{-1}(\mathfrak{a})\theta^jF(\mathfrak{a}\star(A_0,t_0,\omega_{can})) - \beta_{\ell}\ell^{j}\sum_{[\mathfrak{a}]\in\mathcal{C}\ell(\mathcal{O}_K)}(\chi_j)^{-1}(\mathfrak{a})\theta^jF(\overline{v}^{-1}\mathfrak{a}\star(A_0,t_0,\omega_{can}))\\
&= (1-\beta_{\ell}(\chi_j)^{-1}(\overline{v})\ell^{j})\sum_{[\mathfrak{a}]\in\mathcal{C}\ell(\mathcal{O}_K)}(\chi_j)^{-1}(\mathfrak{a})\theta^jF(\mathfrak{a}\star(A_0,t_0,\omega_{can})),
\end{align*}
\begin{align*}&\sum_{[\mathfrak{a}]\in\mathcal{C}\ell(\mathcal{O}_K)}(\chi_j)^{-1}(\mathfrak{a})\theta^jF^{(\ell^-)}(\mathfrak{a}\star(A_0,t_0,\omega_{can}))\\
&=\sum_{[\mathfrak{a}]\in\mathcal{C}\ell(\mathcal{O}_K)}(\chi_j)^{-1}(\mathfrak{a})\theta^jF(\mathfrak{a}\star(A_0,t_0,\omega_{can})) - \alpha_{\ell}\ell^j\sum_{[\mathfrak{a}]\in\mathcal{C}\ell(\mathcal{O}_K)}(\chi_j)^{-1}(\mathfrak{a})\theta^jF(\overline{v}^{-1}\mathfrak{a}\star(A_0,t_0,\omega_{can}))\\
&= (1-\alpha_{\ell}(\chi_j)^{-1}(\overline{v})\ell^j)\sum_{[\mathfrak{a}]\in\mathcal{C}\ell(\mathcal{O}_K)}(\chi_j)^{-1}(\mathfrak{a})\theta^jF(\mathfrak{a}\star(A_0,t_0,\omega_{can})),
\end{align*}
\begin{align*}&\sum_{[\mathfrak{a}]\in\mathcal{C}\ell(\mathcal{O}_K)}(\chi_j)^{-1}(\mathfrak{a})\theta^jF^{(\ell^+)}(\mathfrak{a}\star(A_0,t_0,\omega_{can}))\\
&=\sum_{[\mathfrak{a}]\in\mathcal{C}\ell(\mathcal{O}_K)}(\chi_j)^{-1}(\mathfrak{a})\theta^jF(\mathfrak{a}\star(A_0,t_0,\omega_{can})) - \ell^ja_{\ell}\sum_{[\mathfrak{a}]\in\mathcal{C}\ell(\mathcal{O}_K)}(\chi_j)^{-1}(\mathfrak{a})\theta^jF(\overline{v}^{-1}\mathfrak{a}\star(A_0,t_0,\omega_{can}))\\
&\hspace{6cm}+ \varepsilon_F(\ell)\ell^{k-1+2j} \sum_{[\mathfrak{a}]\in\mathcal{C}\ell(\mathcal{O}_K)}(\chi_j)^{-1}(\mathfrak{a})\theta^jF(\overline{v}^{-2}\mathfrak{a}\star(A_0,t_0,\omega_{can}))\\
&= (1-a_{\ell}(\chi_j)^{-1}(\overline{v})\ell^j + \varepsilon_F(\ell)(\chi_j)^{-2}(\overline{v})\ell^{k-1+2j})\sum_{[\mathfrak{a}]\in\mathcal{C}\ell(\mathcal{O}_K)}(\chi_j)^{-1}(\mathfrak{a})\theta^jF(\mathfrak{a}\star(A_0,t_0,\omega_{can})).
\end{align*}

Thus, rewriting the definitions of $E_k^{\psi_1,\psi_2,((N_+'')^+(N_-'')^-(N_0'')^0)}$ in terms of triples, using the above general identities for $p$-adic modular forms and induction, we have
\begin{align*}&\sum_{[\mathfrak{a}]\in \mathcal{C}\ell(\mathcal{O}_K)}(\chi_j)^{-1}(\mathfrak{a})\theta^jE_k^{\psi_1,\psi_2,((N_+'')^+(N_-'')^-(N_0'')^0)}(\mathfrak{a}\star(A_0,t_0,\omega_{can}))\\
&= \Xi_{\chi}(\psi_1,\psi_2,N_+,N_-,N_0)\sum_{[\mathfrak{a}]\in \mathcal{C}\ell(\mathcal{O}_K)}(\chi_j)^{-1}(\mathfrak{a})\theta^jE_k^{\psi_1,\psi_2}(\mathfrak{a}\star(A_0,t_0,\omega_{can}))
\end{align*}
where $N_+ = N_+''\cdot\frac{t}{(u,t)}, N_- = N_-''\cdot\frac{u}{(u,t)}$, $N_0 = N_0''\cdot(u,t)^2$, and
\begin{align*}\Xi_{\chi}(\psi_1,\psi_2,N_+,N_-,N_0) = &\prod_{\ell|N_+}(1-(\psi_{2/K}\chi^{-1})(\overline{v})\ell^{k-1})\prod_{\ell|N_-}(1-(\psi_{1/K}\chi^{-1})(\overline{v}))\\
&\prod_{\ell|N_0}(1-(\psi_{2/K}\chi^{-1})(\overline{v})\ell^{k-1})(1-(\psi_{1/K}\chi^{-1})(\overline{v})).
\end{align*}
Henceforth, let $E_k^{(\psi_1,\psi_2,N_+,N_-,N_0)} := E_k^{\psi_1,\psi_2,((N_+'')^+(N_-'')^-(N_0'')^0)}$. Now, using the above calculation, Proposition \ref{Lvalueproposition}, the fact that 
$$\omega_{can} = \frac{\Omega_{\infty}}{\Omega_p}\cdot 2\pi idz$$
(where $\Omega_{\infty}$ and $\Omega_p$ are the complex and $p$-adic periods defined in Section \ref{twoformulas}) and part (3) of Theorem \ref{MSthetathm}, we have
\begin{proposition}\label{ModifiedLvalueproposition} Let $k \ge 2$ be an integer. Suppose $\chi$ is of infinity type $(k+j,-j)$ and $\chi_{-\frac{k}{2}}$ has trivial central character and is of finite type $(\mathfrak{N},\psi_1\psi_2)$. Suppose also that $\psi_1$ and $\psi_2$ are Dirichlet characters over $\mathbb{Q}$ with conductors $u$ and $t$, respectively, such that $ut = N$ and $(\psi_1\psi_2)(-1) = (-1)^k$. Then
\begin{align*}&\sum_{[\mathfrak{a}]\in \mathcal{C}\ell(\mathcal{O}_K)}(\chi_j)^{-1}(\mathfrak{a})\theta^jE_k^{(\psi_1,\psi_2,N_+,N_-,N_0)}(\mathfrak{a}\star(A_0,t_0,\omega_{can}))\\
&= \left(\frac{\Omega_p}{\Omega_{\infty}}\right)^{k+2j}\frac{t^k\Gamma(k+j)\psi_1^{-1}(-\sqrt{D_K})\chi^{-1}(\overline{\mathfrak{t}})}{(2\pi i)^{k+j}\mathfrak{g}(\psi_2^{-1})\sqrt{D_K}^j}\Xi_{\chi}(\psi_1,\psi_2,N_+,N_-,N_0)L(\psi_{1/K}\chi^{-1},0)
\end{align*}
where 
\begin{align*}\Xi_{\chi}(\psi_1,\psi_2,N_+,N_-,N_0) = &\prod_{\ell|N_+}(1-(\psi_{2/K}\chi^{-1})(\overline{v})\ell^{k-1})\prod_{\ell|N_-}(1-(\psi_{1/K}\chi^{-1})(\overline{v}))\\
&\prod_{\ell|N_0}(1-(\psi_{2/K}\chi^{-1})(\overline{v})\ell^{k-1})(1-(\psi_{1/K}\chi^{-1})(\overline{v})).
\end{align*}
\end{proposition}
Finally, since $p\nmid N$, note that the ``$p$-depletion" operator $\flat$ of Section \ref{padicmodularforms} coincides with the $((p^2)^0)$-stabilization operator on our family of Eisenstein series, i.e. 
$$(\theta^jE_k^{(\psi_1,\psi_2,N_+,N_-,N_0)})^{\flat} = \theta^jE_k^{(\psi_1,\psi_2,N_+,N_-,p^2N_0)}.$$
(Note that all the above stabilization operators on $p$-adic modular forms commute with $\theta$, so the above notation is unambiguous.)
\subsection{Proof of Main Theorem}
\begin{proof}[Proof of Theorem \ref{newmainthm}]Suppose, as in our assumptions, that $f \in S_k(\Gamma_0(N),\varepsilon_f)$ has partial Eisenstein descent of type $(\psi_1,\psi_2,N_+,N_-,N_0)$ over $M$ mod $\mathfrak{m}$, where $M$ is a $p$-adic field containing $E_f$. Recall the field $F'$ defined in Section \ref{twoformulas}, with $\mathfrak{p}'$ be the prime ideal of $\mathcal{O}_{F'}$ above $p$ determined by our embedding $i_p: \overline{\mathbb{Q}}\hookrightarrow \overline{\mathbb{Q}}_p$. Thus,
$$\theta^jf(q) \equiv \theta^jE_k^{(\psi_1,\psi_2,N_+,N_-,N_0)}(q) \mod \mathfrak{m}\mathcal{O}_{F_{\mathfrak{p}'}'M}$$
viewed as $p$-adic modular forms for all $j \ge 1$. Moreover, 
$$\theta^jf^{\flat}(q) \equiv \theta^jE_k^{(\psi_1,\psi_2,N_+,N_-,p^2N_0)}(q) \mod \mathfrak{m}\mathcal{O}_{F_{\mathfrak{p}'}'M}.$$
Henceforth, let $E_{k,\psi_1,\psi_2}^{(N)} = E_k^{(\psi_1,\psi_2,N_+,N_-,N_0)}$ and $E_{k,\psi_1,\psi_2}^{(pN)} = E_k^{(\psi_1,\psi_2,N_+,N_-,p^2N_0)}$, and recall our notation $A_0 = \mathbb{C}/\mathcal{O}_K$ and $t_0\in A_0[\mathfrak{N}]$ (see Section \ref{twoformulas}). Since $p$ being split in $K$ implies that the curves $i_p(\mathfrak{a}\star A_0)$ are ordinary (viewed as curves over $\mathbb{C}_p$), by the congruence above and the $q$-expansion principle (\cite[Section I.3.2 and Section I.3.5]{Gouvea}) we have
\begin{align*}\sum_{\mathfrak{a} \in \mathcal{C}\ell(\mathcal{O}_K)}(\chi_j)&^{-1}(\mathfrak{a})(\theta^jf)^{\flat}(\mathfrak{a}\star(A_0,t_0,\omega_{can}))\\
&\equiv \sum_{\mathfrak{a} \in \mathcal{C}\ell(\mathcal{O}_K)}(\chi_j)^{-1}(\mathfrak{a})(\theta^j E_{k,\psi_1,\psi_2}^{(pN)})(\mathfrak{a}\star(A_0,t_0,\omega_{can})) \mod \mathfrak{m}\mathcal{O}_{F_{\mathfrak{p}'}'M}.
\end{align*}
By assumption, $\varepsilon_f = \psi_1\psi_2$ and $\chi$ is of finite type $(\mathfrak{N},\varepsilon_f) = (\mathfrak{N},\psi_1\psi_2)$. Now by Proposition \ref{ModifiedLvalueproposition}, we have
\begin{align*}&\sum_{\mathfrak{a} \in \mathcal{C}\ell(\mathcal{O}_K)}(\chi_j)^{-1}(\mathfrak{a})(\theta^j E_{k,\psi_1,\psi_2}^{(pN)})(\mathfrak{a}\star(A_0,t_0,\omega_{can})) \\
&=\left(\frac{\Omega_p}{\Omega_{\infty}}\right)^{k+2j}\frac{t^k\Gamma(k+j)\psi_1^{-1}(-\sqrt{D_K})\chi^{-1}(\overline{\mathfrak{t}})}{(2\pi i)^{k+j}\mathfrak{g}(\psi_2^{-1})\sqrt{D_K}^j}\Xi_{\chi}(\psi_1,\psi_2,N_+,N_-,p^2N_0)L(\psi_{1/K}\chi^{-1},0)
\end{align*}
where $\Xi_{\chi}(\psi_1,\psi_2,N_+,N_-,p^2N_0)$ is defined as in Proposition \ref{ModifiedLvalueproposition}.

Since $\chi$ is of finite type $(\mathfrak{N},\varepsilon_f)$, it is unramifed and in particular unramified at $\mathfrak{p}$. Moreover, since $\mathfrak{p}\nmid u\mathcal{O}_K$, then $\psi_{1/K}$ is unramified at $\mathfrak{p}$. Hence, by \cite[Lemma 5.2.27-5.2.28]{Katz1}, we have $\text{Local}_{\mathfrak{p}}(\psi_{1/K}\chi^{-1}) = 1$ and $\text{Local}_{\mathfrak{p}}(\psi_{1/K}\chi^{-1}\mathbb{N}_K^{1-k}) = 1$.

Now using Theorem \ref{BDPinterpolationproperty} and Theorem \ref{KatzpadicLfunction'} with $\mathfrak{C} = \mathfrak{f}(\psi_{1/K}\chi^{-1}) = lcm(u\mathcal{O}_K,\mathfrak{t})$ where $\mathfrak{F} = lcm(\mathfrak{u},\mathfrak{t}), \mathfrak{F}_c = \overline{\mathfrak{u}}$ and $\mathfrak{I} = 1$, we can write
\begin{align*}&\mathcal{L}_p(f,\chi) \\
&= \left(\sum_{\mathfrak{a}\in\mathcal{C}\ell(\mathcal{O}_K)}(\chi_j)^{-1}(\mathfrak{a})(\theta^j f)^{\flat}(\mathfrak{a}\star(A_0,t_0,\omega_{can}))\right)^2 \\
&\equiv \left(\sum_{\mathfrak{a}\in\mathcal{C}\ell(\mathcal{O}_K)}(\chi_j)^{-1}(\mathfrak{a})(\theta^jE_{k,\psi_1,\psi_2}^{(pN)})(\mathfrak{a}\star(A_0,t_0,\omega_{can}))\right)^2\\
&=\left(\left(\frac{\Omega_p}{\Omega_{\infty}}\right)^{k+2j}\frac{t^k\Gamma(k+j)\psi_1^{-1}(-\sqrt{D_K})\chi^{-1}(\overline{\mathfrak{t}})}{(2\pi i)^{k+j}\mathfrak{g}(\psi_2^{-1})\sqrt{D_K}^j}\Xi_{\chi}(\psi_1,\psi_2,N_+,N_-,p^2N_0)L(\psi_{1/K}\chi^{-1},0)\right)^2\\
&=\psi_1^{-1}(D_K)\left(\frac{t^k\chi^{-1}(\overline{\mathfrak{t}})}{4\mathfrak{g}(\psi_2^{-1})(2\pi i)^{k+2j}}\Xi_{\chi}(\psi_1,\psi_2,N_+,N_-,N_0)L_p(\psi_{1/K}\chi^{-1},0)\right)^2 \mod \mathfrak{m}\mathcal{O}_{F_{\mathfrak{p}'}'M}.
\end{align*}

The above congruence now extends by $p$-adic continuity from $\Sigma_{cc}^{(2)}(\mathfrak{N})$ to $\hat{\Sigma}_{cc}(\mathfrak{N})$ (with respect to the topology described in Section \ref{twoformulas}). Thus the congruence holds on $\hat{\Sigma}_{cc}(\mathfrak{N})$, and the Main Theorem follows after putting $\Xi := \Xi_{\chi}(\psi_1,\psi_2,N_+,N_-,N_0)$. 
\end{proof}

\begin{proof}[Proof of Theorem \ref{newmaincorollary}]In this more specialized situation, we have $\psi_1 = \psi_2^{-1} = \psi$ and $u = t = \mathfrak{f}(\psi)$. Thus $\psi_1\psi_2 = 1$ and $\chi$ is of finite type $(\mathfrak{N},1)$. Recall from Bertolini-Darmon-Prasanna's $p$-adic Waldspurger formula (Theorem \ref{AbelJacobicongruence'}):
$$\mathcal{L}_p(f,\mathbb{N}_K^{-k/2}) = \left(\frac{p^{k/2}-a_p(f) + p^{k/2-1}}{p^{k/2}\Gamma(\frac{k}{2})}\right)^2\left(AJ_{F_{\mathfrak{p}'}'M}(\Delta_f(K))(\omega_f\wedge\omega_A^{k/2-1}\eta_A^{k/2-1})\right)^2.$$
The proof proceeds essentially by plugging $\chi = \mathbb{N}_K^{-k/2}$ into Theorem \ref{newmainthm}. We have by Gross's factorization formula (Theorem \ref{Grosstheorem'}) \emph{applied to $(\psi\omega^{-k/2})_0$} (see Definition \ref{evendefinition}):
\begin{align*}&\frac{\mathfrak{f}(\psi)^{k/2}}{\psi(-\sqrt{D_K})\mathfrak{g}(\psi)}L_p(\psi_{/K}\mathbb{N}_K^{k/2},0) \\
&\frac{\mathfrak{f}(\psi)^{k/2}}{\prod_{\ell|f}(\psi^{-1}\omega^{k/2})_{0,\ell}(-\sqrt{D_K})\mathfrak{g}((\psi^{-1}\omega^{k/2})_{0,\ell})}L_p((\psi\omega^{-k/2})_{0/K},\frac{k}{2}) \\
&=L_p(\psi_0(\varepsilon_K\omega)^{1-k/2},\frac{k}{2})L_p(\psi_0^{-1}(\varepsilon_K\omega)^{k/2},1-\frac{k}{2}).
\end{align*}
Using the interpolation property of the Kubota-Leopoldt $p$-adic $L$-function, we have
\begin{align*}&L_p(\psi_0^{-1}(\varepsilon_K\omega)^{k/2},1-\frac{k}{2}) \\
&= -(1-(\psi_0^{-1}\varepsilon_K^{k/2})(p)p^{k/2-1})\frac{2}{k}B_{\frac{k}{2},\psi_0^{-1}\varepsilon_K^{k/2}} \\
&= -(1-\psi_0^{-1}(p)p^{k/2-1})\frac{2}{k}B_{\frac{k}{2},\psi_0^{-1}\varepsilon_K^{k/2}}
\end{align*}
where the last equality follows since $p$ is split in $K$ and so $\varepsilon_K(p) = 1$.

Suppose $\psi_0(\varepsilon_K\omega)^{1-k/2} \neq 1$. In this case, since $p^2\nmid \mathfrak{f}(\psi_0(\varepsilon_K\omega)^{1-k/2})|pf$, we have by Corollary 5.13 of \cite{Washington} that
\begin{align*}&L_p(\psi_0(\varepsilon_K\omega)^{1-k/2},\frac{k}{2}) \\
&\equiv L_p(\psi_0(\varepsilon_K\omega)^{1-k/2},0) \\
&= -(1-(\psi_0\varepsilon_K(\varepsilon_K\omega)^{-k/2})(p))B_{1,\psi_0\varepsilon_K(\varepsilon_K\omega)^{-k/2}} \\
&= -B_{1,\psi_0\varepsilon_K(\varepsilon_K\omega)^{-k/2}}\mod p
\end{align*}
where the last equality holds since $p|\mathfrak{f}(\psi_0\varepsilon_K(\varepsilon_K\omega)^{1-k/2})$. Note this congruence also holds mod $\lambda$ for any prime $\lambda|p$.

Now suppose $\psi_0(\varepsilon_K\omega)^{1-k/2} = 1$ and $k = 2$, i.e. $\psi_0 = 1$. From \cite[p. 90]{Gross}, we have the following special value for the Katz $p$-adic $L$-function at $\psi_{0/K} = 1_{/K}$. 
$$L_p(1_{/K},0) = \frac{4}{|\mathcal{O}_K^{\times}|}\frac{p-1}{p}\log_p(\overline{\alpha}) = 2\frac{p-1}{p}\log_p(\overline{\alpha})$$
where $\overline{\alpha} \in \mathcal{O}_K$ such that $(\overline{\alpha}) = \overline{\mathfrak{p}}^{h_K}$. (Recall that $|\mathcal{O}_K^{\times}| = 2$ since we assume that $D_K<-4$.)
The statement now follows directly from Theorem \ref{newmainthm}.
\end{proof}

\section{Concrete Applications of the Main Theorem}\label{ConcreteApplications} In this section, we apply Theorem \ref{newmaincorollary} to computations with algebraic cycles, in certain instances verifying a weak form of the Beilinson-Bloch conjecture (as described in Section \ref{Chowmotives'}). In Section \ref{ellipticcurves}, we apply our results to the case of elliptic curves with reducible mod $p$ Galois representation, in particular deriving criteria to show that Heegner points on certain quadratic twists are non-torsion. In Section \ref{positiveproportionsection}, we use this criterion to show that for semistable curves with reducible mod $3$ Galois representation, a positive proportion of real quadratic twists have rank 1 and a positive proportion of imaginary quadratic twists have rank 0.

\subsection{Construction of newforms having Eisenstein descent of type $(1,1,1,1,1)$}

We now give a procedure for constructing newforms $f\in S_k(SL_2(\mathbb{Z}))$ which have Eisenstein descent of type $(1,1,1,1,1)$. For this we follow Skinner \cite[pp. 6-7]{Skinner}. 

\begin{construction}[Eisenstein descent of type (1,1,1,1,1)]\label{newformconstruction} Fix an even integer $k\ge 4$ and a prime $p|B_k$. Recall the classical holomorphic Eisenstein series of weight $k$:
$$E_k(q) = \frac{\zeta(1-k)}{2} + \sum_{n=1}^{\infty}\sigma_{k-1}(n)q^n = -\frac{B_{k}}{2k} + \sum_{n=1}^{\infty}\sigma_{k-1}(n)q^n.$$
(See Remark \ref{Eisensteinremark}.) The Eisenstein series $E_4$ and $E_6$ generate $M_k(SL_2(\mathbb{Z}))/S_k(SL_2(\mathbb{Z}))$ as an algebra, and therefore normalizing $E_4$ and $E_6$ to $G_4 = 240 E_4$ and $G_6 = -504 E_6$ to have constant term 1, we have a cuspform
$$f_k := E_k+ \frac{B_{k}}{2k}G_4^aG_6^b \in S_k(SL_2(\mathbb{Z}),\mathbb{Q})$$
for some $a,b \in \mathbb{Z}_{\ge 0}$ with $k = 4a+6b$.  

Under our assumption $p|B_{k}$, we see immediately that $f_k \equiv E_k \mod p$. However, there is no guarantee that $f_k$ is a \emph{newform}, a problem we can remedy as follows. Consider the Hecke algebra $\mathbb{T}\subset \text{End}_{\mathbb{Z}_p}(S_k(SL_2(\mathbb{Z}),\mathbb{Z}_p))$ generated by the classical Hecke operators $T_{\ell}$ at primes $\ell$. Since $\mathbb{T}$ is a finite torsion-free $\mathbb{Z}_p$-algebra, it is flat. Minimal primes $\mathfrak{p}$ of $\mathbb{T}$ correspond to $\Gal(\overline{\mathbb{Q}}_p/\mathbb{Q}_p)$-conjugacy classes of weight $k$ eigenforms in $S_k(SL_2(\mathbb{Z}))$, which also correspond to $\mathbb{Z}_p$-algebra homomorphisms $\phi : \mathbb{T} \rightarrow \overline{\mathbb{Q}}_p$ (with $\ker(\phi) = \mathfrak{p}$), which in turn correspond to eigenforms $f_{\phi}$ such that $T_{\ell}f_{\phi} = \phi(T_{\ell})f_{\phi}$.

Now write $f_k(q) = \sum_{n=1}^{\infty}a_n q^n$ so that $a_n \equiv \sigma_{k-1}(n) \mod p$, and define a homomorphism $\varphi: \mathbb{T} \rightarrow \mathbb{F}_p$ given by $\varphi(T_{\ell}) = a_{\ell} \equiv 1+\ell^{k-1}\mod p$. Since $\mathbb{F}_p$ is a field, $\mathfrak{m} := \ker(\varphi)$ is maximal, and there exists a minimal prime $\mathfrak{p} \subset\mathfrak{m}$ above $p$ since $\mathbb{T}$ is flat over $\mathbb{Z}_p$. Then the minimal prime $\mathfrak{p}$ corresponds to a map $\phi : \mathbb{T} \rightarrow \overline{\mathbb{Q}}_p$ with $\text{ker}(\phi) = \mathfrak{p}$ and which satisfies $\phi \equiv \varphi \mod p$. Let $F = \text{Frac}(\phi(\mathbb{T}))$ denote the field of fractions of $\phi(\mathbb{T})$, and note that $\phi(\mathbb{T}) \subset \mathcal{O}_F$; since $\mathbb{T}$ is finite over $\mathbb{Z}_p$, then $F$ is finite over $\mathbb{Q}_p$. Now let $\lambda$ denote the maximal ideal of $\mathcal{O}_F$. Since $T_{\ell} \equiv 1 + \ell^{k-1} \mod \mathfrak{m}$ (because $\varphi : \mathbb{T}/\mathfrak{m} \xrightarrow{\sim} \mathbb{F}_p$), then $\phi(T_{\ell}) \equiv 1 + \ell^{k-1} \mod \phi(\mathfrak{m})$. Hence since $\lambda|\phi(\mathfrak{m})$, the modular form $f_\phi$ corresponding to $\phi$ satisfies
$$f_{\phi}(q) \equiv \sum_{n=1}^{\infty}\sigma_{k-1}(n)q^n \mod \lambda$$
and is our desired newform. 
\end{construction}
\begin{remark}The last paragraph of Construction \ref{newformconstruction} is commonly known as the ``Deligne-Serre lifting lemma".
\end{remark}

\subsection{Construction of newforms having Eisenstein descent of type $(1,1,N_+,N_-,N_0)$}

\begin{construction}[Eisenstein descent of type $(1,1,N_+,N_-,N_0)$]\label{newformconstructionlevelN} We can easily use the normalized newform with Eisenstein descent of type $(1,1,1,1,1)$ obtained from Construction \ref{newformconstruction} to produce a newform of descent of type $(1,1,N_+,N_-,N_0)$. We apply certain stabilization operators, as in Section \ref{modify}. Suppose we are given a normalized newform $f \in S_k(\Gamma_1(N'))$ with Eisenstein descent of type $(1,1,N_+',N_-',N_0')$ mod $\lambda$. Applying the $(N_+^+N_-^-N_0^0)$-stabilization operator, the resulting $f^{(N_+^+N_-^-N_0^0)} \in S_k(\Gamma_1(N')\cap\Gamma_0(N))$ is a normalized newform which has Eisenstein descent of type $(1,1,N_+'N_+,N_-'N_-,N_0'N_0)$ mod $\lambda$.

Applying the above $(N_+^+N_-^-N_0^0)$-stabilization operator to the newform $f_{\phi} \in S_k(SL_2(\mathbb{Z}))$ from Construction \ref{newformconstruction}, we have
$$f_{\phi}^{(N_+^+N_-^-N_0^0)}(q)\equiv \sum_{n=1}^{\infty}\sigma_{k-1}^{(N)}(n)q^n \mod \lambda,$$
where $\sigma_{k-1}^{(N)}$ is defined as in Remark \ref{Eisensteinremark}. In other words, $f_{\phi}^{(N_+^+N_-^-N_0^0)}$ has Eisenstein descent of type $(1,1,N_+,N_-,N_0)$ mod $\lambda$.
\end{construction}

\begin{remark}We could similarly produce examples of newforms with Eisenstein descent of type $(\psi_1,\psi_2,N_+,N_-,N_0)$ by starting out with the Eisenstein series $E_k^{\psi_1,\psi_2}$ in Construction \ref{newformconstruction} and applying appropriate stabilization operators as in Construction \ref{newformconstructionlevelN}.
\end{remark}

\subsection{Application to algebraic cycles} We now calculate an explicit example to demonstrate the Main Theorem. We first use Construction \ref{newformconstruction} to construct an approriate newform with Eisenstein descent of type (1,1,1,1,1). We thus look for a positive integer $k$, a rational prime $p$, a real quadratic extension $L/\mathbb{Q}$, and an imaginary quadratic extension $K/\mathbb{Q}$ such that $p$ splits in $K$, $p| B_{k}$, $p\nmid B_{\frac{k}{2}, \varepsilon_L\varepsilon_K}$ and $p\nmid B_{1,\varepsilon_L\omega^{-k/2}}$. To this end, consider $k = 18$, $p = 43867$, $L = \mathbb{Q}$ and $K = \mathbb{Q}(\sqrt{-5})$.  Then $43867$ splits in $K$, and $p\vert\frac{43867}{798} = B_{18} = B_k$. Furthermore 
$$B_{9,\varepsilon_K} = -5444415378 \equiv 5726 \mod 43867$$
and so $p\nmid B_{\frac{k}{2},\varepsilon_{L\cdot K}} = B_{\frac{k}{2},\varepsilon_K}$. 
\begin{remark}\label{Teichmullersimplify}For certain values of $k$, we can simplify the above formula even further. Let $[m]$ denote the smallest nonnegative representative of the residue class mod $p-1$ of an integer $m$. In the case that $2 \le [-\frac{k}{2}] \le p-4$, by standard congruence theorems (see \cite{Washington}, Chapter 5.3 for example), we have
$$B_{1,\omega^{[-\frac{k}{2}]}} \equiv \frac{B_{[-\frac{k}{2}]+1}}{[-\frac{k}{2}]+1} \mod p$$
(and hence this congruence also holds mod $\lambda\vert p$).
\end{remark}
Hence, by Remark \ref{Teichmullersimplify}, we have
$$B_{1,\omega^{-9}} \equiv \frac{B_{43858}}{43858} \equiv 11867\mod 43867$$
and so $p\nmid B_{1,\varepsilon_L\omega^{-k/2}} = B_{1,\omega^{-9}}$. Hence, applying Construction \ref{newformconstruction}, we get a newform $f_{18} \in S_{18}(SL_2(\mathbb{Z}))$ such that 
$$f_{18}(q) \equiv \sum_{n=1}^{\infty}\sigma_{17}(n)q^n \mod \lambda$$
for some prime ideal $\lambda|p$ of a finite extension over $\mathbb{Q}_p$. 
Note we can apply the $(N_+^+N_-^-N_0^0)$-stabilization operator of Construction \ref{newformconstructionlevelN} to obtain a newform $f_{18}^{(N)}$ of weight 18 which has Eisenstein descent of type $(1,1,N_+,N_-,N_0)$ mod $\lambda$. Choose $(N_+,N_-,N_0) = (7,1,1)$. Then 7 splits in $K = \mathbb{Q}(\sqrt{-5})$. Furthermore
$$\Xi(1,1,7,1,1) = 1-7^{8} \equiv 25644 \mod 43867$$
and thus $\Xi(1,1,7,1,1)\in\mathbb{Q}$ is not congruent to $0 \mod \lambda$.

Let $F/H_{\mathfrak{N}}/K$ be in situation $(S)$ as defined in Section \ref{Chowmotives'}. Applying Theorem \ref{newmaincorollary}, we determine the non-triviality of the associated generalized Heegner cycle $\epsilon_{(f_{18},\mathbb{N}_K^{-8})_{/F}}\Delta(\mathbb{N}_K^{-8}) \in \epsilon_{(f_{18},\mathbb{N}_K^{-8})_{/F}}\operatorname{CH}^{17}(X_{16})_{0,E_{f_{18}}}(F)$. Thus we have constructed an algebraic cycle with non-trivial $(f_{18},\mathbb{N}_K^{-8})$-isotypic component, whose existence is predicted by the Beilinson-Bloch conjecture since by the Heegner hypothesis $L(f_{18},\mathbb{N}_K^{8},0) = 0$. (See Sections \ref{Chowmotives'} and \ref{twoformulas} for further details.)

\subsection{Application to the Ramanujan $\varDelta$ function}\label{Ramanujan} Recall the Ramanujan $\varDelta$ function, which is a weight 12 normalized newform of level 1 with $q$-expansion at $\infty$ given by
$$\varDelta(q) = \sum_{n\ge 1}\tau(n)q^n.$$
It is well-known that $\tau$ satisfies the congruence $\tau(n) \equiv \sigma_{11}(n) \mod 691$, and so for any $j \ge 1$ and any quadratic Dirichlet character $\psi$, we have
$$(\theta^j\varDelta\otimes\psi)(q) \equiv (\theta^jE_{12}\otimes\psi)(q) \mod 691,$$
i.e. $\varDelta\otimes \psi$ has partial Eisenstein descent over $\mathbb{Q}_p$ of type $(\psi,\psi,1,1,1)$ at $691$. Choosing an auxiliary imaginary quadratic field $K$ satisfying Assumptions \ref{assumptions} with respect to $(681,\varDelta\otimes\psi)$ and applying Theorem \ref{newmaincorollary}, we get:
\begin{theorem}\label{Ramanujantheorem}
\begin{align*}\left(\frac{691^6-\psi(691)\tau(691)+691^5}{691^6\cdot5!}
\right)^2&\left(AJ_{F_{\mathfrak{p}'}'}(\Delta_{\varDelta\otimes\psi}(\mathbb{N}_K^{-5}))(\omega_{\varDelta}\wedge\omega_A^5\eta_A^5)\right)^2 \\
&\equiv \frac{1}{576}\left(B_{6,\psi_0}B_{1,\psi_0\varepsilon_K\omega^{-6}}\right)^2 \mod 691\mathcal{O}_{F'}
\end{align*}
where $\psi_0$ is defined as in Definition \ref{evendefinition}.
\end{theorem}
\begin{corollary}Let $F/H_{\mathfrak{N}}/K$ be in situation $(S)$ as defined in Section \ref{Chowmotives'}. Suppose $\psi$ is a quadratic character and $K$ is an imaginary quadratic field with odd discriminant $D_K < -4$ such that
\begin{enumerate}
\item $691\nmid \mathfrak{f}(\psi)$,
\item each prime factor of $691\cdot\mathfrak{f}(\psi)$ splits in $K$,
\item $691\nmid B_{6,\psi_0}B_{1,\psi_0\varepsilon_K\omega^{-6}}$.
\end{enumerate}
Then $\epsilon_{(\varDelta\otimes\psi,\mathbb{N}_K^{-5})_{/F}}\Delta_{\varDelta\otimes\psi}(\mathbb{N}_K^{-5}) \in \epsilon_{(\varDelta\otimes\psi,\mathbb{N}_K^{-5})_{/F}}\operatorname{CH}^{11}(X_{10})_{0,\mathbb{Q}}(F)$ is non-trivial.
\end{corollary}
To elucidate this result, we include the following table exhibiting a few values of quadratic characters $\psi$ over $\mathbb{Q}$ and imaginary quadratic fields $K$ for which Theorem \ref{Ramanujantheorem} implies $\epsilon_{(\varDelta\otimes\psi,\mathbb{N}_K^{-5})_{/F}}\Delta_{\varDelta\otimes\psi}(\mathbb{N}_K^{-5}) \in \epsilon_{(\varDelta\otimes\psi,\mathbb{N}_K^{-5})_{/F}}\operatorname{CH}^{11}(X_{10})_{0,\mathbb{Q}}(F)$ is non-trivial.
\begin{center}
\begin{tabular}{| c | c | c | c |}
\hline
$K_\psi$ & $\mathfrak{f}(\psi)$ & $K$ & $B_{6,\psi_0}B_{1,\psi_0\varepsilon_K\omega^{-6}} \mod 691$ \\ \hline
$\mathbb{Q}(\sqrt{3})$ & $12$ & $\mathbb{Q}(\sqrt{-23})$ & 583\\ \hline
$\mathbb{Q}(\sqrt{3})$ & $12$ & $\mathbb{Q}(\sqrt{-95})$ & 126\\ \hline
$\mathbb{Q}(\sqrt{13})$ & $13$ & $\mathbb{Q}(\sqrt{-10})$ & 583\\ \hline
$\mathbb{Q}(\sqrt{-7})$ & $-7$ & $\mathbb{Q}(\sqrt{-10})$ & 176\\
\hline
\end{tabular}
\end{center}
\subsection{Application to elliptic curves}\label{ellipticcurves}We now focus on the case where $f_E \in S_2(\Gamma_0(N))$ is the weight 2 normalized newform associated with an elliptic curve $E/\mathbb{Q}$. If $E[p]$ is reducible, by Theorem \ref{congruence'}, $f_E$ has Eisenstein descent over $\mathbb{Q}_p$ mod $p$. We now prove our main application to elliptic curves, which is Theorem \ref{Heegnercorollary}.

\begin{proof}[Proof of Theorem \ref{Heegnercorollary}] Recall that $f_E$ determines an invariant differential $\omega_{f_E} = 2\pi i f_E(z)dz\in \Omega_{X_1(N)/\mathbb{Q}}^1$. Let $\Phi_E: X_1(N) \twoheadrightarrow E$ be a modular parametrization (i.e., the Eichler-Shimura abelian variety quotient $\Phi_{f_E}: X_1(N) \twoheadrightarrow A_{f_E}$, postcomposed with an appropriate isogeny) and let $\omega_{E}\in \Omega_{E/\mathbb{Q}}^1$ be an invariant differential chosen so that $\Phi_E^*\omega_{E} = \omega_{f_E}$. 

By Theorem \ref{congruence'} and Theorem \ref{newmaincorollary} with $k=2$, we have
\begin{align*}&\left(\frac{1-a_p+p}{p}\right)^2\log_{\omega_E}^2(P_E(K)) = \left(\frac{1-a_p+p}{p}\right)^2\log_{\omega_{f_E}}^2(P(K))\\
&\equiv\frac{\Xi^2}{4}
\cdot \begin{cases}
\left(\frac{1}{2}(1-\psi^{-1}(p))B_{1,\psi_0^{-1}\varepsilon_K}B_{1,\psi_0\omega^{-1}}\right)^2 & \text{if}\; \psi \neq 1,\\
\left(\frac{p-1}{p}\log_p(\overline{\alpha})\right)^2 & \text{if}\; \psi = 1,\\
\end{cases} \mod p\mathcal{O}_{K_{\mathfrak{p}}}
\end{align*}
where $\psi_0$ is defined as in Definition \ref{evendefinition} and
\begin{align*}\Xi = \prod_{\ell|N_+}\left(1-\psi^{-1}(\ell)\right)\prod_{\ell|N_-}\left(1-\frac{\psi(\ell)}{\ell}\right)\prod_{\ell|N_0}\left(1-\psi^{-1}(\ell)\right)\left(1-\frac{\psi(\ell)}{\ell}\right).
\end{align*}
(Note our congruence holds mod $p\mathcal{O}_{K_{\mathfrak{p}}}$ because both sides of the congruence are defined over $K_{\mathfrak{p}}$ in this situation.)
 
Our hypotheses on $(E,p,K)$ and part (2) of Theorem \ref{classify} (with $k = 2$) ensure that none of the terms on the right hand side of the above congruence vanish mod $p$, and hence $\log_{\omega_{E}}(P_{E}(K)) \neq 0$, i.e. $P_{E}(K)$ is non-torsion.
\end{proof}

\begin{remark}\label{Selmerremark}Suppose $(E,p)$ is as in the statement of Theorem \ref{Heegnercorollary}, and for simplicity suppose $\psi$ is even and non-trivial. Thus, in particular $E[p](\mathbb{Q}) = 0$. We will show (Theorem \ref{twistpositiveproportion}) there always exists an imaginary quadratic $K$ satisfying the appropriate congruence conditions, so that the theorem gives $\text{rank}_{\mathbb{Z}}E(K) = 1$. In particular, this implies that we should be able to see that $\text{rank}_{\mathbb{Z}}E(\mathbb{Q}) \le 1$ a priori from the congruence conditions on $(E,p)$. 

Indeed, one can show that $\text{rank}_{\mathbb{Z}/p}\text{Sel}_p(E/\mathbb{Q}) \le 1$ (which implies that $\text{rank}_{\mathbb{Z}}E(\mathbb{Q}) \le 1$) through purely algebraic methods. Using standard techniques, one can show that the congruence conditions on $(N_+,N_-,N_0)$ imply that $\text{Sel}_p(E/\mathbb{Q}) \subset H^1(\mathbb{Q},E[p];\{p\})$. (Here, for $\Gal(\overline{\mathbb{Q}}/\mathbb{Q})$-module $M$ and a finite set $\Sigma$ of places of $\mathbb{Q}$, $H^1(\mathbb{Q},M;\Sigma)$ denotes the subgroup of $H^1(\mathbb{Q},M)$ consisting of classes unramified outside $\Sigma$.) See, for example, \cite[Sections 2.2]{Li} for the case $p = 3$ and $E$ semistable; the case for general $p > 2$ and general $E$ is completely analogous. The hypothesis $p\nmid B_{1,\psi\omega^{-1}}$ (which is equivalent to $p\nmid L_p(\psi,1)$ since $\psi \neq 1$ and $p\nmid \mathfrak{f}(\psi)$,) implies, via the Main Theorem of Iwasawa theory over $\mathbb{Q}$ and a Selmer group control theorem (cf. \cite[Section 4]{Skinner}), that $H^1(\mathbb{Q},\mathbb{F}_p(\psi);\{p\}) = 0$. Since $H^1(\mathbb{Q},\mathbb{F}_p(\psi);\emptyset) \subset H^1(\mathbb{Q},\mathbb{F}_p(\psi);\{p\})$, we have $H^1(\mathbb{Q},\mathbb{F}_p(\psi);\emptyset) = 0$, which in turn implies that $H^1(\mathbb{Q},\mathbb{F}_p(\psi^{-1}\omega);\{p\}) = \mathbb{Z}/p$ (see, for example, loc. cit. Proposition 2.13). Now if $\mathbb{F}_p(\psi) \subset E[p]$, from the standard long exact sequence of cohomology (cf. Section 2.3 of loc. cit., for example), one obtains a map $\phi: \text{Sel}_p(E/\mathbb{Q}) \rightarrow H^1(\mathbb{Q},\mathbb{F}_p(\psi^{-1}\omega))$ such that $\ker(\phi)\subset H^1(\mathbb{Q},\mathbb{F}_p(\psi);\{p\}) = 0$ and $\text{im}(\phi)\subset H^1(\mathbb{Q},\mathbb{F}_p(\psi^{-1}\omega);\{p\}) = \mathbb{Z}/p$. Thus we get $\text{Sel}_p(E/\mathbb{Q}) \subset \text{im}(\phi) \subset H^1(\mathbb{Q},\mathbb{F}_p(\psi^{-1}\omega);\{p\}) = \mathbb{Z}/p$. If, on the other hand, $\mathbb{F}_p(\psi^{-1}\omega) \subset E[p]$, one obtains a map $\phi: \text{Sel}_p(E/\mathbb{Q}) \rightarrow H^1(\mathbb{Q},\mathbb{F}_p(\psi))$ such that $\ker(\phi)\subset H^1(\mathbb{Q},\mathbb{F}_p(\psi^{-1}\omega);\{p\}) = \mathbb{Z}/p$ and $\text{im}(\phi)\subset H^1(\mathbb{Q},\mathbb{F}_p(\psi);\{p\}) = 0$. Thus we get $\text{Sel}_p(E/\mathbb{Q}) \subset \ker(\phi) \subset H^1(\mathbb{Q},\mathbb{F}_p(\psi^{-1}\omega);\{p\}) = \mathbb{Z}/p$.
\end{remark}

\begin{remark}Since the congruence in the proof of Theorem \ref{Heegnercorollary} comes from $p$-adic interpolation of $p$-adically integral period sums, we should be able to see \emph{a priori} that both sides of the congruence are $p$-adically integral. This is self-evident for the right hand side, and also \emph{a priori} true for the left hand side as follows. Note that since $p\nmid N$, $p+1-a_p = |\tilde{E}(\mathbb{F}_p)|$ by the Eichler-Shimura relation. Let $\hat{E}$ denote the formal group of $E$, so that $\hat{E}(\mathfrak{p}\mathcal{O}_{K_{\mathfrak{p}}})$ has index $|\tilde{E}(\mathbb{F}_p)|$ in $E(K_{\mathfrak{p}})$. (Recall $\mathfrak{p}$ is the previously fixed prime above $p$ determined by our embedding $K\hookrightarrow \overline{\mathbb{Q}}_p$.) Then $[1-a_p+p]P_E(K) =  [|\tilde{E}(\mathbb{F}_p)|]P_E(K) \in \hat{E}(\mathfrak{p}\mathcal{O}_{K_{\mathfrak{p}}})$.

Suppose for the moment that $X_1(N) \twoheadrightarrow E$ is optimal (i.e., $E$ is the strong Weil curve in its isogeny class). Well-known results due to Mazur (\cite{Mazur2}) on the Manin constant $c(E)$ imply that if $\ell|c(E)$, then $\ell^2|4N$. Thus for $p$ odd of good reduction, we have $p\nmid c(E)$. Thus letting $\log_E$ denote the canonical formal logarithm on $E$ (i.e., the formal logarithm arising from the unique normalized invariant differential on $E$), we have $\log_{\omega_E}(T) = \frac{1}{c(E)}\log_E(T)$, meaning our normalization of the formal logarithm does not change the $p$-divisibility on the formal group. That is, $\log_{\omega_E}(\hat{E}(\mathfrak{p}\mathcal{O}_{K_{\mathfrak{p}}})) \subset \mathfrak{p}\mathcal{O}_{K_{\mathfrak{p}}}$. Thus by the previous paragraph, $\frac{1-a_p+p}{p}\log_{\omega_E}P_E(K)$ is $p$-adically integral.

For $E$ non-optimal, the choice of modular parametrization might change the normalization of the formal logarithm, since we are postcomposing the Eichler-Shimura projection with a $\mathbb{Q}$-isogeny which does not necessarily preserve normalizations of the formal logarithm. This can still be shown to not affect $p$-adic integrality of $\log_{\omega_E}$ on $\hat{E}(\mathfrak{p}\mathcal{O}_{K_{\mathfrak{p}}})$. 
\end{remark}

\begin{remark}\label{logdivisibility} Suppose that $(E,p,K)$ is as in the hypotheses of Theorem \ref{Heegnercorollary}. One can show that $\log_{\omega_{E}}(P_E(K)) \equiv 0 \mod p$ as follows. Let $F = K(\mu_p)$, and choose a prime $\pi|p$ of $F$; note that $\mathfrak{p}$ is totally ramified in $K(\mu_p)$, so that $\mathcal{O}_{F_{\pi}}/\pi \cong \mathbb{F}_p$ and $v_{\pi}(p) = p-1$. If $P_K(E)$ is torsion, then $\log_{\omega_{E}}(P_E(K)) = 0$ by properties of the formal logarithm (see \cite[Chapter 4]{Silverman}), so assume that $P_E(K)$ is non-torsion. Suppose $\psi \neq 1$. Then $|\tilde{E}(\mathbb{F}_p)| - (1 + p) = -a_p \equiv -\psi(p) \neq -1 \mod p$ implies $|\tilde{E}(\mathbb{F}_p)| \not\equiv 0 \mod p$. Hence $\hat{E}(\pi\mathcal{O}_{F_{\pi}})$ has index prime to $p$ in $E(F_{\pi})$, and so $\log_{\omega_{E}}(E(F_{\pi})) \subset \pi\mathcal{O}_{F_{\pi}}$, and $\log_{\omega_{E}}(P_E(K)) \in \mathfrak{p}\mathcal{O}_{K_{\mathfrak{p}}}$.  

Now suppose $\psi = 1$, so that $|\tilde{E}(\mathbb{F}_p)| - (1 + p) = -a_p \equiv -\psi(p) = -1 \mod p$, and so $|\tilde{E}(\mathbb{F}_p)| \equiv 0 \mod p$. Moreover $p$ is ordinary good reduction and so $|\tilde{E}(\mathbb{F})[p]| = |\tilde{E}(\overline{\mathbb{F}}_p)| = p$ implies $p| |\tilde{E}(\mathbb{F}_p)|$. Then by Corollary \ref{congruence'}, $E[p] = E(F)[p]$, and so the exact sequence
$$0 \rightarrow \hat{E}(\pi\mathcal{O}_{F_{\pi}}) \rightarrow E(F_{\pi}) \rightarrow \tilde{E}(\mathbb{F}_p) \rightarrow 0$$
splits, i.e. $E(F_{\pi}) = \hat{E}(\pi\mathcal{O}_{F_{\pi}})\oplus\tilde{E}(\mathbb{F}_p)$, and moreover $\tilde{E}(\mathbb{F}_p) \subset E(F_{\pi})^{tor}$. (The torsion of $E(F_{\pi})$ that is outside of $p$ injects into $\tilde{E}(\mathbb{F}_p)$, and $\tilde{E}(\mathbb{F}_p)[p] \subset E(F_{\pi})[p]$).) Thus since $\log_{\omega_{E}}(\hat{E}(\pi\mathcal{O}_{F_{\pi}})) \subset \pi\mathcal{O}_{F_{\pi}}$ and $\log_{\omega_{E}}(\tilde{E}(\mathbb{F}_p) ) = 0$, we have $\log_{\omega_{E}}(P_E(K)) \in \pi\mathcal{O}_{F_{\pi}}$. Now since $P_E(K) \in E(K) \subset E(K_{\mathfrak{p}})$, then $\log_{\omega_{E}}(P_E(K)) \in K_{\mathfrak{p}} \cap \pi\mathcal{O}_{F_{\pi}} = \mathfrak{p}\mathcal{O}_{K_{\mathfrak{p}}}$. 
\end{remark}

\begin{remark}\label{forbidden} While the proof of Theorem \ref{Heegnercorollary} accounts for the case $\psi = 1$, it gives no new information since the left side of the relevant special value congruence is always 0 mod $p$: Remark \ref{logdivisibility} shows that $\log_{\omega_{E}}(P_E(K)) \equiv 0 \mod p$, and $a_p \equiv \psi(p) = 1 \mod p$ implies $\frac{1-a_p+p}{p}$ is a unit in $\mathcal{O}_{\mathbb{C}_p}$. Note this forces $\Xi(1,N_{\splt},N_{\nonsplit},N_{\add}) \equiv 0 \mod p$. 

In fact, if $\psi = 1$ and $p>3$ one can show that for any elliptic curve, $N_{\splt}N_{\add} \neq 1$ in the following way: Assume $N_{\add} = 1$, so that $N$ is squarefree. A theorem of Ribet then shows that $N_{\splt} \neq 1$ (see the Ph.D. thesis of Hwajong Yoo \cite[Theorem 2.3]{Yoo}). Thus if $\psi = 1$ and $p>3$, we have $\Xi(1,N_{\splt},N_{\nonsplit},N_{\add}) = 0$. 

When $\psi \neq 1$, $a_p \equiv \psi(p) \neq 1 \mod p$, and so the factor of $\frac{1-a_p+p}{p} \equiv \frac{\text{unit}}{p} \mod p$, thus cancelling out a $p$-divisibility of $\log_{\omega_{E}}(P_E(K))$ in the special value congruence. 
\end{remark}

\begin{remark}\label{divisibilityremark}Suppose we are in the situation of Theorem \ref{Heegnercorollary}, so in particular $p\nmid N$ is a good prime which is split in $K$. Let $\varpi$ be a local uniformizer at $\mathfrak{p}$. (Recall $\mathfrak{p}$ is the prime above $p$ determining the embedding $K \hookrightarrow \overline{\mathbb{Q}}_p$.) Then as $1 - a_p + p = |\tilde{E}(\mathbb{F}_p)|$, we have that $P := [1 - a_p + p]P_E(K)$ belongs to the formal group $\hat{E}(\mathfrak{p}\mathcal{O}_{K_{\mathfrak{p}}})$. One can show that $\log_E(P) \in \varpi\mathcal{O}_{K_{\mathfrak{p}}}^{\times} = p\mathcal{O}_{K_{\mathfrak{p}}}^{\times}$ if and only if the image of $P$ in $E(K_{\mathfrak{p}})/pE(K_{\mathfrak{p}})$ is not in the image of $E(K_{\mathfrak{p}})[p]$ as follows. If $P$ is not in the image of $E(K_{\mathfrak{p}})[p]$ in $E(K_{\mathfrak{p}})/pE(K_{\mathfrak{p}})$, then suppose $\log_E(P) \not\in \varpi\mathcal{O}_{K_{\mathfrak{p}}}^{\times}$. Since $P \in \hat{E}(\mathfrak{p}\mathcal{O}_{K_{\mathfrak{p}}})$, then $P$ is either torsion (i.e. $\log_E P = 0$) or $P \in \mathfrak{p}^2\mathcal{O}_{K_{\mathfrak{p}}}$ (i.e. $\log_E P \in \mathfrak{p}^2\mathcal{O}_{K_{\mathfrak{p}}}$). However, in the first case this implies $P \in E(K_{\mathfrak{p}})[p]$ (since $\hat{E}(\mathfrak{p}\mathcal{O}_{K_{\mathfrak{p}}})$ has no $p$-torsion) and in the second it implies $P \in \hat{E}(\mathfrak{p}^2\mathcal{O}_{K_{\mathfrak{p}}})$, and thus that $P$ is $p$-divisible. Thus $P \in \hat{E}(\mathfrak{p}\mathcal{O}_{K_{\mathfrak{p}}}) + pE(K_{\mathfrak{p}})$, a contradiction. Conversely, if $\log_E(P) \in \varpi\mathcal{O}_{K_{\mathfrak{p}}}^{\times}$, then $P \not \in E(K_{\mathfrak{p}})[p] + pE(K_{\mathfrak{p}})$; otherwise $P \in E(K_{\mathfrak{p}})[p] + p\hat{E}(\mathfrak{p}\mathcal{O}_{K_{\mathfrak{p}}})$ (since \emph{a priori} $P \in \hat{E}(\mathfrak{p}\mathcal{O}_{K_{\mathfrak{p}}})$) and thus $\log_E(P) \in \mathfrak{p}^2\mathcal{O}_{K_{\mathfrak{p}}}$, a contradiction.

Thus the non-vanishing results of $\log_E(P) \mod p$ provided by Theorem \ref{Heegnercorollary} give information on local $p$-divisibility of $P$ (and thus also of $P_E(K)$). Note that when $p \ge 3$, then $\hat{E}(\mathfrak{p}\mathcal{O}_{K_{\mathfrak{p}}})[p] = 0$, and so $E(K_{\mathfrak{p}})[p] \hookrightarrow \tilde{E}(\mathbb{F}_p)[p]$. In particular, if $p$ is of supersingular reduction, then $ \tilde{E}(\mathbb{F}_p)[p] = 0$, so $E(K_{\mathfrak{p}})[p] = 0$ and the nonvanishing of $\log_E P$ mod p is equivalent to $P$ having non-trivial image in $E(K_{\mathfrak{p}})/pE(K_{\mathfrak{p}})$, i.e. the to the condition that $P$ not be $p$-divisible in $E(K_{\mathfrak{p}})$.
\end{remark}

For any elliptic curve $E$ over $\mathbb{Q}$, let $w(E)$ denote the global root number, which factors as a product of local root numbers
$$w(E) = w_{\infty}(E)\prod_{\ell<\infty}w_{\ell}(E).$$

\begin{proposition}\label{rootnumberproposition}Suppose that $(E,p)$ is as in Theorem \ref{Heegnercorollary} and that $\psi$ is a quadratic character. If either of the following conditions holds:
\begin{enumerate}
\item $p \ge 5$, or
\item $E = E'\otimes \psi$ for some semistable elliptic curve $E'$ of conductor coprime with $\mathfrak{f}(\psi)$,
\end{enumerate}
then $w(E) = -\psi(-1)$. Thus, $\text{rank}_{\mathbb{Z}}E(\mathbb{Q}) = \frac{1+\psi(-1)}{2}$, and $\text{rank}_{\mathbb{Z}}E_{K}(\mathbb{Q}) = \frac{1-\psi(-1)}{2}$ for any imaginary quadratic $K$ as in Theorem \ref{Heegnercorollary}.
\end{proposition}
\begin{proof}By our assumptions in the statement of Theorem \ref{Heegnercorollary}, $\mathfrak{f}(\psi)^2|N_{\add}$ and $N_{\splt} = 1$. By standard properties of the root number (see \cite[Section 3.4]{Dokchitser}), we have
\begin{enumerate}
\item if $\ell||N$ (i.e. $\ell|N_{\nonsplit}$), then $w_{\ell}(E) = 1$,
\item $w_{\infty}(E) = -1$. 
\end{enumerate}
Hence, 
\begin{align*}w(E) &= -\prod_{\ell|N_{\add}}w_{\ell}(E).
\end{align*}

Suppose first that (1) in the statement of the proposition holds, i.e. $p \ge 5$. By Proposition \ref{classify}, $\bar{\rho}_{f_E} \cong \mathbb{F}_p\oplus \mathbb{F}_p(\omega)$. Hence $\bar{\rho}_{f_{E}} \cong \mathbb{F}_p(\psi) \oplus \mathbb{F}_p(\psi^{-1}\omega)$ and thus $E$ admits a degree $p$ isogeny $\phi : E \rightarrow E'$. Since $p \ge 5$, by \cite[Theorem 3.25]{Dokchitser}, for $\ell|N_{\add}$ we have 
$$w_{\ell}(E) = (-1,F/\mathbb{Q}_{\ell})$$
where $F := \mathbb{Q}_{\ell}(\text{coordinates of points in $\ker(\phi)$})$, and $(\cdot,F/\mathbb{Q}_{\ell})$ is the norm residue symbol. By our assumptions, we see that $F = \mathbb{Q}_{\ell}(\mu_{2p},\sqrt{\psi(-1)|\mathfrak{f}(\psi)|})$ which is ramified only if $\ell|p\mathfrak{f}(\psi)$. Thus we see $\mathbb{Q}_{\ell}(\mu_{2p})/\mathbb{Q}_{\ell}$ is an unramified extension. Thus by class field theory, $(-1,F/\mathbb{Q}_{\ell}) = (-1,\mathbb{Q}_{\ell}(\sqrt{\psi(-1)|\mathfrak{f}(\psi)|})/\mathbb{Q}_{\ell}) = \psi_{\ell}(-1)$, and so $w_{\ell}(E) = \psi_{\ell}(-1)$. Hence in all, we have
$$w(E) = -\prod_{\ell|N_{\add}}w_{\ell}(E) = -\prod_{\ell|N_{\add}}\psi_{\ell}(-1) = -\psi(-1)$$
where the last equality follows since $\mathfrak{f}(\psi)|N_{\add}$. 

Now suppose that (2) holds. Let $N'$ denote the conductor of $E'$. By the computations of \cite[Proposition 1]{Balsam} describing the changes of local root numbers under quadratic twists, we have that
\begin{enumerate}
\item if $\ell\nmid N'\mathfrak{f}(\psi)$ then $w_{\ell}(E) = 1$,
\item if $\ell|N'$ (so that $\ell\nmid\mathfrak{f}(\psi)$) then $w_{\ell}(E) = w_{\ell}(E')\psi_{\ell}(\ell) = -a_{\ell}(E')\psi_{\ell}(\ell)$,
\item if $\ell|\mathfrak{f}(\psi)$ (so that $\ell\nmid N'$) then $w_{\ell}(E) = w_{\ell}(E')\psi_{\ell}(-1) = \psi_{\ell}(-1)$.
\end{enumerate}
Hence by our assumptions in the statement of Theorem \ref{Heegnercorollary} and Remark \ref{splitremark}, we have $w(E) = -\psi(-1)\prod_{\ell|N'}(-a_{\ell}(E')\psi_{\ell}(\ell)) = -\psi(-1)$. Putting this together, we compute $w(E) = -\psi(-1)$.

Now by Theorem \ref{Heegnercorollary} and parity considerations, we immediately get that $\text{rank}_{\mathbb{Z}}E(\mathbb{Q}) = \frac{1 + \psi(-1)}{2}$ and $\text{rank}_{\mathbb{Z}}E_{K}(\mathbb{Q}) = \frac{1 - \psi(-1)}{2}$ for any $K$ as in Theorem \ref{Heegnercorollary}.
\end{proof}

\subsection{Calculating ranks in positive density subfamilies of quadratic twists}\label{positiveproportionsection}
In the case $p = 3$, the Teichm\"{u}ller character $\omega$ is quadratic and is in fact the character associated with the imaginary quadratic field $\mathbb{Q}(\sqrt{-3})$. Thus we have (using $B_{1,\psi\omega^{-1}} = -L(\psi\omega^{-1},0)$, the functional equation, and the class number formula),
$$B_{1,\psi\omega^{-1}} = 2\frac{h_{K_{\psi}\cdot\mathbb{Q}(\sqrt{-3})}}{|\mathcal{O}_{K_{\psi}\cdot\mathbb{Q}(\sqrt{-3})}^{\times}|}.$$
Hence our nondivisibility criterion involving Bernoulli numbers in Theorem \ref{Heegnercorollary} reduces to the non 3-divisibility of the class numbers of a pair of imaginary quadratic fields, and is thus amenable to class number 3-divisibility results in the tradition of Davenport-Heilbronn \cite{DavenportHeilbronn}. 

For the remainder of this section, suppose $p = 3$ and we are given $E/\mathbb{Q}$ of conductor $N = N_{\splt}N_{\nonsplit}N_{\add}$ such that $E[3]$ is reducible of type $(1,1,N_{\splt},N_{\nonsplit},N_{\add})$. Let $L = K_{\psi}$, and furthermore that $L$ satisfies the congruence conditions
\begin{enumerate}
\item 3 is inert in $L$,
\item $3\nmid h_{L\cdot\mathbb{Q}(\sqrt{-3})}$,
\item if $\ell|N_{\splt}$, $\ell$ is inert or ramified in $L$,
\item if $\ell|N_{\nonsplit}$, $\ell$ is either split or ramified in $L$,
\item if $\ell|N_{\add}$, $\ell$ is either inert in $L$ and $\ell\not\equiv 2\mod 3$, or $\ell$ is ramified in $L$.
\end{enumerate}
Then we can apply Theorem \ref{Heegnercorollary} and Proposition \ref{rootnumberproposition} to the curve $E\otimes \psi$ (using the fact that $a_{\ell}(E\otimes \psi) = \psi(\ell)a_{\ell}(E)$, so that by conditions (2) and (3), $a_{\ell}(E\otimes \psi) = -1$ for $\ell|N_{\splt}N_{\nonsplit}$), thus obtaining $\text{rank}_{\mathbb{Z}}(E\otimes \psi)(\mathbb{Q}) = \frac{1+\psi(-1)}{2}$. Using results of Horie-Nakagawa \cite{HorieNakagawa} and Taya \cite{Taya} regarding $3$-divisibilities of class numbers of quadratric fields, we can produce a positive proportion of real quadratic $L$ as above, thus showing a positive proportion of quadratic twists of $E$ have rank $\frac{1+\psi(-1)}{2}$ over $\mathbb{Q}$. 

Using these same class number divisibility results, we can also produce a positive proportion of imaginary quadratic $K$ satisfying
\begin{enumerate}
\item 3 is split in $K$,
\item $3\nmid h_{L\cdot K}$,
\item $K$ satisfies the Heegner hypothesis with respect to $E$,
\end{enumerate}
and thus, using Theorem \ref{Heegnercorollary} and Proposition \ref{rootnumberproposition}, show a positive proportion of  imaginary quadratic twists $E_{K}$ of $E$ have $\text{rank}_{\mathbb{Z}}E_{K}(\mathbb{Q}) = \frac{1-\psi(-1)}{2}$. 

To this end, let us recall the result of Horie and Nakagawa. For any $x \ge 0$, let $K^+(x)$ denote the set of real quadratic fields with fundamental discriminant $D_K<x$ and $K^-(x)$ the set of imaginary quadratic fields with fundamental discriminant $|D_K| < x$. Put
\begin{align*}&K^+(x,m,M) := \{k \in K^+(x) : D_K \equiv m \mod M\}\\
&K^-(x,m,M) := \{k \in K^-(x) : D_K \equiv m \mod M\}.
\end{align*}
Moreover, for a quadratic field $k$, we denote by $h_k[3]$ the number of ideal class of $k$ whose cubes are principal (i.e., the order of 3-torsion of the ideal class group). Horie and Nakagawa prove the following.

\begin{theorem}[\cite{HorieNakagawa}]\label{HorieNakagawa} Suppose that $m$ and $M$ are positive integers such that if $\ell$ is an odd prime number dividing $(m,M)$, then $\ell^2$ divides $M$ but not $m$. Further, if $M$ is even, suppose that 
\begin{enumerate}
\item $4|M$ and $m \equiv 1 \mod 4$, or
\item $16|M$ and $m \equiv 8 \; \text{or}\; 12 \mod 16$.
\end{enumerate}
Then 
\begin{align*}&\sum_{k \in K^+(x,m,M)}h_k[3] \sim \frac{4}{3}|K^+(x,m,M)| \hspace{.5cm} (x \rightarrow \infty)\\
&\sum_{k \in K^-(x,m,M)}h_k[3] \sim 2|K^-(x,m,M)| \hspace{.5cm} (x \rightarrow \infty).
\end{align*}
Furthermore,
$$|K^+(x,m,M)| \sim |K^-(x,m,M)| \sim \frac{3x}{\pi^2\Phi(M)}\prod_{\ell|M}\frac{q}{\ell+1} \hspace{.5cm} (x \rightarrow \infty).$$
Here $f(x) \sim g(x) \hspace{.5cm} (x\rightarrow \infty)$ means that $\lim_{x\rightarrow \infty} \frac{f(x)}{g(x)} = 1$, $\ell$ ranges over primes dividing $M$, $q = 4$ if $\ell = 2$, and $q = \ell$ otherwise.
\end{theorem}

Now put
\begin{align*}&K_*^+(x,m,M) := \{k \in K^+(x,m,M) : h_k[3] = 1\}\\
&K_*^-(x,m,M) := \{k \in K^-(x,m,M) : h_k[3] = 1\}.
\end{align*}

Taya's argument \cite{Taya} for estimating $|K_*^{\pm}(x,m,M)|$ goes as follows. Since $h_k[3] \ge 3$ if $h_k[3] \neq 1$, we have the bound
$$|K_*^{\pm}(x,m,M)| + 3(|K^{\pm}(x,m,M)| - |K_*^{\pm}(x,m,M)|) \le \sum_{k \in K^{\pm}(x,m,M)}h_k[3].$$
Hence,
$$|K_*^{\pm}(x,m,M)| \ge \frac{3}{2}|K^{\pm}(x,m,M)| - \frac{1}{2}\sum_{k\in K^{\pm}(x,m,M)}h_k[3].$$
Now by Theorem \ref{HorieNakagawa}, we have
\begin{align*}&\frac{3}{2}|K^+(x,m,M)| - \frac{1}{2}\sum_{k\in K^+(x,m,M)}h_k[3] \sim \frac{5}{6}|K^+(x,m,M)| \hspace{.5cm} (x \rightarrow \infty)\\
&\frac{3}{2}|K^-(x,m,M)| - \frac{1}{2}\sum_{k\in K^-(x,m,M)}h_k[3] \sim \frac{1}{2}|K^-(x,m,M)| \hspace{.5cm} (x \rightarrow \infty)
\end{align*}
and hence,
\begin{align*}&\lim_{x\rightarrow\infty}\frac{|K_*^+(x,m,M)|}{x} \ge \frac{5}{2\pi^2\Phi(M)}\prod_{\ell|M}\frac{q}{\ell+1}\\
&\lim_{x\rightarrow\infty}\frac{|K_*^-(x,m,M)|}{x} \ge \frac{3}{2\pi^2\Phi(M)}\prod_{\ell|M}\frac{q}{\ell+1}.
\end{align*}
Thus, we have
\begin{proposition}\label{positiveproportion}Suppose $m$ and $M$ satisfy the conditions of Theorem \ref{HorieNakagawa}. Then
\begin{align*}&\lim_{x\rightarrow \infty}\frac{|K_*^+(x,m,M)|}{|K^+(x,1,1)|} \ge \frac{5}{6\Phi(M)}\prod_{\ell|M}\frac{q}{\ell+1}\\
&\lim_{x\rightarrow \infty}\frac{|K_*^-(x,m,M)|}{|K^-(x,1,1)|} \ge \frac{1}{2\Phi(M)}\prod_{\ell|M}\frac{q}{\ell+1}
\end{align*}
where $q = 4$ if $\ell = 2$ and $q = \ell$ otherwise. In particular, the of real (resp. imaginary) quadratic fields $k$ such that $D_k \equiv m \mod M$ and $3\nmid h_k$ has positive density in the set of all real (resp. imaginary) quadratic fields. 
\end{proposition}
\begin{proof}This follows from the above asymptotic estimates and the fact that $|K^{\pm}(x,1,1)| \sim \frac{3x}{\pi^2}$ by Theorem \ref{HorieNakagawa}.
\end{proof}
We are now ready to prove our positive density results. For a quadratic field $L$, let $E_L$ denote the quadratic twist of $E$ by $L$. 

\begin{theorem}\label{realtwistpositiveproportion} Suppose $(N_{\splt},N_{\nonsplit},N_{\add})$ is a triple of pairwise coprime integers such that $N_{\splt}N_{\nonsplit}$ is squarefree, $N_{\add}$ is squarefull and $N_{\splt}N_{\nonsplit}N_{\add} = N$. If $2\nmid N$ let $M' = N$, if $2||N$ let $M' = lcm(N,8)$, and if $4|N$ let $M' = lcm(N,16)$. Then a proportion of at least
\begin{align*}&\frac{1}{12\Phi(M')}\prod_{\ell|N_{\splt}N_{\nonsplit},\ell \; \text{odd}}\frac{\ell-1}{2}\\
&\cdot\prod_{\ell|N_{\add},\ell \; \text{odd}, \ell \equiv 1 \mod 3}\frac{(\ell+2)(\ell-1)}{2}\prod_{\ell|N_{\add},\ell \; \text{odd}, \ell \equiv 2 \mod 3}(\ell-1)\prod_{4|N_{\add}}2\prod_{\ell|3M'}\frac{q}{\ell+1}
\end{align*}
of real quadratic extensions $L/\mathbb{Q}$ satisfy
\begin{enumerate}
\item 3 is inert in $L$;
\item $3\nmid h_{L\cdot\mathbb{Q}(\sqrt{-3})}$;
\item $\ell|N_{\splt}$ implies $\ell$ is inert in $L$;
\item $\ell|N_{\nonsplit}$ implies $\ell$ is split in $L$;
\item $\ell|N_{\add}$ implies $\ell$ is inert in $L$ and $\ell \not\equiv 2 \mod 3$, or $\ell$ is ramified in $L$;
\item $4|N$ implies $D_L \equiv 8$ or $12 \mod 16$.
\end{enumerate}
Moreover, a proportion of at least
\begin{align*}&\frac{1}{4\Phi(M')}\prod_{\ell|N_{\splt}N_{\nonsplit},\ell \; \text{odd}}\frac{\ell-1}{2}\\
&\cdot\prod_{\ell|N_{\add},\ell \; \text{odd}, \ell \equiv 1 \mod 3}\frac{(\ell+2)(\ell-1)}{2}\prod_{\ell|N_{\add},\ell \; \text{odd}, \ell \equiv 2 \mod 3}(\ell-1)\prod_{4|N_{\add}}2\prod_{\ell|3M'}\frac{q}{\ell+1}
\end{align*}
of imaginary quadratic extensions $L/\mathbb{Q}$ satisfy
\begin{enumerate}
\item 3 is inert in $L$;
\item $3\nmid h_{L}$;
\item $\ell|N_{\splt}$ implies $\ell$ is inert in $L$;
\item $\ell|N_{\nonsplit}$ implies $\ell$ is split  in $L$;
\item $\ell|N_{\add}$ implies $\ell$ is inert in $L$ and $\ell \not\equiv 2 \mod 3$, or $\ell$ is ramified in $L$;
\item $4|N$ implies $D_L \equiv 8$ or $12 \mod 16$.
\end{enumerate}
(Here again, $q = 4$ for $\ell = 2$, and $q = \ell$ for odd primes $\ell$.)
\end{theorem}
\begin{proof}We seek to apply Proposition \ref{positiveproportion}. Let $M = 9N$ if $2\nmid N$, $M = 9lcm(N,8)$ if $2||N$, and $M = 9lcm(N,16)$ if $4|N$. Using the Chinese remainder theorem, choose a positive integer $m$ such that
\begin{enumerate}
\item $m \equiv 3 \mod 9$,
\item $\ell$ odd prime, $\ell|N_{\splt} \implies m \equiv -3[\text{quadratic nonresidue unit}] \mod \ell$,
\item $2|N_{\splt} \implies m \equiv 1 \mod 8$,
\item $\ell$ odd prime, $\ell|N_{\nonsplit} \implies m \equiv -3[\text{quadratic residue unit}] \mod \ell$,
\item $2|N_{\nonsplit} \implies m \equiv 5 \mod 8$,
\item $\ell$ odd prime, $\ell|N_{\add}, \ell\equiv 1 \mod 3 \implies m \equiv -3[\text{quadratic nonresidue unit}] \mod \ell$ or $m \equiv 0 \mod \ell$ and $m \not\equiv 0 \mod \ell^2$,
\item $\ell$ odd prime, $\ell|N_{\add}, \ell\equiv 2 \mod 3 \implies m \equiv 0 \mod \ell$ and $m\not\equiv 0 \mod \ell^2$,
\item $4|N_{\add} \implies m \equiv 8$ or $12 \mod 16$. 
\end{enumerate}
Suppose $L$ is any real quadratic field with fundamental discriminant $D_L$ and $-3D_L \equiv m \mod M$. Then the above congruence conditions on $m$ along with our assumptions imply
\begin{enumerate}
\item 3 is inert in $L$;
\item $\ell$ prime, $\ell|N_{\splt} \implies \ell$ is inert in $L$;
\item $\ell$ prime, $\ell|N_{\nonsplit} \implies \ell$ is split in $L$;
\item $\ell$ odd prime, $\ell|N_{\add}, \ell\equiv 1 \mod 3 \implies \ell$ is inert or ramified in $L$;
\item $\ell$ odd prime, $\ell|N_{\add}, \ell\equiv 2 \mod 3 \implies \ell$ is ramified in $L$;
\item $4|N_{\add} \implies 2$ is ramified in $L$;
\item if $2||N$, then $4|M$ and $m \equiv 1 \mod 4$;
\item if $4|N$, then $16|M$ and $m \equiv 8$ or $12 \mod 16$.
\end{enumerate}
Thus for real quadratic $L$ such that $D_{L\cdot\mathbb{Q}(\sqrt{-3})} = -3D_L \equiv m \mod M$, $L$ satisfies all the desired congruence conditions except for possibly $3\nmid h_{L\cdot\mathbb{Q}(\sqrt{-3})}$. Moreover, the congruence conditions above imply that $m$ and $M$ are valid positive integers for Theorem \ref{HorieNakagawa} (in particular implying that $4|D_L$ if $4|N$). (Note that in congruence conditions (2) and (3) above, we do not allow $m \equiv 0 \mod \ell$, i.e. $\ell$ ramified in $L$,  because the resulting pair $m$ and $M$ would violate the auxiliary hypothesis of Theorem \ref{HorieNakagawa}.) Thus, by Proposition \ref{positiveproportion}, 
$$\lim_{x\rightarrow \infty}\frac{|K_*^-(x,m,M)|}{|K^-(x,1,1)|}\ge \frac{1}{2\Phi(M)}\prod_{\ell|M}\frac{q}{\ell+1}$$
and so a positive proportion of real quadratic $L$ satisfy $D_{L\cdot\mathbb{Q}(\sqrt{-3})} = -3D_L \equiv m \mod M$ and $3\nmid h_{L\cdot\mathbb{Q}(\sqrt{-3})}$. M'oreover, noticing that the congruence conditions (1)-(6) on $m$ above are independent (again by the Chinese remainder theorem), we have 
\begin{align*}&\prod_{\ell|N_{\splt},\ell \; \text{odd}}\frac{\ell-1}{2}\prod_{\ell|N_{\nonsplit},\ell \; \text{odd}}\frac{\ell-1}{2}\\
&\cdot\prod_{\ell|N_{\add},\ell \; \text{odd}, \ell \equiv 1 \mod 3}\frac{(\ell+2)(\ell-1)}{2}\prod_{\ell|N_{\add},\ell \; \text{odd}, \ell \equiv 2 \mod 3}(\ell-1)\prod_{4|N_{\add}}2
\end{align*}
valid choices of residue classes for $m$ mod $M$. Combining all the above and summing over each valid residue class $m$ mod $M$, we immediately obtain our lower bound for the proportion of valid $L$ (with $M = 9M'$). 

For the second case (concerning imaginary quadratic fields), the asserted statement follows from taking $M$ as above, then choosing a positive integer $m$ such that 
\begin{enumerate}
\item $m \equiv -1 \mod 3$,
\item $\ell$ odd prime, $\ell|N_{\splt} \implies m \equiv [\text{quadratic nonresidue unit}] \mod \ell$,
\item $2|N_{\splt} \implies m \equiv 5 \mod 8$,
\item $\ell$ odd prime, $\ell|N_{\nonsplit} \implies m \equiv [\text{quadratic residue unit}] \mod \ell$,
\item $2|N_{\nonsplit} \implies m \equiv 1 \mod 8$,
\item $\ell$ odd prime, $\ell|N_{\add}, \ell\equiv 1 \mod 3 \implies m \equiv [\text{quadratic nonresidue unit}] \mod \ell$ or $m \equiv 0 \mod \ell$ and $m \not\equiv 0 \mod \ell^2$,
\item $\ell$ odd prime, $\ell|N_{\add}, \ell\equiv 2 \mod 3 \implies m \equiv 0 \mod \ell$ and $m\not\equiv 0 \mod \ell^2$,
\item $4|N_{\add} \implies m \equiv 8$ or $12 \mod 16$
\end{enumerate}
and proceeding by the same argument as above. 
\end{proof}

\begin{theorem}\label{twistpositiveproportion} Suppose $E/\mathbb{Q}$ is any elliptic curve whose mod 3 Galois representation $E[3]$ is reducble of type $(1,1,N_{\splt},N_{\nonsplit},N_{\add})$, where $3$ is a good prime of $E$. Let $L$ be any quadratic field such that
\begin{enumerate}
\item 3 is inert in $L$;
\item $3 \nmid h_{L\cdot\mathbb{Q}(\sqrt{-3})}$ if $L$ is real, and $3\nmid h_L$ if $L$ is imaginary;
\item $\ell|N_{\splt}$ implies $\ell$ is inert in $L$;
\item $\ell|N_{\nonsplit}$ implies $\ell$ is split in $L$;
\item $\ell|N_{\add}$ implies $\ell$ is inert in $L$ and $\ell \not\equiv 2 \mod 3$, or $\ell$ is ramified in $L$;
\item $4|N$ implies $D_L \equiv 8$ or $12 \mod 16$.
\end{enumerate}
Let $M' = lcm(N,D_L^2)$ if $lcm(N,D_L^2)$ is odd, $M' = lcm(N,D_L^2,8)$ if $2||lcm(N,D_L^2)$, and $M' = lcm(N,D_L^2,16)$ if $4|lcm(N,D_L^2)$. Then if $L$ is real, for a positive proportion of at least
$$\frac{1}{4\Phi(M')}\prod_{\ell|N_{\splt}N_{\nonsplit}, \ell\nmid D_L, \ell \;\text{odd}}\frac{\ell-1}{2}\prod_{\ell|N_{\add},\ell\nmid D_L, \ell \; \text{odd}}\frac{\ell(\ell-1)}{2}\prod_{\ell|D_L,\ell\;\text{odd}}\frac{\ell-1}{2}\prod_{\ell|3M'}\frac{q}{\ell+1}$$
of imaginary quadratic fields $K$, and if $L$ is imaginary, for a positive proportion of at least
$$\frac{1}{12\Phi(M')}\prod_{\ell|N_{\splt}N_{\nonsplit}, \ell\nmid D_L, \ell \;\text{odd}}\frac{\ell-1}{2}\prod_{\ell|N_{\add},\ell\nmid D_L, \ell \; \text{odd}}\frac{\ell(\ell-1)}{2}\prod_{\ell|D_L,\ell\;\text{odd}}\frac{\ell-1}{2}\prod_{\ell|3M'}\frac{q}{\ell+1}$$
of imaginary quadratic fields $K$, $K$ satisfies the Heegner hypothesis with respect to $E_L$, $(D_K,D_L) = 1$, and the Heegner point $P_{E_L}(K)$ is non-torsion. (Here again, $q = 4$ for $\ell = 2$, and $q = \ell$ for odd primes $\ell$.)
\end{theorem}
\begin{proof}Again we seek to apply Proposition \ref{positiveproportion}, as well as Theorem \ref{Heegnercorollary}. First suppose that $L$ is a real quadratic field. Let $M = 3lcm(N,D_L^2)$ if $lcm(N,D_L^2)$ is odd, $M = 3lcm(N,D_L^2,8)$ if $2||lcm(N,D_L^2)$, and $M = 3lcm(N,D_L^2,16)$ otherwise. Using the Chinese remainder theorem, choose a positive integer $m$ such that
\begin{enumerate}
\item $m \equiv 2 \mod 3$,
\item $\ell$ odd prime, $\ell|N_{\splt} \implies m \equiv [\text{quadratic nonresidue unit}] \mod \ell$,
\item $2|N_{\splt}, \implies m \equiv 5 \mod 8$,
\item $\ell$ prime, $\ell|N_{\nonsplit} \implies m \equiv [\text{quadratic residue unit}] \mod \ell$,
\item $2|N_{\nonsplit}, \implies m \equiv 1 \mod 8$,
\item $\ell$ odd prime, $\ell|N_{\add}, \ell\nmid D_L \implies m \equiv [\text{quadratic nonresidue unit}] \mod \ell$,
\item $\ell$ odd prime, $\ell|N_{\add}, \ell|D_L \implies m \equiv 0 \mod \ell$ where $\frac{m}{D_L} \equiv [\text{quadratic residue unit}] \mod \ell$,
\item $4|N \implies m \equiv D_L \mod 16.$
\end{enumerate}

Suppose $K$ is any imaginary quadratic field with \emph{odd} fundamental discriminant $D_K$ such that $(D_L,D_K) = 1$ and $D_LD_K \equiv m \mod M$. Since $D_K$ is odd, we must have $D_K \equiv 1 \mod 4$, and this is compatible with condition (6) which forces $D_K \equiv 1 \mod 8$, which in turn forces 2 to split in $K$. Then the above congruence conditions on $m$, along with the congruence conditions of our assumptions, imply
\begin{enumerate}
\item 3 is inert in $L$, split in $K$, and inert in $L\cdot K$;
\item $\ell$ prime, $\ell|N_{\splt}, \ell\nmid D_L \implies \ell$ is inert in $L$, split in $K$, and inert in $L\cdot K$;
\item $\ell$ prime, $\ell|N_{\nonsplit}, \ell\nmid D_L \implies \ell$ is split in $L$, split in $K$, and split in $L\cdot K$;
\item $\ell$ odd prime, $\ell|N_{\add}, \ell\nmid D_L \implies \ell$ is inert in $L$, split in $K$, and inert in $L\cdot K$;
\item $\ell$ odd prime, $\ell|D_L \implies \ell$ is ramified in $L$, split in $K$, and ramified in $L\cdot K$;
\item $4|N_{\add} \implies 2$ is ramified in $L$, split in $K$, and ramified in $L\cdot K$;
\item if $2||N$, then $4|M$ and $m \equiv \mod 4$;
\item if $4|N$, then $16|M$ and $m \equiv 8$ or $12 \mod 16$.
\end{enumerate}
Thus for imaginary quadratic $K$ such that $D_{L\cdot K} = D_LD_K \equiv m \mod M$, $(E,3,L,K)$ satisfies all the congruence conditions of Theorem \ref{Heegnercorollary} except for possibly $3\nmid h_{L\cdot K}$. M'oreover, the congruence conditions above imply that $m$ and $M$ are valid positive integers for Theorem \ref{HorieNakagawa}. Thus, by Proposition \ref{positiveproportion},
$$\lim_{x\rightarrow \infty}\frac{|K_*^-(x,m,M)|}{|K^-(x,1,1)|} \ge \frac{1}{2\Phi(M)}\prod_{\ell|M}\frac{q}{\ell+1}$$
and so a positive proportion of imaginary quadratic $K$ satisfy $D_{L\cdot K} \equiv m \mod M$ and $3\nmid h_{L\cdot K}$. Thus, for these $K$, $(E,3,L,K)$ satisfies all the congruence conditions of Theorem \ref{Heegnercorollary}, and so $P_{E_L}(K)$ is non-torsion. Moreover, noticing that the congruence conditions (1)-(6) on $m$ above are independent (again by the Chinese remainder theorem), we have
\begin{align*}\prod_{\ell|N_{\splt}N_{\nonsplit}, \ell\nmid D_L, \ell \;\text{odd}}\frac{\ell-1}{2}\prod_{\ell|N_{\add},\ell\nmid D_L, \ell \; \text{odd}}\frac{\ell(\ell-1)}{2}\prod_{\ell|D_L,\ell\;\text{odd}}\frac{\ell-1}{2}
\end{align*}
choices for residue classes of $m$ mod $M$. Combining all the above and summing over each valid residue class $m$ mod $M$, we immediately obtain our lower bound for the proportion of valid $K$ (with $M = 3M'$). 

For the case when $L$ is an imaginary quadratic field, let $M$ be as above. Then choose a positive integer $m$ such that
\begin{enumerate}
\item $m \equiv 3 \mod 9$,
\item $\ell$ odd prime, $\ell|N_{\splt} \implies m \equiv -3[\text{quadratic nonresidue unit}] \mod \ell$,
\item $2|N_{\splt} \implies m \equiv 1 \mod 8$,
\item $\ell$ odd prime, $\ell|N_{\nonsplit} \implies m \equiv -3[\text{quadratic residue unit}] \mod \ell$,
\item $2|N_{\nonsplit} \implies m \equiv 5 \mod 8$,
\item $\ell$ odd prime, $\ell|N_{\add}, \ell\nmid D_L \implies m \equiv -3[\text{quadratic nonresidue unit}] \mod \ell$,
\item $\ell$ odd prime, $\ell|D_L \implies m \equiv 0 \mod \ell$ where $\frac{m}{D_L} \equiv -3[\text{quadratic residue unit}] \mod \ell$,
\item $4|N \implies m \equiv D_L \mod 16$.
\end{enumerate}
The argument then proceeds in the exact same way as above to establish $3\nmid h_{L\cdot K.\mathbb{Q}(\sqrt{-3})}$ and thus $\text{rank}_{\mathbb{Z}}E_L(K) = 1$ by applying Theorem \ref{Heegnercorollary}.
\end{proof}

\begin{proof}[Proof of Corollary \ref{positiveproportioncorollary}] Since $E[3]$ is a reducible mod $3$ Galois representation, $E$ has Eisenstein descent of type $(\psi,\psi^{-1},N_{\splt},N_{\nonsplit},N_{\add})$ mod 3 where $\psi$ is some quadratic Dirichlet character. We may assume without loss of generality that $\psi = 1$ (after possibly replacing $E$ by $E\otimes \psi^{-1}$). From Theorem \ref{realtwistpositiveproportion}, a positive proportion of quadratic twists $E_L$ satisfy the conditions of Theorem \ref{twistpositiveproportion}, and so by that theorem a positive proportion of imaginary quadratic $K$ have that $P_{E_L}(K)$ is non-torsion. If $E$ is semistable, then $E$ is necessarily has Eisenstein descent of type $(1,1,N_{\splt},N_{\nonsplit},N_{\add})$ by part $(3)$ of Theorem \ref{classify} and Theorem \ref{congruence'}. Since $N_{\add} = 1$, Theorems \ref{realtwistpositiveproportion} and \ref{twistpositiveproportion} produce twists $E_L$ with $(N,D_L) = 1$  and $\text{rank}_{\mathbb{Z}}E(K) = 1$. Then by part (2) of Proposition \ref{rootnumberproposition}, each $E_L$ has $w(E_L) = -\varepsilon_L(-1)$ and so $\text{rank}_{\mathbb{Z}}E_L(\mathbb{Q}) = \frac{1+\varepsilon_L(-1)}{2}$ and $\text{rank}_{\mathbb{Z}}E_{L\cdot K}(\mathbb{Q}) = \frac{1-\varepsilon_L(-1)}{2}$. The more precise lower bounds on these positive proportions follow immediately from Theorem \ref{realtwistpositiveproportion} and Theorem \ref{twistpositiveproportion}. 
\end{proof}

\begin{remark}There is no ``double counting" resulting from using the lower bounds of Theorems \ref{realtwistpositiveproportion} and \ref{twistpositiveproportion} in tandem. The real quadratic twists $E_L$ produced in Theorem 49, which have discriminant $D_L$ prime to $D_K$, are distinct from the real twists $E_{L'\cdot K}$ produced in Theorem 50 (with $L'$ imaginary), which have discriminant $D_{L'}D_K$. Similarly, the imaginary quadratic twists produced in Theorem 49 are distrinct from those produced in Theorem 50.
\end{remark}

\begin{example}Let $E/\mathbb{Q}$ be the elliptic curve 19a1 in Cremona's labeling, which has minimal Weierstrass model
$$y^2 + y = x^3 + x^2 - 9x - 15.$$
Then $E(\mathbb{Q})^{tor} = \mathbb{Z}/3$, and so taking $p = 3$, $E[3]$ is a reducible mod $3$ Galois representation. Furthermore, $E$ has conductor $N = 19$, where 19 is of split multiplicative reduction. Taking the real quadratic field $L = \mathbb{Q}(\sqrt{41})$, one can check that 3 and 19 are inert in $L$. Taking the imaginary quadratic field $K = \mathbb{Q}(\sqrt{-2})$, one sees that 3 splits in $K$ and that $K$ satisfies the Heegner hypothesis with respect to the quadratic twist $E_L$ (and 3 and 19 split in $K$). Furthermore, 3 does not divide the class numbers $h_{L\cdot\mathbb{Q}(\sqrt{-3})} = h_{\mathbb{Q}(\sqrt{-123})} = 4$ and $h_{L\cdot K} = h_{\mathbb{Q}(\sqrt{-82})} = 2$. Our result now gives $\text{rank}_{\mathbb{Z}}E_L(K) = 1$. By Proposition \ref{rootnumberproposition}, one sees that $\text{rank}_{\mathbb{Z}}E_L(\mathbb{Q}) = 1$ and $\text{rank}_{\mathbb{Z}}E_{L\cdot K}(\mathbb{Q}) = 0$. Taking the imaginary quadratic field $L' = \mathbb{Q}(\sqrt{-7})$, one can check that 3 and 19 are inert in $L'$. Furthermore, 3 does not divide the class numbers $h_{L'} = 1$ and $h_{L'\cdot K\cdot\mathbb{Q}(\sqrt{-3})} = h_{\mathbb{Q}(\sqrt{-42})} = 4$, so by Proposition \ref{rootnumberproposition}, one sees that $\text{rank}_{\mathbb{Z}}E_{L'}(\mathbb{Q}) = 0$ and $\text{rank}_{\mathbb{Z}}E_{L'\cdot K}(\mathbb{Q}) = 1$. By Corollary \ref{positiveproportioncorollary} (and adding the explicit lower bounds given in Theorem \ref{realtwistpositiveproportion} and Theorem \ref{twistpositiveproportion} applied to $E, E_L$ and $E_{L'}$), at least $\frac{19}{640} + \frac{19}{10240} = \frac{323}{10240}$ of real quadratic twists of $E$ have rank 1, and at least $\frac{57}{640}+\frac{19}{17920} = \frac{323}{3584}$ of imaginary quadratic twists of $E$ have rank 0.
\end{example}

\end{document}